\documentclass[]{article}

\usepackage{bussproofs}
\usepackage[british]{babel}

\usepackage[geometry]{ifsym}
\usepackage{enumerate}
\usepackage{latexsym}
\usepackage{tikz}
\usepackage{newlfont}
\usepackage{latexsym}
\usepackage{multirow}
\usepackage{array,cmll}
\usepackage{mathdots}
\usepackage{textcomp}
\usepackage{rotating}
\usepackage{logreq}
\usepackage{hyperref}
\usepackage{color}

  \usepackage{amscd}
  \usepackage{amssymb}
  \usepackage{bussproofs}

\usetikzlibrary{arrows}

\usepackage[T1]{fontenc}
\usepackage[utf8]{inputenc}

\usepackage{amsmath}
\usepackage{amsfonts}
\usepackage{amssymb}
\usepackage{amsthm}
\usepackage{graphicx}
\usepackage{amsopn}
\usepackage[all]{xy}
\usepackage{ wasysym }

\theoremstyle{definition}
\theoremstyle{theorem}

\newtheorem{theorem}{Theorem}

\newtheorem{proposition}[theorem]{Proposition}
\newtheorem{fact}[theorem]{Fact}
\newtheorem{lemma}[theorem]{Lemma}
\theoremstyle{definition}
\newtheorem{defn}{Definition}
\newtheorem{rmrk}{Remark}

\def\aga{\texttt{a}}

\newcommand{\ssRESalphaProxy}{\,\rotatebox[origin=c]{90}{$\scriptstyle\{\rotatebox[origin=c]{-90}{$\scriptstyle \alpha$}\}$}\,}

\newcommand{\rab}{\RESalphaBox}

\newcommand{\bbZ}{\mathbb{Z}}
\newcommand{\lsem}{\mathopen{[\![}}
\newcommand{\rsem}{\mathclose{]\!]}}
\newcommand{\sem}[1]{\lsem #1 \rsem}

\newcommand{\Mbb}{{\mathbb{ M}}}
\newcommand{\Zbb}{{\mathbb{ Z}}}

\newcommand{\fBox}[1]{[#1]}
\newcommand{\pBox}[1]{[#1{^{\circ}}]}
\newcommand{\fDiam}[1]{\langle#1\rangle}
\newcommand{\pDiam}[1]{\langle#1{^{\circ}}\rangle}

\newcommand{\conv}[1]{#1{^{\!\circ}}}

\newcommand{\Bigsemic}{\mbox{\Large {\bf ;}}}

\newcommand{\Bdia}{\blacklozenge}
\newcommand{\RESagaProxy}{\,\rotatebox[origin=c]{90}{$\{\rotatebox[origin=c]{-90}{$\aga$}\}$}\,}

\newcommand{\fns}{\footnotesize}
\newcommand{\mc}{\multicolumn}

\newcommand{\RESalphaProxy}{\,\rotatebox[origin=c]{90}{$\{\rotatebox[origin=c]{-90}{$\alpha$}\}$}\,}

\newcommand{\RESalphajProxy}{\,\rotatebox[origin=c]{90}{$\{\rotatebox[origin=c]{-90}{$\!\alpha_{\!j}$}\,\}$}\,}
\newcommand{\RESalphaBox}{\,\rotatebox[origin=c]{90}{$[\mkern1.8mu\rotatebox[origin=c]{-90}{$\alpha$}\mkern1.8mu]$}\,}

\newcommand{\RESalphaDia}{\,\rotatebox[origin=c]{90}{$\langle\rotatebox[origin=c]{-90}{$\alpha$}\rangle$}\,}

\def\aga{\texttt{a}}

\newcommand{\RESagaDia}{\,\rotatebox[origin=c]{90}{$\langle\rotatebox[origin=c]{-90}{$\aga$}\rangle$}\,}

\newcommand{\RESagaBox}{\,\rotatebox[origin=c]{90}{$[\mkern1.8mu{\rotatebox[origin=c]{-90}{$\aga$}}\mkern1.8mu]$}\,}

\newcommand{\ls}{\lbrack}
\newcommand{\rs}{\rbrack}
\newcommand{\lc}{\langle}
\newcommand{\rc}{\rangle}

\newcommand{\LRA}{\leftrightarrow}
\newcommand\val[1]{{\lbrack\!\lbrack} {#1}{\rbrack\!\rbrack}}

\newcommand{\pand}{\wedge}
\newcommand{\por}{\vee}
\newcommand{\pra}{\rightarrow}

\def\aol{\rule[0.5865ex]{1.38ex}{0.1ex}}

\def\pdra{\mbox{$\,{\rotatebox[origin=c]{-90}{\raisebox{0.12ex}{$\pand$}}{\mkern-2mu\aol}}\,$}}

\def\pdla{\mbox{\rotatebox[origin=c]{180}{$\,{\rotatebox[origin=c]{-90}{\raisebox{0.12ex}{$\pand$}}{\mkern-2mu\aol}}\,$}}}

\def\fCenter{{\mbox{$\ \vdash\ $}}}

\EnableBpAbbreviations
\newcommand{\commment}[1]{}

\begin{document}
\title{A Proof-Theoretic Semantic Analysis \\ of Dynamic Epistemic Logic}

\author{Sabine Frittella\footnote{Laboratoire d'Informatique Fondamentale de Marseille (LIF) - Aix-Marseille Universit\'{e}.}\;\; Giuseppe Greco\footnote{Department of Values, Technology and Innovation - TU Delft.}\;\; Alexander Kurz\footnote{Department of Computer Science - University of Leicester.}\\ Alessandra Palmigiano\footnote{Department of Values, Technology and Innovation - TU Delft. The research of the second and fourth author has been made possible by the NWO Vidi grant 016.138.314, by the NWO Aspasia grant 015.008.054, and by a Delft Technology Fellowship awarded in 2013 to the fourth author.}\;\; Vlasta Sikimi\'c\footnote{The Institute of Philosophy of the Faculty of Philosophy - University of Belgrade.}}

\maketitle

\begin{abstract}
\noindent The present paper provides an analysis of the existing proof systems for dynamic epistemic logic from the viewpoint of proof-theoretic semantics. Dynamic epistemic logic is one of the best known members of a family of logical systems which have been successfully applied to diverse scientific disciplines, but the proof theoretic treatment of which presents many difficulties.
After an illustration of the proof-theoretic semantic principles most relevant to the treatment of logical connectives, we turn to illustrating the main features of display calculi, a proof-theoretic paradigm which has been successfully employed to give a proof-theoretic semantic account of modal and substructural logics. Then, we review some of the most significant proposals of proof systems for dynamic epistemic logics, and we critically reflect on them in the light of the previously introduced proof-theoretic semantic principles. The contributions of the present paper include a generalisation of Belnap's cut elimination metatheorem for display calculi, and a revised version of the display-style calculus D.EAK \cite{GKPLori}. We verify that the revised version satisfies the previously mentioned proof-theoretic semantic principles, and show that it enjoys cut elimination as a consequence of the generalised metatheorem. 
{\em Math.\ Subject Class.\ 2010:} 03B45, 06D50, 06D05, 03G10, 06E15.
\end{abstract}

\newpage
\tableofcontents
\newpage

%\textcolor{red}{\textbf{TO TALK ABOUT THE ROLE OF THE DISPLAY PROPERTY FOR THE CUT ELIMINATION}}

\section{Introduction}
In recent years, driven by applications in areas spanning from program semantics to game theory, the logical formalisms pertaining to the family of {\em dynamic logics} \cite{Kozen, Ditmarsch1} have been very intensely investigated, giving rise to a proliferation of variants.

Typically, the language of a given dynamic logic is an expansion of classical propositional logic with an array of modal-type {\em dynamic operators}, each of which takes an {\em action} as a parameter. The set of actions plays in some cases the role of a set of indexes or parameters; in other cases, actions form a quantale-type algebra. When interpreted in relational models, the formulas of a dynamic logic express properties of the model encoding the present state of affairs, as well as the pre- and post-conditions of a given action. Actions formalize transformations of one model into another one, the updated model, which encodes the state of affairs after the action has taken place.

Dynamic logics have been investigated mostly w.r.t.~their semantics and complexity, while their proof-theoretic aspects have been comparatively not so prominent. However, the existing proposals of proof systems for dynamic logics witness a varied enough array of  methodologies, that a methodological evaluation is now timely.

The present paper is aimed at evaluating the current proposals of proof-systems for the best-known dynamic epistemic logics from the viewpoint of {\em proof-theoretic semantics}.

Proof-theoretic semantics \cite{Peter} is a theory of meaning which assigns formal proofs or derivations an autonomous semantic content. That is, formal proofs are treated as entities in terms of which meaning can be accounted for.
Proof-theoretic semantics has been very influential in an area of research in structural proof theory which aims at defining the meaning of logical connectives in terms of an analysis of the behaviour of the logical connectives inside the derivations of a given proof system. Such an analysis is possible only in the context of proof systems which perform well w.r.t.\ certain criteria; hence, one of the main themes in this area is to identify design criteria which both guarantee that the proof system enjoys certain desirable properties such as normalization or cut-elimination, and which make it possible to speak about the proof-theoretic meaning for given logical connectives.

An analysis of dynamic logics from a proof-theoretic semantic viewpoint is beneficial both for dynamic logics and for structural proof theory. Indeed,  such an analysis provides dynamic logics with  sound methodological and foundational  principles, and with an entirely novel perspective on the topic of dynamics and change, which is {\em independent} from the dominating model-theoretic methods. Moreover,  such an analysis provides structural proof theory with a novel array of case studies against which to test the generality of its proof-theoretic semantic principles, and with the opportunity to extend its modus operandi to still uncharted settings, such as the multi-type calculi  introduced in \cite{Multitype}.

\noindent The structure of the paper goes as follows: in section 2, we introduce the basic ideas of proof-theoretic semantics, as well as some of the principles in structural proof theory that were inspired by it, and we explain their consequences and spirit, in view of their applications in the following sections. In section 3, we prove a generalisation of Belnap's cut elimination metatheorem. In section 4, we review some of the most significant proposals of proof systems for dynamic epistemic logics, focusing mainly on the logic of Public Announcements (PAL) \cite{Plaza} and the logic of Epistemic Knowledge and Actions (EAK) \cite{BMS}, and we critically reflect on them in the light of the principles of proof-theoretic semantics stated in section 2; in particular, in subsection \ref{D.EAK},  we focus on the display-type calculus D.EAK for PAL/EAK introduced in \cite{GKPLori}: we highlight its critical issues---the main of which being that a smooth ({\em Belnap-style}) proof of cut-elimination is not readily available for it. In section 5, we expand on the final coalgebra semantics for D.EAK, which will be relevant for the following developments. In section 6, we propose a revised version of D.EAK, discuss why the revision is more adequate for proof-theoretic semantics, and finally prove the cut-elimination theorem for the revised version as a consequence of the metatheorem proven in section 3.
 In section 7, we collect some conclusions and indicate further directions.  Most of the proofs and derivations are collected in appendices \ref{Special rules}, \ref{cut in D'.EAK} and \ref{ssec: completeness}.

\section{Preliminaries on proof-theoretic semantics and Display Calculi}
In the present section, we review and discuss the proof-theoretic notions which will be used in the further development of the paper. In the following subsection, we outline the conceptual foundations of proof-theoretic semantics; in subsection \ref{ssec:DisplayLogic}, Belnap-style display calculi will be discussed; in subsection \ref{ssec:Wansing} a refinement of Belnap's analysis, due to Wansing, will be reported on. Our presentation is certainly not exhaustive, and will  limit itself to targeting the issues  needed in the further development of the paper. The reader is referred to \cite{Peter, Sch06} for a detailed presentation of proof-theoretic semantics, and to \cite{Wansing, Wa00} for a discussion of proof-theoretic semantic principles in structural proof theory.

\subsection{Basic ideas in proof-theoretic semantics}

\emph{Proof-theoretic semantics} is a line of research which covers both philosophical and technical aspects, and is concerned with methodological issues. Proof-theoretic semantics is based on the idea that a purely inferential theory of meaning is possible. That is, that the meaning of expressions (in a formal language or in natural language) can be captured purely in terms of the proofs and the inference rules which participate in the generation of the given expression, or in which the given expression participates. This {\em inferential} view is opposed to the mainstream {\em denotational} view on the theory of meaning, and is influential in e.g.\ linguistics, linking up to the idea, commonly attributed to Wittgenstein, that `meaning is use'. In proof theory, this idea links up with Gentzen's famous observation about the introduction and elimination rules of his natural deduction calculi:

\begin{quote}
`The introductions represent, as it were, the definitions of the symbols concerned, and the eliminations are no more, in the final analysis, than the consequences of these definitions. This fact may be expressed as follows: In eliminating a symbol, we may use the formula with whose terminal symbol we are dealing only in the sense afforded it by the introduction of that symbol'. (\cite{Gentzen} p.\,80)
\end{quote}

In the proof-theoretic semantic literature, this observation is brought to its consequences: rather than viewing proofs as entities the meaning of which is dependent on denotation, proof-theoretic semantics assigns proofs (in the sense of formal deductions) an \emph{autonomous semantic role}; that is, proofs are entities in terms of which meaning can be accounted for.

Proof-theoretic semantics has inspired and unified much of the research in structural proof theory focusing on the purely inferential characterization of logical constants (i.e.\ logical connectives) in the setting of a given proof system.

\subsection{Display calculi}\label{ssec:DisplayLogic}

Display calculi are among the approaches in structural proof theory aimed at the uniform development of an inferential theory of meaning of logical constants aligned with the ideas of proof-theoretic semantics. Display calculi have been successful in giving adequate proof-theoretic accounts of logics---such as modal logics and substructural logics---which have notoriously been difficult to treat with other approaches. In particular, the contributions in this line of research which are most relevant to our analysis are Belnap's \cite{Belnap}, Wansing's \cite{Wansing}, Gor\'{e}'s \cite{Gore1}, and Restall's \cite{Restall}.

\paragraph{Display Logic.} Nuel Belnap introduced the first display calculus, which he calls {\em Display Logic} \cite{Belnap}, as a sequent system augmenting and refining Gentzen's basic observations on structural rules. Belnap's refinement is based on the introduction of a special syntax for the constituents of each sequent. Indeed, his calculus treats sequents $ X \vdash Y$ where $X$ and $Y$ are so-called \emph{structures}, i.e.\ syntactic objects inductively defined from formulas using an array of special connectives. Belnap's basic idea is that, in the standard Gentzen formulation, the comma symbol `$,$' separating formulas in the precedent and in the succedent of sequents can be recognized as a metalinguistic connective, of which the structural rules define the behaviour.

Belnap took this idea further by admitting not only the comma, but also several other connectives to keep formulas together in a structure, and called them {\em structural connectives}. Just like the comma in standard Gentzen sequents is interpreted contextually (that is, as conjunction when occurring on the left-hand side and as disjunction when occurring on the right-hand side), each structural connective typically corresponds to a pair of logical connectives, and is interpreted as one or the other of them  contextually (more of this in sections \ref{sec:final coalgebra sem} and \ref{D'.EAK}). Structural connectives maintain relations with one another, the most fundamental of which take the form of adjunctions and residuations. These relations make it possible for the calculus to enjoy the powerful property which gives it its name, namely, the {\em display property}. Before introducing it formally, let us agree on some auxiliary definitions and nomenclature: \emph{structures} are defined much in the same way as formulas, taking formulas as atomic components and closing under the given structural connectives; therefore, each structure can be uniquely associated with a generation tree. Every node of such a generation tree defines a {\em substructure}. A \emph{sequent}  $X\vdash Y$ is a pair of structures $X,Y$. The display property was introduced by Belnap, see Theorem 3.2 of \cite{Belnap} (where $X\vdash Y$ is called a consecution and $X$ the antecedent and $Y$ the consequent):

\begin{defn} \label{def: display prop} A proof system enjoys the {\em display property}  iff for every sequent $X \vdash Y$ and every substructure $Z$  of either  $X$ or  $Y$, the sequent  $X \vdash Y$ can be equivalently transformed, using the rules of the system, into a sequent which is either of the form $Z \vdash W$ or of the form $W \vdash Z$, for some structure $W$. In the first case, $Z$ is \emph{displayed in precedent position}, and in the second case, $Z$ is \emph{displayed in succedent position}.
The rules enabling this equivalent rewriting  are called \emph{display postulates}.
\end{defn}

Thanks to the fact that display postulates are based on adjunction and residuation,  in display calculi exactly one of the two alternatives mentioned in the definition above occurs. In other words, in a system enjoying the display property, any substructure of any sequent  $X \vdash Y$ is always displayed either only in  precedent position or only in succedent position. This is why we can talk about occurrences of substructures in {\em precedent} or in {\em succedent} position, even if they are nested deep within a given sequent, as illustrated in the following example:
\begin{center}
{\AX$Y \fCenter X > Z$
\UI$X\,; Y \fCenter Z$
\UI$Y\,; X \fCenter Z$
\UI$X \fCenter Y > Z$
\DisplayProof}
\label{example adj}
\end{center}
\noindent In the derivation above, the structure $X$ is on the right side of the turnstile, but it is  displayable on the left, and therefore is in  precedent position.   As we will see next, the display property is a crucial technical ingredient for display calculi, but it is also at the basis of
Belnap's methodology for characterizing operational connectives: according to Belnap, any logical  connective should be introduced {\em in isolation}, i.e., when it is introduced, the context on the side it has been introduced must be empty. The display property guarantees that this condition is not too restrictive.
\\

%%%%%%%%%

To illustrate the fundamental role played by the display property in the transformation steps of the cut elimination metatheorem,  consider the elimination step of the following cut application, in which the cut formula is principal on both premises of the cut.

\begin{center}
\scriptsize{
\begin{tabular}{c c c}
\bottomAlignProof
\AXC{\ \ \ $\vdots$ \raisebox{1mm}{$\pi_1$}}
\noLine
\UI$X\fCenter A$
\AXC{\ \ \ $\vdots$ \raisebox{1mm}{$\pi_2$}}
\noLine
\UI$Y \fCenter B$
\BI$X \,; Y \fCenter A \wedge B$
\AXC{\ \ \ $\vdots$ \raisebox{1mm}{$\pi$}}
\noLine
\UI$A \,; B \fCenter Z$
\UI$A \wedge B \fCenter Z$
\BI$X \,; Y \fCenter Z$
\DisplayProof
&$ \quad \rightsquigarrow \quad $&
\bottomAlignProof
\AXC{\ \ \ $\vdots$ \raisebox{1mm}{$\pi_2$}}
\noLine
\UI$Y \fCenter B$
\AXC{\ \ \ $\vdots$ \raisebox{1mm}{$\pi_1$}}
\noLine
\UI$X \fCenter A$
\AXC{\ \ \ $\vdots$ \raisebox{1mm}{$\pi$}}
\noLine
\UI$A \,; B\fCenter Z$
\dashedLine
\UI$A \fCenter Z < B$
\BI$X \fCenter Z < B$
\dashedLine
\UI$X \,; B \fCenter Z$
\dashedLine
\UI$B \fCenter X > Z$
\BI$Y \fCenter X > Z$
\dashedLine
\UI$X \,;Y \fCenter Z$
\DisplayProof
\end{tabular}}
\end{center}

\normalsize

The dashed lines in the prooftree on the right-hand side  correspond to applications of display postulates. Clearly, this transformation step  has been made  possible because the display postulates disassemble, as it were, compound structures so as to give us access to the immediate subformulas of the original cut formula, and then reassemble them so as to `put things back again'. Hence, it is possible to break down the original cut into two cut applications  on the immediate subformulas, as required by the original Gentzen strategy.

%%%%%%

\paragraph{Canonical cut elimination.} In \cite{Belnap}, a meta-theorem is proven, which gives sufficient conditions in order for a sequent calculus to enjoy cut-elimination.\footnote{Note that, as Belnap observed on pag.\ 389 in \cite{Belnap}: `The eight conditions are supposed to be a reminiscent of those of Curry' in \cite{Curry}.} This meta-theorem captures the essentials of the Gentzen-style cut-elimination procedure, and is the main technical motivation for the design of Display Logic.  Belnap's meta-theorem  gives a set of eight conditions on sequent calculi, which are relatively easy to check, since most of them are verified by inspection on the shape of the rules. Together, these conditions guarantee that the cut is eliminable in the given sequent calculus, and that the calculus enjoys the subformula property. When Belnap's metatheorem can be applied, it provides a much smoother and more modular route to cut-elimination than the Gentzen-style proofs. Moreover, as we will see later, a Belnap style cut-elimination theorem is robust with respect to adding structural rules and with respect to adding new logical connectives, whereas a Gentzen-style cut-elimination proof for the modified system cannot be deduced from the old one, but must be proved from scratch.

In a slogan, we could say that Belnap-style cut-elimination is to ordinary cut-elimination what canonicity is to completeness: indeed, canonicity provides a {\em uniform strategy} to achieve completeness. In the same way, the conditions required by Belnap's meta-theorem ensure that {\em one and the same}  given set of transformation steps is enough to achieve Gentzen-style cut elimination for any system satisfying them.\footnote{The relationship between canonicity and Belnap-style cut-elimination is in fact more than a mere analogy, see  \cite[Theorem 20]{Kracht}.}

In what follows, we review and discuss eight conditions which are stronger in certain respects than those in \cite{Belnap},\footnote{See also \cite{Be2, Restall} and the `second formulation' of condition C6/7 in subsection 4.4 of \cite{Wansing}.} and which define the notion of {\em proper display calculus} in \cite{Wansing}.\footnote{See the `first formulation' of conditions C6, C7 in subsection 4.1 of \cite{Wansing}.}

\paragraph{ C$_1$: Preservation of formulas.} This condition requires each formula occurring in a premise of a given inference to be the subformula of some formula in the conclusion of that inference. That is, structures may disappear, but not formulas. This condition is not included in the list of sufficient conditions of the cut-elimination meta-theorem, but, in the presence of cut-elimination, it guarantees the subformula property of a system.
Condition $C_1$ can be verified by inspection on the shape of the rules.
\paragraph{ C$_2$: Shape-alikeness of parameters.}
This condition is based on the relation of {\em congruence} between {\em parameters} (i.e., non-active parts) in inferences; %the rules;
the congruence relation is an equivalence relation which is meant to identify the different occurrences of the same formula or substructure along the branches of a derivation \cite[section 4]{Belnap}, \cite[Definition 6.5]{Restall}. %For example, in the following inference
%\[
%\AxiomC{$A \vdash A,A$}
%\UnaryInfC{$A\vdash A$}
%\DisplayProof
%\]
%the two occurrences of $A$ on the left are congruent and the three occurrences of $A$ on the right are congruent.
Condition C$_2$ requires that congruent parameters be occurrences of the same structure. This can be understood as a condition on the {\em design} of the rules of the system if the congruence relation is understood as part of the specification of each given rule; that is, each rule of the system comes with an explicit specification of which elements are congruent to which (and then the congruence relation is defined as the reflexive and transitive closure of the resulting relation). In this respect, C$_2$ is nothing but a sanity check, requiring that the congruence is defined in such a way that indeed identifies the occurrences which are intuitively ``the same''.

\paragraph{C$_3$: Non-proliferation of parameters.}   Like the previous one, also this condition is actually about the definition of the congruence relation on  parameters. Condition C$_3$ requires that, for every inference (i.e.\ rule application), each of its parameters  is congruent to at most one parameter in the conclusion of that inference. Hence, the condition stipulates that for a rule such as  the following,
\[
\AxiomC{$X \vdash Y$}
\UnaryInfC{$X \,; X \vdash Y$}
\DisplayProof
\]
\noindent  the structure $X$ from the premise is congruent to {\em only one} occurrence of $X$ in the conclusion sequent. Indeed, the introduced occurrence of $X$ should be considered congruent only to itself. Moreover, given that the congruence is an equivalence relation, condition C$_3$ implies that, within a given sequent, any substructure is congruent only to itself.

\begin{rmrk}
\label{rem: history tree}
Conditions C$_2$ and C$_3$ make it possible to follow the history of a formula along the branches of any given derivation. In particular, C$_3$ implies that the the history of any formula within a given derivation has the shape of a tree, which we refer to as the \emph{history-tree} of that formula in the given derivation. Notice, however, that the history-tree of a formula might have a different shape than the portion of the underlying derivation corresponding to it; for instance, the following application of the Contraction rule gives rise to a bifurcation of the history-tree of $A$ which is absent in the underlying branch of the derivation tree, given that Contraction is a unary rule.

\begin{tabular}{lcr}
\ \ \ \ \ \ \ \ \ \ \ \ \ \ \ \ \ \ \ \ \
\bottomAlignProof
\AXC{$\vdots$}
\noLine
\UI$A \,; A\fCenter X$
\UI$A\fCenter X$
\DisplayProof

 & \ \ \ \ \ \ \ \  &

$\unitlength=0.80mm
\put(0.00,10.00){\circle*{2}}
\put(10.00,0.00){\circle*{2}}
\put(20.00, 10.00){\circle*{2}}
\put(40.00, 10.00){\circle*{2}}
\put(0.00,10.00){\line(1,-1){10}}
\put(10.00,0.00){\line(1,1){10}}
\put(40.00, 0.00){\circle*{2}}
\put(40.00,0.00){\line(0,1){10}}
$
\\
\end{tabular}

\end{rmrk}

\paragraph{C$_4$: Position-alikeness of parameters.} This condition bans any rule in which a (sub)structure in precedent (resp.~succedent) position in a premise is congruent to a (sub)structure in succedent (resp.\ precedent) position in the conclusion.

\bigskip

\paragraph{C$_5$: Display of principal constituents.} This condition requires that any principal occurrence be always either the entire antecedent or the entire consequent part of the sequent in which it occurs. In the following section, a generalization of this condition will be discussed, in view of its application to the main focus of interest of the present paper.

\bigskip
The following conditions C$_6$ and C$_7$ are not reported below as they are stated in the original paper \cite{Belnap}, but as they appear in \cite[subsection 4.1]{Wansing}. More about this difference is discussed in section \ref{ssec:further directions}.
\paragraph{C$_6$: Closure under substitution for succedent parameters.} This condition requires  each rule to be closed under simultaneous substitution of arbitrary structures for congruent formulas which occur in succedent position.
Condition C$_6$ ensures, for instance, that if the following inference is an application of the rule $R$:

\begin{center}
\AX$(X \fCenter Y) \big([A]^{suc}_{i} \,|\, i \in I\big)$
\RightLabel{$R$}
\UI$(X' \fCenter Y') [A]^{suc}$
\DisplayProof
\end{center}

\noindent and $\big([A]^{suc}_{i} \,|\, i \in I\big)$ represents all and only  the occurrences of $A$ in the premiss which are congruent to the occurrence of $A$  in the conclusion\footnote{Clearly, if $I = \varnothing$, then the occurrence of $A$ in the conclusion is congruent to itself.}, %is congruent to all the occurrences $[A]^{suc}_{i}$ if the set of indices $I$ is nonempty (otherwise $A$ is congruent to itself),
then also the following inference is an application of the same rule $R$:

\begin{center}
\AX$(X \fCenter Y) \big([Z/A]^{suc}_{i} \,|\, i \in I\big)$
\RightLabel{$R$}
\UI$(X' \fCenter Y') [Z/A]^{suc}$
\DisplayProof
\end{center}

\noindent where the structure $Z$ is substituted for $A$.

\noindent This condition caters for the step in the cut elimination procedure in which the cut needs to be ``pushed up'' over rules in which the cut-formula in  succedent position  is parametric. Indeed, condition C$_6$ guarantees that, in the picture below, a well-formed subtree $\pi_1[Y/A]$ can be obtained from $\pi_1$ by replacing any occurrence of $A$ corresponding to a node in the history tree of the cut-formula $A$ by $Y$, and hence  the following transformation step is guaranteed go through uniformly and ``canonically'':

\begin{center}
\footnotesize{
\bottomAlignProof
\begin{tabular}{lcr}
\AXC{\ \ \ $\vdots$ \raisebox{1mm}{$\pi'_1$}}
\noLine
\UI$X' \fCenter A$
\noLine
\UIC{\ \ \ $\vdots$ \raisebox{1mm}{$\pi_1$}}
\noLine
\UI$X \fCenter A$
\AXC{\ \ \ $\vdots$ \raisebox{1mm}{$\pi_2$}}
\noLine
\UI$A \fCenter Y$
\BI$X \fCenter Y$
\DisplayProof
 & $\rightsquigarrow$ &
\bottomAlignProof
\AXC{\ \ \ $\vdots$ \raisebox{1mm}{$\pi'_1$}}
\noLine
\UI$X' \fCenter A$
\AXC{\ \ \ $\vdots$ \raisebox{1mm}{$\pi_2$}}
\noLine
\UI$A \fCenter Y$
\BI$X' \fCenter Y $
\noLine
\UIC{\ \ \ \ \ \ \ \ \ $\vdots$ \raisebox{1mm}{$\pi_1[Y/A]$}}
\noLine
\UI$X \fCenter Y$
\DisplayProof
 \\
\end{tabular}
}
\end{center}
\noindent if each rule in $\pi_1$ verifies condition C$_6$.

\paragraph{C$_7$: Closure under substitution for precedent parameters.} This condition requires  each rule to be closed under simultaneous substitution of arbitrary structures for congruent formulas which occur in precedent position.
Condition C$_7$ can be understood analogously to C$_6$, relative to formulas in precedent position. Therefore, for instance, if the following inference is an application of the rule $R$:

\begin{center}
\AxiomC{$(X \fCenter Y) \big([A]^{pre}_{i} \,|\, i \in I\big)$}
\RightLabel{$R$}
\UnaryInfC{$ (X' \vdash Y') [A]^{pre}$}
\DisplayProof
\end{center}

\noindent then also the following inference is an instance of $R$:

\begin{center}
\AxiomC{$(X \fCenter Y) \big([Z/A]^{pre}_{i} \,|\, i \in I\big)$}
\RightLabel{$R$}
\UnaryInfC{$ (X' \vdash Y') [Z/A]^{pre}$}
\DisplayProof
\end{center}

\noindent Similarly to what has been discussed for condition C$_6$, condition C$_7$ caters for the step in the cut elimination procedure in which the cut needs to be ``pushed up'' over rules in which the cut-formula in  precedent position  is parametric.

\paragraph{C$_8$: Eliminability of matching principal constituents.}

This condition  requests a standard Gentzen-style checking, which is now limited to the case in which  both cut formulas  are {\em principal}, i.e.~each of them has been introduced with the last rule application of each corresponding subdeduction. In this case, analogously to the proof  Gentzen-style, condition C$_8$ requires being able to transform the given deduction into a deduction with the same conclusion in which either the cut is eliminated altogether, or is transformed in one or more applications of cut involving proper subformulas of the original cut-formulas.

\bigskip

\paragraph{Rules introducing logical connectives.}   In display calculi, these rules, sometimes referred to as {\em operational rules} as opposed to the structural rules, typically occur in two flavors: operational rules which translate one structural connective in the premises in the corresponding connective in the conclusion, and operational rules in which both the operational connective and its structural counterpart are introduced in the  conclusion. An example of this pattern is provided below for the case of the modal operator `diamond':

\[
\AX$\circ A \fCenter X$
\RightLabel{$\Diamond_{L}$}
\UI$\Diamond A \fCenter X$
\DisplayProof
\qquad
\AX$X \fCenter A$
\RightLabel{$\Diamond_{R}$}
\UI$\circ X \fCenter \Diamond A$
\DisplayProof
\]

\noindent  This introduction pattern obeys very strict criteria, which will be expanded on in the next subsection.  From this example, it is clear that the introduction rules capture the rock bottom behavior of the logical connective in question; additional properties (for instance, normality, in the case in point), which might vary depending on the logical system, are to be captured at the level of additional (purely  structural) rules. This enforces a  clear-cut division of labour between operational rules, which only encode the basic proof-theoretic meaning of logical connectives, and structural rules, which account for all extra relations and properties, and which can be modularly added or removed, thus accounting for the space of logics.

Summing up, the  two main benefits of display calculi are a ``canonical'' proof of cut elimination, and an explicit and modular account of logical connectives.

\subsection{Wansing's criteria}\label{ssec:Wansing}

In \cite[subsubsection 1.3]{Wansing}, referring to the well known idea that `a proof-theoretic semantics exemplifies the Wittgensteinian slogan that meaning is use', Wansing stresses that, for this slogan to serve as a conceptual  basis for a general inferential theory of meaning, `use' should be understood as `{\em correct} use'. The consequences of the idea of meaning as correct use then precipitate into  the following principles for the introduction rules for operational connectives, which he discusses in the same subsection and which are reported below. These principles are hence to be understood as the general requirements a (sequent-style) proof system needs to satisfy in order to encode the correct use, and hence for being suitable for proof-theoretic semantics.

\paragraph{Separation.} This principle requires that the introduction rules for a given connective $f$ should not exhibit any other connective rather than $f$. Hence the meaning of a given operational connective cannot be dependent from any other operational connectives. For instance, the following rule does not satisfy  \emph{separation}:

\[
\AxiomC{$ \Box \Gamma \vdash  A, \Diamond \Delta$}
\UnaryInfC{$ \Box \Gamma \vdash \Box A\,, \Diamond \Delta $}
\DisplayProof
\]

\noindent This criterion does not ban the possibility of defining  composite connectives; however, it ensures that the dependence relation between connectives creates no  vicious circles. In fact, as it is formulated, this criterion is much stronger, since it requires every connective to be independent of any other.

\paragraph{Isolation.} This is a stronger requirement than separation, and stipulates that, in addition, the precedent (resp.\ succedent) of the conclusion sequent in a left (resp.\ right) introduction rule must not exhibit any structure operation. In \cite{Belnap}, Belnap explains this requirement by remarking that an introduction rule with nonempty context on the principal side would fail to account for the meaning of the logical connective involved in a context-independent way.

\paragraph{Segregation.} This is an even stronger requirement than isolation, and stipulates that, in addition, also the auxiliary formulas in the premise(s) must occur within an empty context. This property appears under the name of \emph{visibility} in \cite{BatFagSam}.\footnote{In \cite{TrendsXIII}, following ideas from \cite{BatFagSam}, the visibility property has been identified as an essential ingredient to generalise Belnap's metatheorem beyond display calculi.}

\paragraph{Weak symmetry.} This requirement stipulates that each introduction rule for  a given connective $f$ should either belong to a set of rules $(f \vdash)$ which introduce $f$ on the left-hand side of the turnstile $\vdash$ in the conclusion sequent, or to a set of rules $(\vdash f)$ which introduce $f$ on the right-hand side of the turnstile $\vdash$ in the conclusion sequent. Understanding the  either-or as exclusive disjunction, this criterion prevents an operational connective to be introduced  on both sides by the application of one and the same rule. Thus, weak symmetry stipulates that the sets $(f\vdash)$ and $(\vdash f)$ be disjoint. However, weak symmetry does not exclude that either $(f\vdash)$ or $(\vdash f)$ be empty.

\paragraph{Symmetry.} This condition strengthens weak symmetry by requiring both $(f\vdash)$ and $(\vdash f)$ to be nonempty for each connective $f$. Rather than a requirement on individual rules, this principle is a requirement on the set of the introduction rules for any given connective. Notice that symmetry does not exclude the possibility of having, for instance, two rules that introduce a given connective on the left and one that introduces it on the right side of the turnstile.

\paragraph{Weak explicitness.}  An introduction  rule for $f$ is {\em weakly explicit} if $f$ occurs only in the conclusion of a rule  and not in its premisses.

\paragraph{Explicitness.} An introduction  rule for $f$ is {\em explicit} if it is weakly explicit and in addition to this, $f$ appears {\em only once} in the conclusion of the rule.

\bigskip

\noindent The following principles are of a more global nature, which involves the proof system as a whole:

\paragraph{Unique characterization.} This principle requires each logical connective to be uniquely characterized by its behaviour in the system, in the following sense. Let $\Lambda$ be a logical system with a syntactic presentation $S$ in which $f$ occurs. Let $S^\ast$ be the result of rewriting $f$ everywhere in S as $f^\ast$, and let $\Lambda\Lambda^\ast$  be the system presented by the union $SS^\ast$ of $S$ and $S^\ast$ in the combined language with both $f$ and $f^\ast$. Let $A_f$ denote a formula (in this language) that contains a certain occurrence of $f$, and let $A_{f^\ast}$ denote the result of replacing this occurrence of $f$ in $A_f$ by $f^\ast$. The connectives $f$ and $f^\ast$ are  {\em uniquely characterized} in $\Lambda\Lambda^\ast$ (cfr.\ \cite[subsubsection 1.4]{Wansing}) if for every formula $A_f$ in the language of $\Lambda\Lambda^\ast$, $A_f$ is provable in $SS^\ast$ iff $A_{f^\ast}$ is provable in $SS^\ast$.

\paragraph{Do\v{s}en's principle.} Hilbert style presentations  are modular in the following sense: if $\Lambda_1$ and $\Lambda_2$ are finitely axiomatizable logics over the same language and $\Lambda_1$ is stronger than $\Lambda_2$, then an axiomatization of $\Lambda_2$ can be obtained from one of $\Lambda_1$ by adding finitely many axioms to it.
This makes it possible to modularly generate all finite axiomatic extensions of a given logic. Although it is arguably more difficult to achieve an analogous degree of modularity in the sequent calculi presentation, a principle aimed to achieve it has been advocated by Wansing under the name of \emph{Do\v{s}en's principle} (cfr.\ \cite[subsubsection 1.5]{Wansing}): ``The rules for the logical operations are never changed; all changes are made in the structural rules''. Thus, suitable finite axiomatic extensions of a given logic $L$ can be captured by adding structural rules to the proof system associated with $L$.
Display calculi are particularly suitable to implement Do\v{s}en's principle. As remarked early on, besides featuring structural rules which encode properties of single structural connectives (which is the case e.g.~of the rule exchange), display calculi typically feature rules which concern the interaction between different structural connectives (the adjunction between two structural connectives is an example of the latter type of rule, see for instance the rules applied in the example on page \pageref{example adj}).
%For instance,  an intuitionistic logic from the classical one can be obtained by forbidding certain structural rules (like the $Grishin$ rules), while in classical logi??

\paragraph{Cut-eliminability.} Finally, Wansing considers the eliminability of the \emph{cut} rule as an important requirement for the proof-theoretic semantics of logical connectives.

\section{Belnap-style metatheorem for quasi proper display calculi}
In the present section, we discuss a slight extension of Wansing's notion of proper display calculus (cf.\ Subsection \ref{ssec:DisplayLogic}), and prove its associated Belnap-style cut elimination metatheorem. The cut elimination for the calculus D'.EAK introduced in Section \ref{sec:C_1-C7for D'.EAK} (see also Appendix \ref{cut in D'.EAK}) will be derived as an instance of the metatheorem below.

\subsection{Quasi proper display calculi}
\label{ssec:quasi-def}
\begin{defn}
\label{def:quasiproper}
A sequent calculus is a {\em quasi proper display calculus} if it verifies conditions C$_1$, C$_2$, C$_3$, C$_4$, C$_6$, C$_7$, C$_8$ of section \ref{ssec:DisplayLogic}, and moreover it satisfies the following conditions C$'_5$, C$''_5$ and C$'_8$:
%\noindent The next condition is given as in \cite[Definition 6.8]{Restall}: %It is slightly more general than the one which specifically applies to  display calculi as originally defined in \cite{Belnap}.
\paragraph{C$'_5$: Quasi-display of principal constituents.} If a formula $A$ is principal in the conclusion sequent $s$ of a derivation $\pi$, then $A$ is in display, unless $\pi$ consists only of its conclusion sequent $s$ (i.e.\ $s$ is an axiom).

\paragraph{C$''_5$: Display-invariance of axioms.} If a display rule can be applied to an axiom $s$, the result of that rule application is again an axiom.

\paragraph{C$'_8$: Closure of axioms under cut.} If $X\vdash A$ and $A\vdash Y$ are axioms, then $X\vdash Y$ is again an axiom.

\end{defn}

\noindent Notice that condition C$_5$ in Subsection \ref{ssec:DisplayLogic} is stronger than both C$'_5$ and C$''_5$, and that
the strength of condition C$'_5$ is intermediate between that of C$_5$ and of the following one, appearing in  \cite[Definition 6.8]{Restall}:

\paragraph{C$'''_5$: Single principal constituents.}
This condition requires that, in the conclusion of any rule, there be at most one  {\em non-parametric} formula---which is the formula introduced by the application of the rule in question---unless the rule is an axiom.

\medskip
The above condition C$'''_5$ is introduced in \cite{Restall} within a setting accounting  for  sequent calculi which do not necessarily enjoy the full display property. The calculi considered in \cite{Restall} are such that the introduction rules do not need to enjoy the requirement of isolation (cf.\ Chapter 6), and the (multiple) cut rule applies at any depth. The calculus introduced in Section \ref{D'.EAK} enjoys the full display property, therefore the following cut rule, in which both cut formulas occur in isolation:\begin{center}
\AX$X \fCenter A$
\AX$A \fCenter Y$
\RightLabel{$Cut$}
\BI$X \fCenter Y$
\DisplayProof
\end{center} will be taken as primitive in it without loss of generality, as is standardly done in display calculi. However, the calculus in Section \ref{D'.EAK} fails to enjoy the property of \emph{isolation}, which typically plays a role in the cut elimination metatheorem for display calculi, and indeed appears in \cite{Wansing} as condition C$_5$.
 In the next subsection, we show that, even when the cut rule is the one above, requiring the combination of C$'_5$ and C$''_5$ suffices.\footnote{In \cite{Multitype}, we give a metatheorem which is based on a different tradeoff: on the one hand, we will not require the full display property, but on the other we will require a condition close to segregation. }

\subsection{Belnap-style metatheorem}
\label{ssec:B-style metathrm quasiprop}

The aim of the present subsection is to prove the following theorem:\commment{
In order to account for calculi admitting Contraction as a structural rule, a  slightly more general version of the cut rule (e.g.\ the Gentzen's mix rule, also known as Multi-Cut) can be introduced
%, such as the following one
\cite[Footnote 5]{Belnap} \cite[Chapter 6]{Re},

\begin{center}
{\fns
\begin{tabular}{lr}
\AX$X \fCenter A$
\AX$(X' \fCenter Y') [A_1^{pre} \ldots A_n^{pre}]$
\RightLabel{$Cut^{pre}$}
\BI$X' \fCenter Y' [X_1^{pre} \ldots X_m^{pre}]$
\DisplayProof
 &
\AX$(X' \fCenter Y') [A_1^{suc} \ldots A_n^{suc}]$
\AX$A \fCenter X$
\RightLabel{$Cut^{suc}$}
\BI$X' \fCenter Y' [X_1^{suc} \ldots X_m^{suc}]$
\DisplayProof
 \\
\end{tabular}
}
\end{center}
where $m\leq n$, or even more as follows:
\begin{center}
{\fns
\begin{tabular}{lr}
\AX$(Z \fCenter W) [A^{suc}]$
\AX$(X' \fCenter Y') [A_1^{pre} \ldots A_n^{pre}]$
\RightLabel{$Cut'^{pre}$}
\BI$X' \fCenter Y' [X_1^{pre} \ldots X_m^{pre}]$
\DisplayProof
 &
\AX$(X' \fCenter Y') [A_1^{suc} \ldots A_n^{suc}]$
\AX$Z \fCenter W [A^{pre}]$
\RightLabel{$Cut'^{suc}$}
\BI$X' \fCenter Y' [X_1^{suc} \ldots X_m^{suc}]$
\DisplayProof
 \\
\end{tabular}
}
\end{center}
where $m\leq n$, $(Z \vdash W) [A^{suc}]$ is display-equivalent to $X \fCenter A$ and $Z \vdash W [A^{pre}]$ is display-equivalent to $A \vdash X$.
However, the proof with multi-cut is entirely analogous to the one which will be given here, so for the sake of simplicity  we present it for the cut rule in its original formulation.}

\begin{theorem}
\label{thm:meta}
Any calculus satisfying conditions C$_2$, C$_3$, C$_4$, $C'_5$, $C''_5$, C$_6$, C$_7$, C$_8$, and $C'_8$ enjoys cut elimination. If C$_1$ is also satisfied, then the calculus enjoys the subformula property.
\end{theorem}

\begin{proof} This is a generalization of the proof in \cite[Section 3.3, Appendix A]{Wan02}. For the sake of conciseness, we will expand only on the parts of the proof which depart from that treatment.

%Without loss of generality, we can assume that the cut in the original proof is a Left Cut (the proof for a Right Cut is symmetric), that is we are going to consider the following situation
Our original derivation is
\begin{center}
\AXC{\ \ \ $\vdots$ \raisebox{1mm}{$\pi_1$}}
\noLine
\UI$X\fCenter A$
\AXC{\ \ \ $\vdots$ \raisebox{1mm}{$\pi_2$}}
\noLine
\UI$A \fCenter Y$
\BI$X \fCenter Y$
\DisplayProof
\end{center}

\paragraph*{Principal stage: both cut formulas are principal.}
\noindent There are three subcases.

If the end sequent $X \fCenter Y$ is identical to the conclusion of $\pi_1$ (resp.\ $\pi_2$), then we can eliminate the cut simply replacing the derivation above with $\pi_1$ (resp.\ $\pi_2$).

If the premises $X \vdash A$ and $A\vdash Y$ are axioms, then, by C$'_8$, the conclusion $X\vdash Y$ is an axiom, therefore the cut can be eliminated by simply replacing the original derivation with $X \fCenter Y$.

If one of the two premises of the cut in the original derivation is not an axiom, then, by C$_8$, there is a proof of $X \fCenter Y$ which uses the same premise(s) of the original derivation and which involves only cuts on proper subformulas of $A$.

\paragraph*{Parametric stage: at least one cut formula is parametric.}
There are two subcases: either one cut formula is principal or they are both parametric.

Consider the subcase in which one cut formula is principal. W.l.o.g.\ we assume that the cut-formula $A$ is principal in the the left-premise $X \vdash A$ of the cut in the original proof (the other case is symmetric). As discussed in Remark \ref{rem: history tree}, conditions C$_2$ and C$_3$ make it possible to consider the history-tree of the right-hand-side cut formula $A$ in $\pi_2$. %Consider the history tree of $A$ in $\pi_2$: each uppermost ancestor of $A_i$ can be parametric or principal.
 The situation can be pictured as follows:
\begin{center}
\AXC{\ \ \ $\vdots$ \raisebox{1mm}{$\pi_1$}}
\noLine
\def\fCenter{\vdash}
\UI$ X \fCenter A$
\AXC{\ \ \ \ \ $\vdots$ \raisebox{1mm}{$\pi_{2.i}$}}
\noLine
\UI$\underline{A}_i \fCenter Y_i$
\noLine
\UIC{$\ \ \ddots$}
\noLine
\AXC{\ \ \ \ \ $\vdots$ \raisebox{1mm}{$\pi_{2.j}$}}
\noLine
\UI$(X_j \fCenter Y_j)[\underline{A}_j]^{pre}$
\noLine
\UIC{$\vdots$}
\noLine
%\AXC{$\ \ \ \ \ \ \ \ \ \ \ \ \cdots{\phantom{\vdash}}$}
\AXC{\ \ \ \ \ $\vdots$ \raisebox{1mm}{$\pi_{2.k}$}}
\noLine
\UI$(X_k \fCenter Y_k)[\overline{A}_k]^{pre}$
\noLine
\UIC{$\!\!\!\!\!\!\!\!\!\!\!\!\!\!\!\!\!\!\!\!\!\!\!\!\!\!\!\!\!\!\!\!\!\!\!\!\!\!\!\!\!\!\!\!\iddots$}
\noLine
\TIC{$\ \ \ \ \ \ \ \rule[-5.2mm]{0mm}{0mm}\ddots\vdots\iddots\rule{0mm}{10mm}$ \raisebox{1mm}{$\pi_2$}}
\noLine
\def\fCenter{\vdash}
\UIC{$\ \ \ A \fCenter Y$}
%\UI$\!\!\!\!\!\!\!\!\!\!\!\!\!\!\!\!\!\!\!\!\! A \fCenter Y$
%\LeftLabel{\fns{\ScissorRight}}
\BI$ X \fCenter Y$
\DisplayProof
\end{center}
\smallskip
where, for $i, j, k \in \{1, \ldots, n\}$, the nodes \[\underline{A}_i \fCenter Y_i, \quad (X_j\vdash Y_j)[A_j]^{pre}, \ \mbox { and }\  (X_k \fCenter Y_k)[\overline{A}_k]^{pre}\] represent the three ways in which the leaves $A_i$, $A_j$ and $A_k$ in the history-tree of $A$ in $\pi_2$ can be introduced, and which will be discussed below. The notation $\underline{A}$ and (resp.\ $\overline{A}$) indicates that the given occurrence is principal (resp.\ parametric). Notice that condition C$_4$ guarantees that all occurrences in the history of $A$ are in precedent position in the underlying derivation tree.

Let $A_l$ be introduced as a parameter (as represented in the picture above in the conclusion of $\pi_{2.k}$ for $A_l = A_k$). Assume that $(X_k \vdash Y_k)[\overline{A}_k]$ is the conclusion of an application \emph{inf} of the rule \emph{Ru} (for instance, in the calculus of section \ref{D'.EAK}, this situation arises if $A_k$ has been introduced with an application of Weakening). Since $A_k$ is a leaf in the history-tree of $A$, we have that $A_k$ is congruent only to itself in $X_k \vdash Y_k$. Hence, C$_7$ implies that it is possible to substitute $X$ for $A_k$ by means of an application of the same rule \emph{Ru}. That is,  $(X_k \vdash Y_k)[\overline{A}_k]$ can be replaced by $(X_k \vdash Y_k)[X/\overline{A}_k]$. %If all the uppermost ancestors $A_i$ are parametric, then the transformation gives a new derivation in which the given application of cut disappears and no new cut applications have been introduced.

Let $A_l$ be introduced as a principal formula. The corresponding subcase in \cite{Wan02} splits into two subsubcases: either $A_l$ is introduced in display or it is not.

If $A_l $ is in display (as represented in the picture above in the conclusion of $\pi_{2.i}$ for $A_l = A_i$), then we form a subderivation using $\pi_1$  and $\pi_{2.i}$ and applying cut as the last rule.
%Condition C$_7$ implies that the substitution of $X$ for $A_i$ in $\pi_2$ gives rise to an admissible proof $\pi_2[X/\underline{A}_i]$ (use C$_6$ for the symmetric case).

If $A_l $ is not in display (as represented in the picture above in the conclusion of $\pi_{2.j}$ for $A_l = A_j$), then condition C$'_5$ implies that $(X_j \vdash Y_j)[\underline{A}_j]^{pre}$ is an axiom (so, in particular, there is at least another occurrence of $A$ in succedent position), and C$''_5$ implies that some axiom $\underline{A}_j \vdash Y'_j$ exists, which is display-equivalent to the first axiom, and in which $A_j$ occurs in display.  Let $\pi'$ be the derivation which transforms $\underline{A}_j \vdash Y'_i$ into $(X_j \vdash Y_j)[\underline{A}_j]^{pre}$. We form a subderivation using $\pi_1$ and $\underline{A}_j \vdash Y'_j$ and joining them with a cut application, then attaching $\pi'[X/A_j]^{pre}$ below the new cut.

The transformations just discussed explain how to transform the leaves of the history tree of $A$. Finally, condition C$_7$ implies that substituting $X$ for each occurrence of $A$ in the history tree of the cut formula $A$ in $\pi_2$ (or in a display-equivalent proof $\pi'$) gives rise to an admissible derivation $\pi_2[X/A]^{pre}$ (use C$_6$ for the symmetric case).

\commment{
We illustrate this special situation in the diagrammatic proof below:

\begin{center}
\begin{tabular}{lcr}
\bottomAlignProof
\AXC{\ \ \ $\vdots$ \raisebox{1mm}{$\pi_1$}}
\noLine
\UI$X \fCenter A$
%
%\AXC{\raisebox{1mm}{$\pi_{2.i}$}}
%\noLine
%\AXC{$\ \ \ \ \cdots{\phantom{\vdash}}$}
\AXC{$(X_i \fCenter Y_i) [\underline{A}_i]^{pre} [\underline{A}]^{suc}$}
%\AXC{$\ \ \ \ \cdots{\phantom{\vdash}}$}
\noLine
%\UIC{\ \ \,$\vdots$ \raisebox{1mm}{$\pi_2[a]^{pre}$}}
\UIC{\ \ \ $\vdots$ \raisebox{1mm}{$\pi_2$}}
\noLine
\UI$A\fCenter Y$
\BI$X \fCenter Y$
\DisplayProof

 & $\rightsquigarrow$ &

\bottomAlignProof
\AXC{\ \ \ $\vdots$ \raisebox{1mm}{$\pi_1$}}
\noLine
\UI$X \fCenter A$
%
%\AXC{\raisebox{1mm}{$\pi'_{2.i}$}}
%\noLine
\AXC{$\underline{A}_i \fCenter Y' [\underline{A}]^{suc}$}
\BI$X \fCenter Y' [A]^{suc}$
\noLine
\UIC{\ \ \ $\vdots$ \raisebox{1mm}{$\pi'$}}
\noLine
\UI$(X_i \fCenter Y_i) [X/A_i]^{pre} [A]^{suc}$
\noLine
%\UIC{\ \ \ \ \ \ \ \ \ \ \,$\rule[-5.2mm]{0mm}{0mm}\ddots\vdots\iddots\rule{0mm}{10mm}$ \raisebox{1mm}{$\pi_2 [X/p_u]$}}
\UIC{\ \ \ \ \ \ \ \ \ \ \ \ \ \ \ \ \ $\vdots$ \raisebox{1mm}{$\pi_2 [X/A]^{pre}$}}
\noLine
\UI$X \fCenter Y$
\DisplayProof
 \\
\end{tabular}
\end{center}
}
Summing up, this procedure generates the following proof tree:
\begin{center}
\AXC{\ \ \ $\vdots$ \raisebox{1mm}{$\pi_1$}}
\noLine
\def\fCenter{\vdash}
\UI$X \fCenter A$
\AXC{\ \ \ \ \ $\vdots$ \raisebox{1mm}{$\pi_{2.i}$}}
\noLine
\UI$\underline{A}_i \fCenter Y_i$
%\LeftLabel{$Cut$}
\BI $ X \fCenter Y_i$
\noLine
\UIC{\ \  $\ddots$}
\AXC{\ \ \ $\vdots$ \raisebox{1mm}{$\pi_1$}}
\noLine
\UI$X \fCenter A$
%
%\AXC{\raisebox{1mm}{$\pi'_{2.i}$}}
%\noLine
\AXC{$\underline{A}_j \fCenter Y' [\underline{A}]^{suc}$}
\BI$X \fCenter Y' [A]^{suc}$
\noLine
\UIC{\ \ \ \ \ \ \ \ \ \ \ \ \ \ $\vdots$ \raisebox{1mm}{$\pi'[X/A]^{pre}$}}
\noLine
\UI$(X_j \fCenter Y_j) [X/A_j]^{pre} [A]^{suc}$
\noLine
\UIC{$\vdots\ $}
\AXC{\ \ \ \ \ $\vdots$ \raisebox{1mm}{$\pi_{2.k}$}}
\noLine
\UI$(X_k \fCenter Y_k)[\overline{X}/\overline{A}_k]^{pre}$
\noLine
\UIC{$\!\!\!\!\!\!\!\!\!\!\!\!\!\!\!\!\!\!\!\! \iddots$}
\noLine
\TIC{$\ \ \ \ \ \ \ \ \ \ \ \ \ \ \ \,\rule[-5.2mm]{0mm}{0mm}\ddots\vdots\iddots\rule{0mm}{10mm}$ \raisebox{1mm}{$\pi_2[X/A]^{pre}$}}
\noLine
\def\fCenter{\vdash}
\UIC{$X \fCenter Y$}
\DisplayProof
\end{center}

If, in the original derivation, the history-tree of the cut formula $A$ (in the right-hand-side premise of the given cut application)  contains at most one leaf $A_l$ which is principal, then the height of the new cuts is lower than the height of the original cut.

If, in the original derivation, the history-tree of the cut formula $A$ (in the right-hand-side premise of the given cut application) contains more than one leaf $A_l$ which is principal, then we cannot conclude that the height of the new cuts is always lower than the height of the original cut (for instance, in the calculus introduced in Section \ref{D'.EAK}, this situation may arise when two ancestors of a cut formula are introduced as principal, and then are identified via an application of the rule Contraction). In this case, we observe that in each newly introduced application of the cut rule, both cut formulas are principal. Hence, we can apply the procedure described in the Principal stage and transform the original derivation in a derivation in which the cut formulas of the newly introduced cuts have strictly lower complexity than the original cut formula.

Finally, as to the subcase in which both cut formulas are parametric, consider a proof with at least one cut. The procedure is analogous to the previous case. Namely, following the history of one of the cut formulas up to the leaves, and applying the transformation steps described above, we arrive at a situation in which, whenever new applications of cuts are generated, in each such application at least one of the cut formulas is principal. To each such cut, we can apply (the symmetric version of) the Parametric stage described so far.  %might generate new cuts in which either cut formulas are principal or not. In the first case, we can apply the procedure described for the Principal stage. In the second case, the cut formula in the right-hand-side premise of the new cut is principal, therefore we can apply the procedure for the case symmetric to the one considered before.
\end{proof}

\section{Dynamic Epistemic Logics and their proof systems}
\label{sec:del}

%Dynamic logics form a large family of non-classical logics, which are  designed to formalize change caused by actions of diverse nature: updates on the memory of a computer, displacements of moving robots in a given space, measurements in quantum physics models, belief updates, etc. In each interpretation, formulas express properties of the current model, and also   the pre- and post-conditions of a given action. Actions are interpreted as transformations of one model into another one, the {\em updated model}, which represents the state of affairs after the action has taken place. Languages for dynamic logics are expansions of classical propositional logic with dynamic modal operators, each of which takes an action as its parameter; dynamic operators are interpreted in terms of the transformation of models corresponding to their action-parameters.

%\texttt{There has been much activity in the field, extending the domain and applicability of the logics, e.g.~to belief revision, and developing semantic automated tools; for references and a comprehensive survey see [van Ditmarsch et al. 2007]. The field has, however, had less activity on the proof-theoretic side.}

In the present section, we first review the two best known logical systems in the family of dynamic epistemic logics, namely public announcement logic (PAL) \cite{Plaza}, and the logic of epistemic actions and knowledge (EAK) \cite{BMS}, focusing mainly on the latter one. Our presentation in subsection \ref{ssec:EAK}  is different but equivalent to the original version  from  \cite{BMS} (without common knowledge), and rather follows the presentation given in \cite{AMM} and in \cite{GKPLori}. In subsections \ref{Bad_proof_systems} and \ref{D.EAK} we discuss their existing proof-theoretic formalizations, particularly in relation to the  viewpoint of proof-theoretic semantics, and mention the system D.EAK as a promising approximation of a setting for proof-theoretic semantics. Finally, in subsection \ref{sec:final coalgebra sem}, we discuss the final coalgebra semantics, since this is a semantic environment in which all connectives of the language of D.EAK (and of its improved version D'.EAK) can be naturally interpreted.

\subsection{The logic of epistemic actions and knowledge}
\label{ssec:EAK}

The logic of epistemic actions and knowledge (further on EAK) is a logical framework which combines a multi-modal classical logic with a dynamic-type propositional logic. Static modalities in EAK are parametrized with agents, and their intended interpretation is epistemic, that is, $\langle \aga\rangle A$ intuitively stands for `agent $\aga$ thinks that $A$ might be the case'. Dynamic modalities in EAK are parametrized with epistemic \emph{action-structures} (defined below) and their intended interpretation is analogous to that of dynamic modalities in e.g.\ Propositional Dynamic Logic. That is, $\langle \alpha\rangle A$ intuitively stands for `the action $\alpha$ is executable, and after its execution $A$ is the case'. Informally, action structures loosely resemble Kripke models, and encode information about epistemic actions such as e.g.\ public announcements, private announcements to a group of agents, with or without (actual or suspected) wiretapping, etc. Action structures consist of a finite nonempty domain of action-states, a designated state, binary relations on the domain for each agent, and a precondition map. Each state in the domain of an action structure $\alpha$ represents the possible appearance of the epistemic action encoded by $\alpha$. The designated state represents the action actually taking place. Each binary relation of an action structure represents the type, or degree, of uncertainty entertained by the agent associated with the given binary relation about the action taking place; for instance, the agents' knowledge, ignorance, suspicions. Finally, the precondition function maps each state in the domain to a formula, which is intended to describe the state of affairs under which it is possible to execute the (appearing) action encoded by the given state. This formula encodes the \emph{preconditions} of the action-state. The reader is referred to \cite{BMS} for further intuition and concrete examples.

Let \textsf{AtProp} be a countable set of atomic propositions, and $\mathsf{Ag}$ be a  nonempty set (of agents). The set $\mathcal{L}$ of  formulas $A$  of the logic of epistemic actions and knowledge (EAK), and the set $\mathsf{Act}(\mathcal{L})$ of the {\em action structures} $\alpha$ {\em over} $\mathcal{L}$ are defined simultaneously as follows:

\begin{center}
$A := p\in \mathsf{AtProp} \mid \neg A \mid A\vee A \mid \langle\aga\rangle A \mid \langle\alpha\rangle A\;\; \ \ (\alpha\in \mathsf{Act}(\mathcal{L}), \aga\in \mathsf{Ag}),$

\end{center}
where an {\em action structure over} $\mathcal{L}$ is a tuple  $\alpha = (K, k, (\alpha_\aga)_{\aga\in\mathsf{Ag}}, Pre_\alpha)$, such that $K$ is a finite nonempty set,  $k\in K$,  $\alpha_\aga\subseteq K\times K$  and $Pre_\alpha: K\to \mathcal{L}$.

The symbol $Pre(\alpha)$ stands for $Pre_\alpha(k)$.  For each action structure $\alpha$ and every $i\in K$, let $\alpha_i := (K, i, (\alpha_\aga)_{\aga\in\mathsf{Ag}}, Pre_\alpha)$. Intuitively, the family of action structures $\{\alpha_i\mid k\alpha_\aga i\}$ encodes the uncertainty of agent $\aga$ about the action $\alpha=\alpha_k$ that is actually taking place. Perhaps the best known epistemic actions are {\em public announcements}, formalized as action structures $\alpha$ such that $K = \{k\}$, and $\alpha_\aga = \{(k, k)\}$ for all $\aga\in\mathsf{Ag}$. The logic of public announcements (PAL) \cite{Plaza} can then be subsumed as the fragment of EAK restricted to action structures of the form described above.
The connectives $\top$, $\bot$, $\wedge$, $\rightarrow$ and $\leftrightarrow$ are defined as usual.

Standard models for EAK are relational structures $M = (W, (R_\aga)_{\aga\in \mathsf{Ag}}, V)$ such that $W$ is a nonempty set, $R_\aga\subseteq W\times W$ for each $\aga\in \mathsf{Ag}$, and $V:\mathsf{AtProp}\to \mathcal{P}(W)$. The interpretation of the static fragment of the language is standard. For every Kripke frame $\mathcal{F} = (W, (R_\aga)_{\aga\in\mathsf{Ag}})$ and each  action structure  $\alpha$, let the Kripke frame $\coprod_{\alpha}\mathcal{F} : = (\coprod_{K}W, ((R\times \alpha)_\aga)_{\aga\in \mathsf{Ag}})$ be defined
%\footnote{This definition is of course intended to be applied to relations $\alpha$ which are part of the specification of some action structure; in these cases, the symbol $\alpha$ in $\coprod_{\alpha}\mathcal{F}$ will be  understood as the action structure. This is why the abuse of notation turns out to be useful.}
as follows: $\coprod_{K}W$ is the $|K|$-fold coproduct of $W$ (which is set-isomorphic to $W\times K$), and $(R\times \alpha)_\aga$ is a   binary relation on $\coprod_{K}W$ defined as $$(w, i)(R\times \alpha)_\aga (u, j)\quad \mbox{ iff }\quad w R_\aga u\ \mbox{ and }\ i\alpha_\aga j.$$

 For every model $M$ and each action structure $\alpha$, let $$\coprod_{\alpha}M := (\coprod_{\alpha}\mathcal{F}, \coprod_{K}V )$$ be such that $\coprod_{\alpha}\mathcal{F}$ is defined as above, and $(\coprod_{K}V)(p): = \coprod_{K}V(p)$ for every $p\in \mathsf{AtProp}$. Finally,  let the {\em update} of $M$ with the action structure $\alpha$ be the submodel $M^\alpha: = (W^\alpha, (R^\alpha_\aga)_{\aga\in \mathsf{Ag}}, V^\alpha)$ of $\coprod_{\alpha}M$ the domain of which is the subset $$W^\alpha: = \{(w, j)\in \coprod_{K}W\mid M, w\Vdash Pre_\alpha(j)\}.$$
 Given this preliminary definition, formulas of the form $\langle\alpha\rangle A$ are interpreted as follows:
$$M, w\Vdash \langle\alpha \rangle A\quad \mbox{ iff } \quad M, w\Vdash  Pre_\alpha(k)  \mbox{ and } M^\alpha, (w, k)\Vdash A.$$

The model $M^\alpha$ is intended to encode the (factual and epistemic) state of affairs after the execution of the action $\alpha$. Summing up, the construction of $M^\alpha$ is done in two stages: in the first stage, as many copies of the original model $M$ are taken as there are `epistemic potential appearances' of the given action (encoded by the action states in the domain of $\alpha$); in the second stage, states in the copies are removed if their associated original state does not satisfy the preconditions of their paired action-state.
%where $M^\alpha = (W^\alpha, R^\alpha, V^\alpha)$ is defined as follows: the underlying frame of $M^\alpha$ is the underlying frame of $M$ relativized to $\val{\alpha}_M$, i.e.\ $W^\alpha := \val{\alpha}_M$, and $R^{\alpha} := R\cap (W^\alpha\times W^\alpha)$; for every $p\in \mathsf{AtProp}$, $V^\alpha(p) = V(p)\cap W^\alpha$.

%\begin{proposition}[{\cite[Theorem 3.5]{BMS}}]

A complete axiomatization of EAK consists of copies of the axioms and rules of the minimal normal modal logic K  for each modal operator, either epistemic or dynamic, plus the following (interaction) axioms:
\begin{eqnarray}
\langle\alpha\rangle p & \leftrightarrow & (Pre(\alpha)\wedge p); \label{eq:facts}\\
\langle \alpha\rangle \neg A & \leftrightarrow & (Pre(\alpha)\wedge \neg\langle \alpha\rangle A ); \label{eq:neg}\\
\langle \alpha\rangle (A\vee B) & \leftrightarrow & (\langle \alpha\rangle A\vee \langle \alpha\rangle B ); \label{eq:vee}\\
\langle \alpha\rangle \langle\aga\rangle A & \leftrightarrow & (Pre(\alpha)\wedge \bigvee\{\langle\aga\rangle\langle\alpha_i\rangle A\mid k\alpha_\aga i\}). \label{eq:interact axiom}
\end{eqnarray}
%\noindent where the notation $\alpha_i$ has been introduced and discussed above.
 %= (K, i, \alpha, Pre_\alpha)$ for every action structure $\alpha = (K, k, \alpha, Pre_\alpha)$ and every $i\in K$ (see discussion above).
%\end{proposition}

The interaction axioms above can be understood as attempts at defining the meaning of any given dynamic modality $\langle\alpha\rangle$ in terms of its interaction with the other connectives. In particular, while axioms \eqref{eq:neg} and \eqref{eq:vee} occur also in other dynamic logics such as PDL, axioms \eqref{eq:facts} and \eqref{eq:interact axiom} capture the specific behaviour of epistemic actions. Specifically, axiom \eqref{eq:facts} encodes the fact that epistemic actions do not change the factual state of affairs, and axiom \eqref{eq:interact axiom} plausibly rephrases the fact that `after the execution of $\alpha$, agent $\aga$ thinks that $A$ might be the case' in terms of `there being some epistemic appearance of $\alpha$ to $\aga$ such that $\aga$ thinks that, after its execution, $A$ is the case'.
An interesting aspect of these axioms is that they work as rewriting rules which can be iteratively used to transform any EAK-formula into an equivalent one free of dynamic modalities. Hence, the completeness of EAK follows from the completeness of its static fragment, and EAK is not more expressive than its static fragment. However, and interestingly, there is an exponential gap in succinctness between equivalent formulas in the two languages \cite{Luz}.

Action structures are one among many possible ways to represent actions. Following \cite{GKPLori},
%, and for the sake of both the interpretation in the final coalgebra and a smooth and general proof-theoretic treatment,
we prefer to keep a black-box perspective on actions, and to identify agents $\aga$ with the indistinguishability relation they induce on actions; so, in the remainder of the article, the role of the action-structures $\alpha_i$ for $k\alpha i$ will be played by actions $\beta$ such that $\alpha\aga\beta$, allowing us to reformulate \eqref{eq:interact axiom} as
$$\langle \alpha\rangle \langle\aga\rangle A \  \leftrightarrow \  (Pre(\alpha)\wedge \bigvee\{\langle\aga\rangle\langle\beta\rangle A \mid \alpha\aga\beta\}).$$

\subsection{The intuitionistic version of EAK}
\label{ssec:intEAK}
In \cite{AMM, KP}, an analysis of PAL and EAK has been given from the point of view of algebraic semantics, resulting in the definition of the intuitionistic counterparts of PAL and EAK. In the present subsection, we briefly review the definition of the latter one, as it reveals a more subtle interaction between the various modalities, thus preparing the ground for the even richer picture that will arise from the proof-theoretic analysis.

Let \textsf{AtProp} be a countable set of atomic propositions, and let $\mathsf{Ag}$ be a nonempty set (of agents). The set $\mathcal{L}$(m-IK) of the formulas $A$ of the multi-modal version m-IK of Fischer Servi's intuitionistic modal logic IK are inductively defined as follows:
\[A := p\in \mathsf{AtProp} \mid \bot \mid A\vee A \mid A\wedge A\mid A\rightarrow A \mid \langle\aga\rangle A \mid [\aga] A\  \ \  (\aga\in \mathsf{Ag})\]

The logic  m-IK is the smallest set of formulas in the language $\mathcal{L}$(m-IK) (where $\neg A$ abbreviates as usual $A \rightarrow \bot$) containing the following axioms and closed under modus ponens and necessitation rules:
\begin{center}
\begin{tabular}{rl}
\mc{2}{c}{\textbf{Axioms}}                                                                                                                 \\
%\cline{2-2}
       & \\
$\phantom{FS2}$  & $A\rightarrow (B\rightarrow A)$                                                                                              \\
       & $(A\rightarrow (B\rightarrow C))\rightarrow ((A\rightarrow B)\rightarrow (A\rightarrow C))$  \\
       & $A\rightarrow (B\rightarrow A\pand B)$                                                                                  \\
       & $A\pand B \rightarrow A$                                                                                                        \\
       & $A\pand B \rightarrow B$                                                                                                        \\
       & $A\rightarrow A \vee B$                                                                                                           \\
       & $B\rightarrow A \vee B$                                                                                                           \\
       & $(A\rightarrow C)\rightarrow ((B\rightarrow C)\rightarrow (A\vee B\rightarrow C))$                \\
       & $\bot\rightarrow A$                                                                                                                   \\
\end{tabular}
\end{center}

\begin{center}
\begin{tabular}{rl}
%       & \\
        & $[\aga] (A\rightarrow B)\rightarrow ([\aga] A\rightarrow [\aga] B)$                                          \\
        & $\langle\aga\rangle (A\vee B)\rightarrow \langle\aga\rangle A\vee \langle\aga\rangle B$      \\
        & $\neg\langle\aga\rangle \bot$                                                                                                  \\
FS1 & $\langle\aga\rangle (A\rightarrow B)\rightarrow ([\aga]  A\rightarrow \langle\aga\rangle B)$  \\
FS2 & $(\langle\aga\rangle A\rightarrow [\aga] B)\rightarrow [\aga]  (A\rightarrow B)$                      \\
       & $\phantom{(A\rightarrow (B\rightarrow C))\rightarrow ((A\rightarrow B)\rightarrow (A\rightarrow C))}$  \\
%\cline{2-2}
%\mc{2}{c}{}                                                                                                                                            \\
\mc{2}{c}{\textbf{Inference Rules}}                                                                                                       \\
%\cline{2-2}
       & \\
MP  & if $\vdash A\rightarrow B$ and $\vdash A$, then $\vdash B$                                                    \\
%\cline{2-2}
Nec & if $\vdash A$, then $\vdash [\aga]  A$                                                                                       \\
%\cline{2-2}
\end{tabular}
\end{center}

\noindent To define the language of the intuitionistic counterpart of EAK, let \textsf{AtProp} be a countable set of atomic propositions, and let $\mathsf{Ag}$ be a nonempty set. The set $\mathcal{L}$(IEAK) of the formulas $A$ of  the intuitionistic logic of epistemic actions and knowledge (IEAK), and the set $\mathsf{Act}(\mathcal{L})$ of the {\em action structures} $\alpha$ {\em over} $\mathcal{L}$ are defined simultaneously as follows:
\begin{center}
$A := p\in \mathsf{AtProp} \mid \bot\mid A\rightarrow A \mid A\vee A \mid A\wedge A\mid   \langle\aga\rangle A\mid [\aga] A\mid \langle\alpha\rangle A\mid [\alpha] A,$
\end{center}
where $\aga\in \mathsf{Ag}$, and an {\em action structure} $\alpha$ {\em over} $\mathcal{L}$(IEAK) is defined in just the same way as action structures in section \ref{ssec:EAK}. Then, the logic IEAK is defined in a Hilbert-style presentation which includes the axioms and rules of m-IK plus the Fischer Servi axioms FS1 and FS2 for each dynamic modal operator, plus the following axioms  and rules:

\begin{center}
\begin{tabular}{rl}
\mc{2}{c}{\textbf{Interaction Axioms}}                                                                                                                  \\
%\cline{2-2}
     & \\
     & $\lc\alpha\rc p \leftrightarrow Pre(\alpha) \pand p$                                                                                         \\
     & $\ls\alpha\rs p  \leftrightarrow Pre(\alpha) \rightarrow  p$                                                                               \\
%\cline{2-2}
     & \\
     & $\lc\alpha\rc \bot \leftrightarrow \bot$                                                                                                              \\
     & $\lc\alpha\rc \top \leftrightarrow Pre(\alpha)$                                                                                                  \\
     & $\ls\alpha\rs \top \leftrightarrow  \top$                                                                                                             \\
     & $\ls\alpha\rs \bot \leftrightarrow \neg Pre(\alpha)$                                                                                          \\
$\phantom{Nec}$     & $\phantom{\lc\alpha\rc (A \rightarrow B) \leftrightarrow Pre(\alpha) \pand (\lc\alpha\rc A \rightarrow \lc\alpha\rc B)}$     \\
\end{tabular}

\begin{tabular}{rl}
%\cline{2-2}
     & $\ls\alpha\rs (A\pand B) \leftrightarrow \ls\alpha\rs A \pand \ls\alpha\rs B$                                                     \\
     & $\lc\alpha\rc (A\pand B)\leftrightarrow  \lc\alpha\rc A \pand \lc\alpha\rc B$                                                     \\
     & \\
     & $\lc\alpha\rc (A\vee B) \leftrightarrow \lc\alpha\rc A \vee \lc\alpha\rc B$                                                        \\
     & $\ls\alpha\rs (A\vee B) \leftrightarrow Pre(\alpha) \rightarrow  (\lc\alpha\rc A \vee \lc\alpha\rc B)$                \\
%\cline{2-2}
     & \\
     & $\lc\alpha\rc (A \rightarrow B) \leftrightarrow Pre(\alpha) \pand (\lc\alpha\rc A \rightarrow \lc\alpha\rc B)$     \\
     & $\ls\alpha\rs (A \rightarrow B)\leftrightarrow \lc\alpha\rc A \rightarrow \lc\alpha\rc B$                                     \\
$\phantom{Nec}$     & $\phantom{\lc\alpha\rc (A \rightarrow B) \leftrightarrow Pre(\alpha) \pand (\lc\alpha\rc A \rightarrow \lc\alpha\rc B)}$     \\
\end{tabular}

\begin{tabular}{rl}
     & $\lc\alpha\rc \lc\aga\rc A \LRA Pre(\alpha)\pand \bigvee\{\lc\aga\rc\lc\beta\rc A \mid \alpha\aga\beta\}$       \\
     & $\ls\alpha\rs\lc\aga\rc A \LRA Pre(\alpha)\rightarrow \bigvee\{\lc\aga\rc\lc\beta\rc A \mid \alpha\aga\beta\}$ \\
%\cline{2-2}
     & \\
     & $\ls\alpha\rs[\aga] A \LRA Pre(\alpha)\rightarrow  \bigwedge\{[\aga]\ls\beta\rs A \mid \alpha\aga\beta\}$      \\
     & $\lc\alpha\rc[\aga] A \LRA Pre(\alpha)\pand \bigwedge\{[\aga]\ls\beta\rs A \mid \alpha\aga\beta\}$              \\
$\phantom{Nec}$     & $\phantom{\lc\alpha\rc (A \rightarrow B) \leftrightarrow Pre(\alpha) \pand (\lc\alpha\rc A \rightarrow \lc\alpha\rc B)}$     \\
\end{tabular}

\begin{tabular}{rl}
\mc{2}{c}{}                                                                                                                                                             \\
\mc{2}{c}{\textbf{Inference Rules}}                                                                                                                        \\
$\phantom{Nec}$     & $\phantom{\lc\alpha\rc (A \rightarrow B) \leftrightarrow Pre(\alpha) \pand (\lc\alpha\rc A \rightarrow \lc\alpha\rc B)}$     \\
Nec & if $\vdash A$, then $\vdash \ls\alpha\rs A$                                                                                               \\
%Nec & if $\vdash A$, then $\vdash \ls\alpha\rs A$                                                                                               \\
%\cline{2-2}
\end{tabular}
\end{center}

%{\theorem{Axiomatic system for EAK}}

\subsection{Proof theoretic formalisms for PAL and DEL}\label{Bad_proof_systems}

In the present subsection, we discuss the most relevant existing proof-theoretic accounts \cite{Balbiani,NegriM1,NegriM2,Alexandru,Dyckhoff,Aucher1,Aucher2,Aucher} for the logic of public announcements \cite{Plaza} and for the logic of epistemic actions and knowledge \cite{BMS}.

\paragraph{Labelled tableaux for PAL. } In \cite{Balbiani}, a labelled tableaux system is proposed for public announcement logic. This system is sound and complete with respect to the semantics of PAL.
%A labeled tableaux system can be considered for a proof system, because??
Moreover, the computational complexity of this tableaux system is  shown to be  optimal for satisfiability checking in the language of PAL.
The system manipulates triples, called labelled formulas, of the form $\langle \mu, n, \phi \rangle$  such that $\mu$ is a (possibly empty) list of PAL-formulas, $n$ is a natural number,  and $\phi$ is a PAL-formula. Intuitively, the tuple  $\langle \mu, n \rangle$ stands for an epistemic state of the model updated with a sequence of announcements encoded by $\mu$. To give a closer impression of this tableaux system, consider the following rule:

\[
\AxiomC{$\langle (\alpha_1,..., \alpha_k), n, \neg K_a A \rangle $}
\LeftLabel{$R\widehat{K}$}
\RightLabel{$n'$ fresh}
\UnaryInfC{$\langle \epsilon, n',\neg [\alpha_1]...  [\alpha_k] A \rangle: \langle a,n,n' \rangle$}
\DisplayProof
\]

\vspace{5px}

This rule can be read as follows: if a state $n$ does not satisfy $K_a A$ after the sequence of announcements $\alpha_1,..., \alpha_k$, then at least one of its $R_a$-successor states $n'$ in the original model, represented by the tuple $ \langle \epsilon, n' \rangle$ in the rule, must survive the updates and not satisfy $A$. Hence, $ \langle \epsilon, n' \rangle$ must satisfy  the formula $\langle \alpha_1\rangle... \langle \alpha_k\rangle\neg A$, which is classically equivalent to $\neg[\alpha_1]...[\alpha_k] A.$

Clearly, rules such as this one incorporate the relational semantics of PAL.  This is not satisfactory from the point of view of proof-theoretic semantics, since it prevents these rules from providing an independent contribution to the meaning of the logical connectives. A second issue, of a more technical nature, is that the statement of this rule is grounded on the classical interdefinability between the box-type and diamond-type modalities. This implies that if we dispense with the classical propositional base, we would need to reformulate this rule. Hence the calculus is non-modular in the sense discussed in section \ref{ssec:Wansing}.

\paragraph{Labelled sequent calculi for PAL. } In \cite{NegriM1} and \cite{NegriM2},    cut-free labelled sequent calculi for PAL are introduced with  truthful  and non-truthful announcements, respectively.
Also in this case, the statement of the rules of these calculi incorporates the relational semantics. For instance, this is illustrated here below for the case of truthful announcements.

\[
\AxiomC{$w:^{\mu, \alpha} A, w:^\mu [\alpha]A, w:^\mu \alpha, \Gamma \vdash \Delta $}
\RightLabel{$L [\ ]{:}^\mu$}
\UnaryInfC{$w:^\mu [\alpha] A, w:^\mu \alpha, \Gamma \vdash \Delta$}
\DisplayProof
\qquad
\AxiomC{$w:^\mu \alpha, \Gamma \vdash \Delta, w:^{\mu, \alpha} A$}
\RightLabel{$R[\ ]{:}^\mu$}
\UnaryInfC{$\Gamma \vdash \Delta, w:^\mu [\alpha] A $}
\DisplayProof
\]

In the rules above, symbols such as $w:^{\mu} A$ can be rearranged and then understood as the labelled formulas $\langle \mu, w, A\rangle$  in the tableaux system presented before.   The only difference is that $w$ is an individual variable which stands for a given state of a relational structure, and not for a natural number; however, this difference is completely nonessential. Under this interpretation, it is clear that e.g.~the rule  $L [\ ]{:}^\mu$ encodes the relational satisfaction clause of $[\alpha]A$, when $\alpha$ is a truthful announcement.
The following rules are also part of the calculi.
%Now, let us consider the rules for announcements that might be non-truthful.

%\[
%\AxiomC{$w:^{\mu, \alpha} A, \Gamma \vdash \Delta $}
%\RightLabel{$L [\ ]{:}^\mu$}
%\UnaryInfC{$w:^\mu [\alpha] A, \Gamma \vdash \Delta$}
%\DisplayProof
%\qquad
%\AxiomC{$\Gamma \vdash \Delta, w:^{\mu, \alpha} A$}
%\RightLabel{$R[\ ]{:}^\mu$}
%\UnaryInfC{$\Gamma \vdash \Delta, w:^\mu [\alpha] A $}
%\DisplayProof
%\]

\[
\AxiomC{$v: A, w:K_a A, wR_av, \Gamma \vdash \Delta $}
\RightLabel{$LK_a$}
\UnaryInfC{$w:K_aA, wR_av, \Gamma \vdash \Delta$}
\DisplayProof
\qquad
\AxiomC{$wR_av, \Gamma \vdash \Delta, v: A$}
\RightLabel{$RK_a$}
\UnaryInfC{$\Gamma \vdash \Delta, w: K_aA $}
\DisplayProof
\]

 Besides the individual variables $w$ and $v$, the rules above feature the binary relation symbol $R_a$ encoding the epistemic uncertainty of the agent $a$. Since  the relational semantics is imported in the definitions of the rules, the same issues  pointed out in the case of the tableaux system appear also here. On the other hand, importing the relational semantics allows for some remarkable extra power. Indeed, the interaction axiom \eqref{eq:interact axiom} can be derived from the four rules above, which deal with static and dynamic modalities in complete independence of one another.

\paragraph{Merging different logics. } In \cite{Alexandru} and \cite{Dyckhoff}, sequent calculi have been defined for dynamic logics arising in an algebraic way, motivated by program semantics, with a methodology introduced by  \cite{Abramsky}. Essentially, this approach is based on the idea of merging a linear-type logic of actions (more precisely, \cite{Moortgat95}) with a classical or intuitionistic logic of propositions. Following the treatment of \cite{Abramsky}, this logic arises semantically as the logic of certain quantale-modules, namely of maps $\star: M\times Q\rightarrow M$,  preserving complete joins in each coordinate, where $Q$ is a quantale and $M$ is a complete join-semilattice.  Each $q\in Q$ induces a completely join-preserving operation $(-\star q): M\to M$, which, by general order-theoretic facts, has a unique right adjoint $[q]: M\to M$. That is, for every $m, m'\in M$,
\begin{equation}
\label{adjunction}
m\star q\leq m' \ \mbox{ iff }\ m\leq [q] m'.
\end{equation}
 Intuitively, the elements of $Q$ are actions (or rather, inverses of actions), and $M$ is an algebra interpreting propositions, which in the best known cases arises as the complex algebra of some relational structure, and therefore will be e.g.~a complete and atomic Boolean algebra with operators. Thus the framework of \cite{Alexandru} and \cite{Dyckhoff} is vastly more general than dynamic epistemic logic as it is usually understood. A remarkable feature of this setting is that the dynamic operations which are intended as the interpretation of the primitive dynamic  connectives arise in this setting as adjoints of ``more primitive'' operations; thus, and much more importantly, every dynamic modality comes with its adjoint.  Moreover, every epistemic modality (parametrized as usual with an agent) comes in two copies: one as an operation on $Q$ and one as an operation on $M$, and these two copies are stipulated to interact in a suitable way. More formally, the semantic structures are defined as tuples $(M, Q, \{ f_A\} _{A \in Ag})$, where $M$ and $Q$ are as above, and for every agent $A$, $f_A$ is a pair of completely join preserving maps $(f^{M}_A: M \to M, f^Q_A: Q \to Q)$ such that the following three conditions hold:

\begin{equation}
f^Q_A(q \cdot q' )\leq f^Q_A(q) \cdot  f^Q_A(q' )
\end{equation}

\begin{equation}\label{Interaction}
f^M_A(m \star q)\leq f^M_A(m) \star  f^Q_A(q )
\end{equation}

\begin{equation}
1 \leq f^Q_A(1).
\end{equation}

Intuitively, for every agent $A$, the operation  $f^M_A$ is the diamond-type  modal operator encoding the epistemic uncertainty of $A$, and  $f^Q_A$  is the  diamond-type modal operator encoding the epistemic uncertainty of $A$ about the action that is actually taking place. Given this understanding, condition (\ref{Interaction}) hardcodes  the following well-known DEL-axiom  in the semantic structures above:

\begin{equation}\label{axiomDEL}
\bigwedge\{ [A][ q'] m\mid q A q'\} \vdash [ q ] [A] m.
\end{equation}

\noindent where the notation $q Aq'$ means that the   action $q'$ is indistinguishable from $q$ for the agent $A$. In (\ref{Interaction}), the element $f^Q_A(q)$ encodes the join of all such actions. Because $\star$ is bilinear, we get:

$$f^M_A(m)\star f^Q_A(q) = f^M_A(m)\star \bigvee_Q\{q'\mid qAq'\} = \bigvee_M\{f^M_A(m)\star q'\mid qAq'\}. $$
Hence, (\ref{Interaction}) can be equivalently rewritten in the form of a rule as follows:
\[
\AX $\bigvee\{f^M_A(m)\star q'\mid qAq'\} \fCenter m'$
\UI $ f^M_A(m\star q)\fCenter m'$
\DisplayProof
\]
Applying adjunction to the premise and to the conclusion gets us to:
\[
\AX $m\fCenter \bigwedge\{ [A][q']m' \mid qAq'\} $
\UI $ m\fCenter [q][A]m'$
\DisplayProof
\]
Finally, rewriting the rule above back as an inequality gets us to (\ref{axiomDEL}).
The  first pioneering proposal is the sequent calculus developed in \cite{Alexandru}. This calculus manipulates two kinds of sequents: Q-sequents, of the form $\Gamma \vdash_Q q$, where $q$ is an action and $\Gamma$ is a sequence of actions and agents, and M-sequents, of the form $\Gamma \vdash_M m$, where $m$ is a proposition and $\Gamma$ is a sequence of propositions, actions and agents. These different entailment relations need to be brought together by means of rules of hybrid type, such as the left one below.

\[
\AxiomC{$m' \vdash_M m  $}
\AxiomC{$\Gamma_Q \vdash_Q q$}
\RightLabel{$Dy$L}
\BinaryInfC{$[q] m', \Gamma_Q \vdash_M m$}
\DisplayProof
\qquad
\AxiomC{$\Gamma, q \vdash_M m$}
\RightLabel{$Dy$R}
\UnaryInfC{$\Gamma \vdash_M [q] m$}
\DisplayProof
\]

As to the soundness of the rule $DyL$, let us identify the logical symbols with their interpretation, assume that the inequalities $m\leq m'$ and $\Gamma_{Q}\leq q$ are satisfied on  given $M$ and $Q$ respectively,\footnote{where $\Gamma_Q$ now stands for a suitable product in $Q$ of the interpretations of its individual components.} and prove that $[q] m', \Gamma_Q \leq m$ in $M$. Indeed, $$[q]m'\star \Gamma_{Q}\leq [q]m'\star q\leq m'\leq m.$$

The first inequality follows from $\Gamma_{Q}\leq q$ and $\star$ being order-preserving in its second coordinate; the second inequality is obtained by applying the right-to-left direction of (\ref{adjunction}) to the inequality $[q]m'\leq [q]m'$; the last inequality holds by assumption. The soundness of $Dy$R follows likewise from the left-to-right direction of (\ref{adjunction}).

This calculus is shown to be both sound and complete w.r.t.~this algebraic semantics. The setting illustrated above is powerful enough that sufficiently many epistemic actions can be encoded in it to support the formalisation of various variants of the Muddy Children Puzzle in which children might be cheating. However, cut-elimination for this system has not been proven.

In \cite{Dyckhoff}, a similar framework is presented  which exploits the same basic ideas, and results in a system with more explicit proof-theoretic performances and which is shown to be cut-free.  However, like its previous version, this system focuses on a logic semantically arising from an algebraic setting which is vastly more general than the usual relational setting. The issue about how it precisely restricts to the usual setting, and hence how the usual DEL-type logics can be captured within this more general  calculus, is left largely implicit.  The semantic setting of \cite{Alexandru}, where propositions are interpreted as elements of a right module $M$ on a quantale $Q$,  specialises in \cite{Dyckhoff} to a setting in which $M = (\mathbb{A}, \{\Box_A, \blacklozenge_A: A\in Ag\})$, where $\mathbb{A}$ is a Heyting algebra and, for every agent $A$, the modalities $\Box_A$  and $\Bdia_A$ are adjoint to each other. Notice that $\Diamond_A$, which in the classical case is defined as $\neg \Box_A \neg$,  cannot be expressed any more in this way, and needs to be added as a primitive connective, which has not been done in  \cite{Dyckhoff}.

As mentioned before, the design of this calculus gives a more explicit account than its previous version to certain technical aspects which come from the semantic setting; for instance, the semantic setting motivating both papers features two domains of interpretation (one for the actions and one for the propositions), which are intended to give rise to two consequence relations  which are to be treated on a par and then made to interact. In \cite{Alexandru}, the calculus manipulates sequents which are made of heterogeneous components. For instance, in action-sequents $\Gamma\vdash_Q q$, the precedent $\Gamma$ is a sequence in which both actions and agents may occur. Since $\Gamma$ is to be semantically interpreted as an element of $Q$, they need to resort to a rather clumsy technical solution which consists in interpreting, e.g.\ the sequence $(q, A, q')$ as the element $f_A^Q(q)\cdot q'$.   In  \cite{Dyckhoff}, the calculus is given in a {\em deep-inference} format; namely, rules of this calculus make it possible to manipulate formulas inside a given context. This more explicit bookkeeping makes it possible to prove the cut-elimination, following the original Gentzen strategy. However, the presence of two different consequence relations and the need to account for their interaction calls for the development of an extensive theory-of-contexts, in which no less than five different types of contexts need to be introduced. This also causes a proliferation of rules, since the possibility of performing some inferences depends on the type of  context under which they are to be performed.
%In \ref{}, we will expand on the similarities and differences between this calculus and the display-type sequent calculi.

\paragraph{Calculi for updates. }In \cite{Aucher1}, a formal framework accounting for dynamic revisions or updates is introduced, in which the revisions/updates are formalized using the turnstile symbol. This framework has aspects similar to Hoare logic: indeed, it manipulates sequent-type structures of the form $\phi, \phi'\models \phi''$, such that $\phi$ and $\phi''$ are formulas of proposition-type, and $\phi'$ is a formula of event-type. This formalism has also common aspects to \cite{Alexandru} and \cite{Dyckhoff}: indeed, both proposition-type and event-type (i.e.~action-type) formulas allow epistemic modalities for each agent, respectively accounting for the agent's epistemic uncertainty about the world and about the actions actually taking place.

In \cite{Aucher} and \cite{Aucher2}, three formal calculi are introduced,  manipulating the syntactic structures above. Given that the turnstile encodes the update rather than a consequence relation or entailment,  the syntactic structures above are not sequents in a proper sense. Rather than sequent calculi, these calculi should be rather regarded as being of natural deduction-type. As such, the design of these calculi presents many issues from a proof-theoretic semantic viewpoint; to mention only one, multiple connectives are introduced at the same time, for instance in the following rule:

\[
\AxiomC{$\phi, \phi' \vdash \phi''$}
\RightLabel{$R_5.$}
\UnaryInfC{$\langle Bj \rangle (\phi \wedge Pre(p')), \langle Bj \rangle (\phi' \wedge p') \vdash \langle Bj \rangle \phi'' $}
\DisplayProof
\]

\vspace{3px}
These calculi are shown to be sound and complete w.r.t.\ three semantic consequence relations, respectively.

\subsection{First attempt at a display calculus for EAK}\label{D.EAK}

In \cite{GKPLori}, a display-style sequent calculus D.EAK has been introduced, which is sound with respect to the final coalgebra semantics (cf.\ section \ref{sec:final coalgebra sem}), and complete w.r.t.~EAK, of which it is a conservative extension. Moreover,  Gentzen-style cut elimination  holds for D.EAK. Finally, this system is defined independently of the relational semantics of EAK, and therefore is suitable for a fine-grained  proof-theoretic semantic analysis.

\noindent Here below, we are not going to report on it in detail, but we limit ourselves to mention the structural rules which capture the specific features of EAK:

%\marginpar{\raggedright\tiny{A: make these rules scriptsized}}

\smallskip
\begin{center}
{\scriptsize{
\begin{tabular}{rl}
\mc{2}{c}{\textbf{\normalsize{Structural Rules with Side Conditions}}} \\
%\hline

 & \\

\AX$Pre(\alpha)\,; \{\alpha\} A \fCenter X$
\LeftLabel{\fns$reduce_L$}
\UI$\{\alpha\} A \fCenter X$
\DisplayProof &
\AX$X \fCenter Pre(\alpha) > \{\alpha\} A$
\RightLabel{\fns$reduce_R$}
\UI$X \fCenter \{\alpha\} A$
\DisplayProof \\

&\\

\AXC{$Pre(\alpha)\,; \{\alpha\} \{ \aga \} X \fCenter Y$}
\LeftLabel{\fns\emph{swap-in}$_{L}$}
\UIC{$Pre(\alpha)\,; {\{ \aga \}\{\beta\}_{\alpha\aga\beta}\, X} \fCenter Y$}
\DisplayProof &
\AXC{$Y \fCenter Pre(\alpha) > \{\alpha\} \{ \aga \} X$}
\RightLabel{\fns\emph{swap-in}$_{R}$}
\UIC{$Y \fCenter Pre(\alpha) > {\{ \aga \} \{\beta\}_{\alpha\aga\beta}\,X}$}
\DisplayProof \\

 & \\

\AX$\Big ( Pre(\alpha)\,; \{\aga\}\{\beta\} \,X \fCenter Y\mid \alpha\aga\beta\Big)$
\LeftLabel{\emph{swap-out}$_{L}$}
\UI$Pre(\alpha)\,;  \{\alpha\} \{\aga\} X \fCenter \Bigsemic\Big(Y\mid \alpha\aga\beta\Big)$
\DisplayProof
&
\AX$\Big(Y \fCenter Pre(\alpha) > \{\aga\}\{\beta\}\,X \mid \alpha\aga\beta\Big)$
\RightLabel{\emph{swap-out}$_{R}$}
\UI$\Bigsemic\Big(Y\mid \alpha\aga\beta\Big) \fCenter Pre(\alpha) > \{\alpha\} \{\aga\} X$
\DisplayProof\\

 & \\
%\hline
\end{tabular}
}}
\end{center}

\noindent The {\em swap-out} rules do not have a fixed arity; they have as many premises as there are actions $\beta$ such that $\alpha\aga\beta$. In the conclusion, the symbol $\Bigsemic \Big(Y \mid \alpha\aga\beta \Big)$ refers to a string $(\cdots(Y\,; Y) \,;\cdots \,; Y)$ with $n$ occurrences of $Y$, where $n = |\{\beta\mid \alpha\aga\beta\}|$.

\begin{center}
{\scriptsize{
\begin{tabular}{rl}
\mc{2}{c}{\textbf{\normalsize{Operational Rules with Side Conditions}}} \\
%\hline

 & \\

\AX$Pre(\alpha)\,; \{\alpha\} A \fCenter X$
\LeftLabel{\fns$reverse_L$}
\UI$Pre(\alpha)\,; \ls\alpha\rs A \fCenter X$
\DisplayProof
 &
\AX$X \fCenter Pre(\alpha) > \{\alpha\} A$
\RightLabel{\fns$reverse_R$}
\UI$X \fCenter Pre(\alpha) > \lc\alpha\rc A $
\DisplayProof\\

 & \\
%\hline
\end{tabular}
}}
\end{center}
\smallskip

The main issues of D.EAK from the point of view of Wansing's criteria are linked with the presence of the formula $Pre(\alpha)$: namely, the \emph{swap-in} and \emph{swap-out} rules violate the principle that all parametric variables should occur unrestricted. Indeed, the occurrences of the formula $Pre(\alpha)$ in these rules is easily seen to be parametric, since $Pre(\alpha)$ occurs both in the premises and in the conclusion. Since $Pre(\alpha)$ is (the metalinguistic abbreviation of) a formula, it  is a structure of a very restricted shape. As to the \emph{swap-out} rules, it is not difficult to see, e.g.~semantically (cf.~\cite[Definition 4.2.]{KP}), that the occurrences of $Pre(\alpha)$ can be removed both in the premises and in the conclusion without affecting either  the soundness of the rule or the proof power of the system; this entirely remedies the problem. Likewise, as to \emph{swap-in}, it is not difficult to see that the occurrences of $Pre(\alpha)$ can be removed
in the premises, but not in the conclusion. However, even modified in this way, the \emph{swap-in} rules would not be satisfactory. Indeed, the new form of \emph{swap-in} would introduce $Pre(\alpha)$  in the conclusion. Since $Pre(\alpha)$ is a metalinguistic abbreviation of a formula which as such has no other specific restrictions, the occurrence of $Pre(\alpha)$ in the conclusion of \emph{swap-in} must also be regarded as parametric. However, we still would not be able to substitute arbitrary structures for it, which is the source of the problem.
%   \emph{swap-in}  would behave similarly to a \emph{weakening} rule.According to  Belnap's understanding,  \emph{weakening} is a rule in which all structures are parametric, including the one introduced in the conclusion. By analogy then, the occurrence of $Pre(\alpha)$ in the conclusion of \emph{swap-in} should again be counted as parametric, and hence, as before, it is problematic.
This problem would be solved if $Pre(\alpha)$ could be expressed, as a structure, purely in terms of the parameter $\alpha$ and structural constants (but no structural variables). If this was the case, \emph{swap-in} would encode the relations between all these logical constants, and all the occurring structural variables would be unrestricted.

Secondly, the rules \emph{reduce} violate condition C$_1$: indeed, in each of them, a formula in the premisses, namely $Pre(\alpha)$, is not a subformula of any formula occurring in the conclusion. Together with the cut-elimination, condition C$_1$ guarantees the subformula property (cf.~\cite[Theorem 4.3]{Belnap}), but is not itself essential  for the cut-elimination, and indeed,  cut-elimination has been proven for D.EAK (albeit not {\em \`a la} Belnap). The specific way in which \emph{reduce}  violates C$_1$ is also not a very serious one. Indeed, if the formula $Pre(\alpha)$ could be expressed in a structural way, this violation would disappear.

This solution cannot be implemented in D.EAK because the language of D.EAK does not have enough expressivity to talk about $Pre(\alpha)$ in any other way than as an arbitrary formula, which needs to be introduced via weakening or via identity (if atomic). Being able to account for $Pre(\alpha)$ in a satisfactory way from a proof-theoretic semantic perspective would require being able to state rules which, for any $\alpha$, would introduce $Pre(\alpha)$ {\em specifically}, thus capturing its proof-theoretic meaning. Thus, by having structural  and operational rules for $Pre(\alpha)$, we would solve many problems in one stroke: on the one hand, we would gain the {\em practical} advantage of achieving the satisfaction of C$_1$, thus guaranteeing the subformula property; on the other hand,  and  more importantly, from a {\em methodological} perspective, we would be able to have a setting in which the occurrences of $Pre(\alpha)$ are not to be regarded as side formulas, but rather, they would occur as {\em structures}, on a par with all the other structures they would be interacting with.

Finally, the only operational rules violating Wansing's  \emph{separation} principle (cf.\ subsection \ref{ssec:Wansing}) are the  \emph{reverse} rules: %\marginpar{\raggedright\tiny{A: remove to save space? expand on why they violate separation}}

\smallskip
\begin{center}
\begin{tabular}{c}
\AxiomC{$ Pre(\alpha);  \{\alpha\} A\vdash X$}
\LeftLabel{$rev_{L}$}
\UnaryInfC{$ Pre (\alpha); [\alpha] A \vdash X$}
\DisplayProof
\qquad
\qquad
\AxiomC{$ X \vdash  Pre (\alpha) > \{ \alpha \} A$}
\RightLabel{$rev_{R}$}
\UnaryInfC{$ X \vdash  Pre (\alpha) > \langle \alpha \rangle A$}
\DisplayProof
\end{tabular}
\end{center}
\smallskip
Here again, the problem comes from the fact that the language is not expressive enough to capture the principles encoded in the rules above at a purely structural level. In this operational formulation, these rules are to participate, in our view improperly, in the proof-theoretic meaning of the connectives $[\alpha]$ and $\langle \alpha\rangle$. Thus, it would be desirable that the rules above could be either derived, so that they disappear altogether, or alternatively, be reformulated as structural rules.

\section{Final coalgebra semantics of dynamic logics}
\label{sec:final coalgebra sem}

In order to provide a  justification for the soundness of the display postulates involving the dynamic connectives, in \cite{GKPLori} the final coalgebra was used as a semantic environment for the calculus D.EAK. Specifically, the final coalgebra was there used to show that D.EAK is sound, and conservatively extends EAK. %are shown using the final coalgebra was proved to be sound and complete w.r.t.\ EAK, of which it was shown to be a conservative extension. Before moving on and introducing the revised version of D.EAK,
In the present section, we briefly review the needed preliminaries on the final coalgebra, and then  the interpretation of EAK-formulas in the final coalgebra, which we will use in section \ref{ssec:properties of D'.EAK} to show that D'.EAK is sound, and conservatively extends  EAK.\footnote{This semantics specifically applies to the classical base. Analogous ideas can be developed for weaker propositional bases, but in the present paper we do not pursue them further.}

%Modal formulas $A$ are interpreted in Kripke models $\Mbb=(W,R,V)$ as subsets of their domains $W$, and we write $\sem{A}_\Mbb\subseteq W$ for their interpretation. Equivalently, we can describe the interpretation of $A$ in  the \emph{final coalgebra}.

%Some basic properties of the final coalgebra are recalled in the flowing subsection, before turning to its use as a semantics of dynamic logics, that is, logics which have modalities expressing a change of Kripke model. In particular we will see that the final coalgebra semantics allows us to use the standard relational semantics also for dynamic modalities \emph{and their adjoints}.

\subsection{The final coalgebra}
The general notion of a coalgebra, as an arrow
$$W\to FW$$
is given w.r.t.\ a functor $F:\cal C\to\cal C$ on an arbitrary category $\cal C$, and much of the theory of coalgebras is devoted to establishing results on coalgebras parametric in that
%`base' category $\mathcal C$ and the
functor $F$. For example, important notions such as bisimilarity and Hennessy-Milner logics can be given for arbitrary functors on the category of sets (and many other concrete categories).
But even if one is interested, as in our case here, only in one particular functor, the notion of a final coalgebra is of value, as we are going to see.

\medskip\noindent
Aczel \cite{aczel:nwfs} observed that coalgebras
$$W\to\cal P W$$
for the powerset functor $\cal P$ (which maps a set $W$ to the set $\cal PW$ of subsets of $W$)
are exactly Kripke frames. Indeed, a map $W\to\cal PW$ equivalently encodes a binary relation $R$ on $W$. More importantly, the category theoretic notion of a coalgebra morphism coincides with the notion of bounded (or p-) morphism in modal logic, and the coalgebraic notion of bisimulation coincides with the notion in modal logic. This observation generalises easily to Kripke models over a set $\textsf{AtProp}$ of atomic propositions and with multiple relations indexed by a set of agents $\mathsf{Ag}$, which are exactly coalgebras
$$W\to (\cal PW)^\mathsf{Ag}\times 2^\textsf{AtProp}.$$
As shown by \cite{AczelM}, one can construct a `universal model' $\Zbb$ by taking the disjoint union of all coalgebras $\Mbb$ and quotienting by bisimilarity. This coalgebra $\Zbb$ is \emph{final}, that is, for any coalgebra $\Mbb$ there is exactly one morphism $\Mbb\to\Zbb$. The property of finality characterises $\Zbb$ up to isomorphism.

\paragraph{$\Zbb$ may be a proper class. }  In \cite{AczelM},  any functor $F$ on sets is extended to classes and it is shown that the extended functor always has a final coalgebra, constructed as the bisimilarity collapse of the disjoint union of all coalgebras. In \cite{barr:terminal},  the same construction is recast in terms of an inaccessible cardinal, staying inside the set-theoretic universe without using classes. In \cite{AdamekMV:general}  these results  are generalized from sets to other similar categories such as posets, and in \cite{AdamekMV:classes}, it is shown that any functor $F$ on classes is the extension of a functor $F$ on sets.

\paragraph{$\Zbb$ classifies bisimilarity. } The importance of the theorems above is not merely the existence of the final coalgebra. Since all of these theorems involve two functors, one on `large' sets extending another one on `small' sets, and since one is interested in the notion of bisimilarity associated with the small functor, the existence of a final coalgebra for the large functor is not in itself the result one is interested in. But it is a fact, expressed for example as the small subcoalgebra lemma in \cite{AczelM}, that in all of the constructions above, the final coalgebra for the large functor classifies  the notion of bisimilarity associated with the small functor. In other words, passing from small to large does not extend---up to bisimilarity---the range of available models.

\paragraph{Frame conditions on  $\Zbb$. }  Often, one is interested in Kripke models satisfying additional frame conditions such as reflexivity, transitivity, equivalence, etc. A sufficient condition for the existence of a final coalgebra under such additional conditions is that these conditions can be formulated by modal axioms or rules, see \cite{Kurz02,Kurz01} for details.

\subsection{Final coalgebra semantics of modal logic}   Summing up the discussion in the previous  subsection, there is a one-to-one correspondence between subsets of the final coalgebra and unary predicates invariant under bisimilarity. Therefore, whenever we know that $A$ is a formula invariant under bisimilarity, we may declare the subset $\sem{A}_\Zbb=\{z\in\Zbb\mid \Zbb,z\Vdash A\}$ of the final coalgebra as the (final) semantics of $A$ and recover $\sem{A}_M\subseteq W$ as
\begin{equation}\label{eq:semZ}
\sem{A}_\Mbb=f^{-1}(\sem{A}_\Zbb),
\end{equation}
 where $f$ is the unique homomorphism $$f: \Mbb\to\Zbb$$ provided by the property of $\Zbb$ being final. Let us note that this approach is quite general: it only needs a notion of bisimilarity tied to the morphisms  of some category (see \cite{KurzR05} for a general definition) and a notion of modal formula whose semantics is invariant under this notion of bisimilarity.

\paragraph{Final coalgebra semantics of dynamic modalities. } Dynamic logics add to Kripke semantics a facility for updating the Kripke model interpreting a formula. Typically,  despite seemingly increasing the expressiveness of modal logic, such dynamic logics also enjoy bisimulation invariance and can therefore be interpreted in the final coalgebra.

\medskip\noindent Whereas the Kripke semantics of an action $\alpha$ is a relation between pointed models, the final coalgebra semantics of an action $\alpha$ is simply a relation on the carrier $Z$ of the final coalgebra $\Zbb$. The precise relationship between Kripke semantics and final coalgebra semantics of actions is as follows. Let us write  $$z\,\alpha_\Zbb\,z'$$ to express that the two points $z,z'$ of the final coalgebra are related by $\alpha$, formalising that in $z$ the action $\alpha$ can happen and has $z'$ as a successor. Then $z\,\alpha_\Zbb\,z'$ iff there are pointed models $(\Mbb,w)$ and $(\Mbb',w')$ related by the action $\alpha$ such that  the unique morphisms $\Mbb\to\Zbb$ and $\Mbb'\to\Zbb$ map $w$ to $z$ and $w'$ to $z'$.

\paragraph*{Specific desiderata for epistemic actions.} 
\label{par : Specific desiderata for epistemic actions}
The specific feature of epistemic actions versus arbitrary actions 
is that epistemic actions do not change the factual states of affairs. 
Semantically, this motivates the additional requirement that 
if $\alpha_Z \subseteq Z\times Z$ is the interpretation of an epistemic action $\alpha$ and   $z,z' \in Z$ are such that  $z \alpha_Z z'$,
then 
$$ \{ p\in AtProp \mid z \Vdash p \}  \quad = \quad  \{ p\in AtProp \mid z' \Vdash p \}.$$

\paragraph{Adjoints of dynamic modalities. }  To semantically justify the full display property of  display calculi for dynamic logics,  adjoints need to be available not only for the standard modalities, but also for the dynamic ones. Now, it is well known that modalities induced by a relation come in adjoint pairs. Let us recall

\begin{proposition}
\label{prop:finalsemantics}
  Every relation $R\subseteq X\times Y$ gives rise to the  modal
  operators $$\fDiam{R},\fBox{R}: PY\to PX \ \textrm{and } \
  \pDiam{R},\pBox{R}:PX\to PY$$
defined as follows: for every $V\subseteq X$ and every $U\subseteq Y$,
%
%\begin{align}
\begin{center}
\begin{tabular}{c c c}
%\label{eq:fDiam}
$\fDiam{R} U = \{x\in X\mid \exists y\;.\; x R y \ \& \ y\in U \}$ & $\quad$ &
%\label{eq:fBox}
$\fBox{R}U = \{x\in X\mid \forall y\;.\; x R y \ \Rightarrow \ y\in U \}$ \\
%\label{eq:pDiam}
$\pDiam{R}V = \{y\in Y\mid \exists x\;.\; x R y \ \& \ x\in V \} $&  $\quad$ &
%\label{eq:pBox}
$\pBox{R}V = \{y\in Y\mid \forall x\;.\; x R y \ \Rightarrow \ x\in V \}$.\\
\end{tabular}
\end{center}
%\end{align}
These operators come in adjoint pairs:
\begin{gather}\label{eq:fDiam-pBox}
\fDiam{R}U \subseteq V \ \mbox{ iff } \ U\subseteq
\pBox{R}V \\
\label{eq:pDiam-fBox}
\pDiam{R}V \subseteq U \ \mbox{ iff } \ V\subseteq
\fBox{R}U.
\end{gather}
%$\fDiam{R}$ preserves the top-element iff $R$ is total (since $\fDiam{R}\top=\dom(R)$); $\fDiam{R}$ preserves binary intersections iff $R$ is single-valued. $\fBox{R}$ preserves the empty set iff $R$ is total; $\fBox{R}$ preserves directed unions iff it is image-finite; $\fBox{R}$ preserves non-empty unions iff it is single-valued.
\end{proposition}
In order to apply this proposition to dynamic modalities, we need to consider the relation corresponding to an action $\alpha$. Kripke semantics suggests to consider $\alpha$ as a relation on all pointed Kripke models $(\Mbb,w)$, but this would introduce a two-tiered semantics: with the semantics of an ordinary modality given by a relation on the carrier of a model $\Mbb$ and the semantics of a dynamic modality given by a relation on the set of all pointed models $(\Mbb,w)$. In the final coalgebra semantics all relations are relations on the final coalgebra $\Zbb$ and we can directly apply the above proposition to both static and dynamic modalities (with the $X$ and $Y$ of the proposition being the carrier of the final coalgebra).

\paragraph{Soundness of the display postulates. }
Let us expand on how to interpret display-type structures and sequents in the final coalgebra. Structures will be translated into formulas, and formulas  will be interpreted as subsets of the final coalgebra. In order to translate structures as formulas, structural connectives need to be translated as logical connectives; to this effect,
%as mentioned in Section \ref{ssec:DisplayLogic},
structural connectives are associated with pairs of logical connectives and any given occurrence of a structural connective is translated as one or the other, according to which side of the sequent the given occurrence can be displayed on as main connective, % on and on its position within  the generation tree of the structure %{\em pairs} of logical connectives %(some of which do not belong to the original logical signature; however, the special properties of perfect distributive lattices guarantee the existence of the semantic interpretation of these additional logical connectives),
as reported in
Table \ref{pic:1}.
\begin{table}
\begin{center}
\begin{tabular}{c||c|c}
Structural                & $\quad$ if in precedent$\quad$          & $\quad$ if in succedent$\quad$\\
$\quad$ connective $\quad$             \ &\  position           \ &\  position \\[1pt]
\hline
I                  & $\top$                   & $\bot$            \\[1pt]
$A\,; B$           & $A\wedge B$              & $A \vee B$        \\[2pt]
$A > B$           \ & $A \pdra B$              & $A \rightarrow B$ \\[2pt]
$\{\aga\} A$          & $\langle\aga\rangle A$             & $[\aga] A$          \\[4pt]
$\RESagaProxy A$ &$\RESagaDia A$          & $\RESagaBox A$  \\[4pt]
$\{\alpha\} A$     & $\langle\alpha\rangle A$ & $[\alpha] A$      \\[4pt]
$\RESalphaProxy A$ &$\RESalphaDia A$          & $\RESalphaBox A$  \\
\end{tabular}
\end{center}
\caption{Translation of structural connectives into logical connectives}\label{pic:1}
\end{table}
These logical connectives in turn are interpreted in the final coalgebra in the standard way. For example,
\begin{center}
\begin{tabular}{c c}
$\sem{\langle\alpha\rangle A}_\Zbb  =\langle \alpha_\Zbb \rangle \sem{A}_\Zbb $ &
$\sem{[\alpha] A}_\Zbb  = [ \alpha_\Zbb ] \sem{A}_\Zbb$\\
$\sem{\RESalphaDia  A}_\Zbb  =\pDiam{\alpha_\Zbb} \sem{A}_\Zbb$ &
$\sem{\RESalphaBox A}_\Zbb  = \pBox{\alpha_\Zbb} \sem{A}_\Zbb$\\
\end{tabular}
\end{center}
where the notation on the right-hand sides refers to the one defined in Proposition~\ref{prop:finalsemantics}.

\medskip\noindent Sequents $A\vdash B$ will be interpreted as inclusions $\val{A}_Z\subseteq \val{B}_Z$; rules $(A_i\vdash B_i\mid i\in I)/C\vdash D$ will be interpreted as implications of the form ``if $\val{A_i}_Z\subseteq \val{B_i}_Z$ for every $i\in I$, then $\val{C}_Z\subseteq \val{D}_Z$''.
As a direct consequence of the adjunctions
\eqref{eq:fDiam-pBox} and \eqref{eq:pDiam-fBox}, the following display postulates are sound under the interpretation above.
%$\ssRESalphaProxy \atop \{\alpha\}$ and $\{\alpha\} \atop \ssRESalphaProxy$

 \begin{center}
 \begin{tabular}{c c}
 \AX$\{\alpha\} X \fCenter Y$
\RightLabel{\scriptsize${\{\alpha\} \atop \ssRESalphaProxy}$}
\doubleLine
\UI$X \fCenter \RESalphaProxy Y$
\DisplayProof

\quad\quad & \quad\quad

\AX$X \fCenter \{\alpha\} Y$
\RightLabel{\scriptsize${\ssRESalphaProxy \atop \{\alpha\}}$}
\doubleLine
\UI$\RESalphaProxy X \fCenter Y$
\DisplayProof

 \end{tabular}
 \end{center}

\medskip\noindent\textbf{Remark.} On the other hand, standard Kripke models are not in general closed under (the interpretations of) $\alpha$ and $\conv{\alpha}$.
%\begin{remark}
 As a direct consequence of this fact,  we can show that e.g.\ the display postulate $\{\alpha\} \choose \ssRESalphaProxy$ is not sound if we interpret it in a Kripke model $\Mbb$ for any interpretation of formulas of the form %$\rap$
 $\RESalphaBox B$ in $\Mbb$.
%\begin{remark}
%the  display postulates above %of D.EAK (featuring $\RESalphaProxy$ and $\{\alpha\}$ as  the structural proxies of $\RESalphaDia, \RESalphaBox$ and of $\fDiam{\alpha}, \fBox{\alpha}$ respectively):
% are not in general sound w.r.t.\ the standard semantics.
% For instance, a necessary condition for the display postulate $\{\alpha\} \atop \ssRESalphaProxy$ to be sound w.r.t.\ the standard semantics is that (after having translated the structures into logical formulas in the standard way, cf.\ Subsection \ref{ssec: soundness}) the logical connective $\rab$ can be semantically interpreted on any  standard model $M$ in such a way that, for all formulas $\phi$ and $\psi$, if $\val{\phi}_M \subseteq \val{\rab \psi}_M$, then $\val{\langle\alpha\rangle \phi}_{M}\subseteq \val{\psi}_M$.

\begin{figure}
\begin{center}
\ \ \ \ \ \ \ \ \ \ \ \ \ \ \ \ \ \ \ \
%\noindent \hspace*{\fill}
\setlength{\unitlength}{.25mm}
\begin{picture}(140,40)(-32,-15)
%\put(-120,-40){\framebox(240,180){}}
\thicklines

\put(-100,0){\circle{40}} %\put(0,0){\circle{16}}
  \put(-100,7){\makebox(0,0){$\mathbf{u}$}}
  \put(-100,-7){\makebox(0,0){$p, r$}}

\qbezier(-150, 0)(-150,-50)(-110, -19) \qbezier(-150, 0)(-150, 50)(-110,
19) \put(-110,19){\vector(4,-3){0}}

%\put(-100,100){\circle{20}}
%  \put(-100,100){\makebox(0,0){$s_{1}$}}
%  \put(-110,100){\makebox(0,0)[r]{$\Rightarrow$}}
%\put(0,100){\circle{20}}
%  \put(0,100){\makebox(0,0){$s_{2}$}}
%\put(100,100){\circle{20}} \put(100,100){\circle{16}}
 % \put(100,100){\makebox(0,0){$s_{4}$}}
\put(0,0){\circle{40}} %\put(0,0){\circle{16}}
  \put(0,7){\makebox(0,0){$\mathbf{u}$}}
  \put(0,-7){\makebox(0,0){$p, r$}}
\put(100,0){\circle{40}}
  \put(100,7){\makebox(0,0){$\mathbf{v}$}}
  \put(100,-7){\makebox(0,0){$q$}}

%12
%\put(-90,100){\vector(1,0){80}}

%13
%\put(-93,93){\vector(1,-1){86}}

%23
\qbezier(9,21)(48, 36)(90,21)
  \put(91,21){\vector(3,-1){0}}

\qbezier(9,-21)(48, -36)(90,-21)
  \put(8,-21){\vector(-3,1){0}}

%24
%\put(10,100){\vector(1,0){80}}

%35
%\put(10,0){\vector(1,0){80}}

%33
\qbezier(-50, 0)(-50,-50)(-10, -19) \qbezier(-50, 0)(-50, 50)(-10,
19) \put(-10,19){\vector(4,-3){0}}

%55
\qbezier(110,-19)(150,-50)(150, 0) \qbezier(110, 19)(150,50)(150, 0)
\put(110,19){\vector(-4,-3){0}}

\end{picture}
\end{center}
\caption{The models $\Mbb^\alpha$ and $\Mbb$.}\label{counterexample}
\end{figure}

\noindent    Indeed, consider the model $\Mbb$ represented on the right-hand side of the Figure \ref{counterexample}
%such that $W = \{u, v\}$, $R = \{(u, v), (v, u), (u, u), (v, v)\}$,  $V(p) = \{u\}  = V(r)$ and  $V(q) = \{v\}$;
and let the action $\alpha$ be so that updating $(\Mbb,u)$ gives the model $\Mbb^\alpha$ depicted on the left-hand side of the figure. In other words, $\alpha$ is the public announcement (cf.\ \cite{BMS}) of the atomic proposition $r$. Further, let $A := [\aga] p$ and $B := q$, where $\aga$ is the agent whose equivalence relation is depicted by the arrows of the figure.
%
%Note that $\Mbb^\alpha$
%= (\{u\}, \{(u, u)\}, V^\alpha)$, where for each $s\in \{p, q, r\}$, $V^\alpha(s) = V(s)\cap V(\alpha)$.
%is the submodel on the left-hand side of the picture.
Let $i: \Mbb^\alpha\hookrightarrow \Mbb$ be the submodel injection map. Clearly, $\val{[\aga] p}_\Mbb =\varnothing$, %which is a subset of $\val{\rab q}_\Mbb$,
which implies that the inclusion $\val{A}_\Mbb\subseteq \val{\rab B}_\Mbb$ trivially holds for any interpretation of $\rab B$ in $\Mbb$;
however, $i[\val{[\aga] p}_{\Mbb^\alpha}] = \{u\}$, hence $\val{\langle\alpha\rangle [\aga] p}_{\Mbb} =
%\val{\alpha}_\Mbb\cap i[\val{[\aga] p}_{\Mbb^\alpha}] =
V(r)\cap \{u\} = \{u\}\not \subseteq \{v\} = \val{q}_\Mbb$, which falsifies the inclusion $\val{\langle\alpha\rangle A}_\Mbb\subseteq \val{B}_\Mbb$.
%This shows that the semantic interpretation of $\langle\alpha\rangle$ in (the underlying frame of) $\Mbb$ is not a left adjoint (since it does not preserve the empty join). Hence
%$$\val{\Box p}_\Mbb \leq \val{\rab q}_\Mbb \quad \not\Rightarrow\quad \val{\langle\alpha\rangle \Box p}_{\Mbb}\leq \val{q}_\Mbb.$$
%Notice also that the inequality $\val{\Box p}_\Mbb \leq \val{\rab q}_\Mbb $ holds independently of the way in which $\val{\rab q}_\Mbb$ is actually defined. This means that
This proves our claim.

\paragraph{Related work. }
Final coalgebra semantics for dynamic logics was employed by Gerbrandy and Groeneveld \cite{GerbrandyG97}, Gerbrandy \cite{GerbrandyPhD}, Baltag \cite{Ba03}, and C\^{\i}rstea and Sadrzadeh \cite{CS07}.  Adjoints of dynamic modalities with Kripke semantics were considered in Baltag, Coecke, Sadrzadeh \cite{Alexandru}. To guarantee the soundness of the rules involving the adjoints, they have to close the Kripke models under actions, which amounts, from our point of view, to generating a subcoalgebra of the final coalgebra closed under actions. The arguments reported here in favour of the final coalgebra semantics for treating dynamic modalities with their adjoints are taken from  \cite{GKPLori}.

% !TEX root = main.tex
\section{Proof-Theoretic Semantics for EAK}

In the present section, we introduce the calculus D'.EAK for the logic EAK, which is a revised and improved version of the calculus D.EAK discussed in  section \ref{D.EAK}. We argue that D'.EAK satisfies the requirements discussed in section \ref{ssec:Wansing}. On the basis of this, we propose D'.EAK as an adequate calculus from the viewpoint of proof-theoretic semantics.  We also verify that D'.EAK is a quasi proper display calculus (cf.\ definition \ref{def:quasiproper}), and hence its cut elimination theorem follows from theorem \ref{thm:meta}. %This calculus will be presented in detail, since it is the one which most closely approximates the criteria of proof-theoretic semantics.

\subsection{The  calculus D'.EAK}\label{D'.EAK} %\marginpar{\raggedright\tiny{A: this subsection needs editing and more discussion pointing to the discussion on $Pre(\alpha)$}}
As is typical of display  calculi, D'.EAK manipulates sequents of type $X\vdash Y$, where $X$ and $Y$ are structures, i.e.\ syntactic objects inductively built from formulas using {\em structural connectives}, or {\em proxies}. Every proxy is typically associated with two logical (operational) connectives, and is interpreted {\em contextually} as one or the other of the two, depending on whether it occurs in precedent or in succedent position (cf.\ definition \ref{def: display prop}).  The design of D'.EAK follows Do\v{s}en's principle (cf.\ section \ref{ssec:Wansing}); consequently, D'.EAK is modular along many dimensions. For instance, the space of the versions of EAK on  nonclassical bases, down to e.g.\ the Lambek calculus, can be captured by suitably removing structural rules. Moreover, also w.r.t.~static modal logic, the space of properly displayable normal modal logics (cf.\ \cite{Kracht}) can be  reconstructed by adding or removing structural rules in a suitable way. Finally, different types of interaction between the dynamic and the epistemic modalities can be captured by changing the relative structural rules.

In order to highlight  this modularity, we will present the system piecewise. First we give rules for the propositional base, divided into structural rules and operational rules; then we do the same for the static modal operators; finally, we introduce the rules for the dynamic modalities. %\marginpar{\raggedright\tiny{A: add something like: Notice that, in the context of the full calculus, the variables $X, Y, Z, W$ appearing in the rules in the present subsection are to be interpreted as structures of the full language of D'.EAK, unless explicitly indicated otherwise with symbols such as $X^{-\alpha}$.}}

 In the table below, we give an overview of the logical connectives of the propositional base and their proxies.

%Since we are talking about display calculus, every connective has its corresponding structural proxy.

\smallskip
\begin{center}
\begin{tabular}{|r|c|c|c|c|c|c|c|c|c}
%\hline
\cline{1-9}
 \scriptsize{Structural symbols} & \mc{2}{c|}{$<$}   & \mc{2}{c|}{$>$} & \mc{2}{c|}{$;$} & \mc{2}{c|}{I}    \\
 \hline
 \scriptsize{Operational symbols} & $\pdla$ & $\leftarrow$ & $\pdra$ & $\rightarrow$  & $\wedge$ & $\vee$  & $\top$ & $\bot$    \\
%\hline
\cline{1-9}
\end{tabular}
\end{center}
\smallskip

\noindent The table below contains  the structural rules for the propositional base:

\smallskip
\begin{center}
{\scriptsize{
\begin{tabular}{rl}

\mc{2}{c}{\normalsize{\textbf{Structural Rules}}} \\
%\hline

& \\

\AXC{\phantom{$p \fCenter p$}}
\LeftLabel{$Id$}
\UI$p \fCenter p$
\DisplayProof
 &
\AX$X \fCenter A$
\AX$A \fCenter Y$
\RightLabel{$Cut$}
\BI$X \fCenter Y$
\DisplayProof
 \\

 & \\

\AX$X \fCenter Y$
\doubleLine
\LeftLabel{$\textrm{I}^1_{L}$}
\UI$\textrm{I} \fCenter Y < X$
\DisplayProof
 &
\AX$X \fCenter Y$
\doubleLine
\RightLabel{$\textrm{I}^1_{R}$}
\UI$X < Y \fCenter \textrm{I}$
\DisplayProof
 \\

 & \\

\AX$X \fCenter Y$
\LeftLabel{$\textrm{I}^2_{L}$}
\doubleLine
\UI$\textrm{I} \fCenter X > Y$
\DisplayProof &
\AX$X \fCenter Y$
\RightLabel{$\textrm{I}^2_{R}$}
\doubleLine
\UI$Y > X \fCenter \textrm{I}$
\DisplayProof \\

 & \\

\AX$\textrm{I}  \fCenter X$
\LeftLabel{$\textrm{I} W_{L}$}
\UI$Y \fCenter X $
\DisplayProof &
\AX$X \fCenter \textrm{I}$
\RightLabel{$\textrm{I} W_{R}$}

\UI$X \fCenter Y$
\DisplayProof \\

 & \\

\AX$X \fCenter Z$
\LeftLabel{$W^1_{L}$}
\UI$Y \fCenter Z < X$
\DisplayProof &
\AX$X \fCenter Z$
\RightLabel{$W^1_{R}$}
\UI$X < Z \fCenter Y$
\DisplayProof \\

 & \\

\AX$X \fCenter Z$
\LeftLabel{$W^2_{L}$}
\UI$Y \fCenter X > Z$
\DisplayProof &
\AX$X \fCenter Z$
\RightLabel{$W^2_{R}$}
\UI$Z > X \fCenter Y$
\DisplayProof \\

 & \\

\AX$X \,; X \fCenter Y$
\LeftLabel{$C_L$}
\UI$X \fCenter Y $
\DisplayProof &
\AX$Y \fCenter X \,; X$
\RightLabel{$C_R$}
\UI$Y \fCenter X$
\DisplayProof \\

 & \\

\AX$Y \,; X \fCenter Z$
\LeftLabel{$E_L$}
\UI$X \,; Y \fCenter Z $
\DisplayProof &
\AX$Z \fCenter X \,; Y$
\RightLabel{$E_R$}
\UI$Z \fCenter Y \,; X$
\DisplayProof \\

 & \\

\AX$X \,; (Y \,; Z) \fCenter W$
\LeftLabel{$A_{L}$}
\UI$(X \,; Y) \,; Z \fCenter W $
\DisplayProof &
\AX$W \fCenter (Z \,; Y) \,; X$
\RightLabel{$A_{R}$}
\UI$W \fCenter Z \,; (Y \,; X)$
\DisplayProof \\

 & \\
%\hline
\end{tabular}
}}
\end{center}
\smallskip

The top-to-bottom  direction of each  $\textrm{I}$-rule is a special case of the corresponding \emph{weakening} rule. However, we state them all the same for the sake of modularity, since they might still be part of a calculus for a substructural logic  without weakening. The {\em weakening} rules are not given in the usual shape; the present version has the advantage that  the new structure is introduced in isolation; nevertheless, the standard version is derivable from the display postulates, as shown below:

\smallskip
\begin{center}
\begin{tabular}{c}
\AX$X \fCenter Z$
\UI$Y \fCenter Z < X$
%\UI$Z; Y \fCenter Z$
\UI$Y\,; X \fCenter Z$
\DisplayProof
\end{tabular}
\end{center}
\smallskip

\noindent Having both versions of {\em weakening} as primitive rules is useful for reducing the size of   derivations.  In the following table, we include the display postulates linking the structural connective $;$ with  $>$ and $<$:
%Without the rule \emph{exchange}, the two right adjoints $>$ and $<$ of $;$ cannot be identified and need to be both introduced as primitives. Again for the sake of modularity, we present the rules for both of them. %The rule $W^1_L$ is present as a rule which introduces the proxy for the left implication and disimplication.

\begin{center}
{\scriptsize{
\begin{tabular}{rl}

\mc{2}{c}{\normalsize{\textbf{Display Postulates}}} \\
%\hline

 & \\

\AX$X\, ; Y \fCenter Z$
\LeftLabel{$(;, <)$}
\doubleLine
\UI$X \fCenter Z < Y$
\DisplayProof &
\AX$Z \fCenter X ; Y$
\RightLabel{$(<, ;)$}
\doubleLine
\UI$Z < Y \fCenter X$
\DisplayProof \\

 & \\

\AX$X\, ; Y \fCenter Z$
\LeftLabel{$(;, >)$}
\doubleLine
\UI$Y \fCenter X > Z$
\DisplayProof &
\AX$Z \fCenter X\, ; Y$
\RightLabel{$(>, ;)$}
\doubleLine
\UI$X > Z \fCenter Y$
\DisplayProof \\

 & \\
%\hline
\end{tabular}
}}
\end{center}

\noindent In the current presentation, more  connectives with their associated rules are accounted for than in \cite{GKPLori}. The additional rules can be proved to be derivable from the remaining ones in the presence of the rules \emph{exchange} $E_L$ and $E_R$. Likewise, as is well known, by dispensing with {\em contraction}, {\em weakening} and {\em associativity}, an even wider array of connectives would ensue (for instance, dispensing with {\em weakening} and {\em contraction} would separate the additive and the multiplicative versions of each connective, etc.). We are not going to expand on these well known ideas any further, but only point out that, in the context of the whole system that we are going to introduce below, this would give a modular account of different versions of EAK with different substructural logics as propositional base. The calculus introduced here is amenable to this line of investigation.  A natural question in this respect would be to relate these ensuing proof formalisms with the semantic settings of \cite{Alexandru}.

In line with this modular perspective on the propositional base for EAK, the classical base is obtained by adding the so-called \emph{Grishin} rules (following e.g.~\cite{Gore}), encoding validities which are classical but not intuitionistic:

\smallskip
\begin{center}
{\scriptsize{
\begin{tabular}{rl}

\mc{2}{c}{\normalsize{\textbf{Grishin rules}}} \\
%\hline

 & \\

\AX$X\, > ( Y;Z) \fCenter W$
\LeftLabel{$Gri_L$}
\doubleLine
\UI$(X>Y); Z \fCenter W$
\DisplayProof &
\AX$W \fCenter X >(Y; Z)$
\RightLabel{$Gri_R$}
\doubleLine
\UI$W \fCenter (X> Y);Z$
\DisplayProof \\

 & \\
%\hline
\end{tabular}
}}
\end{center}
\smallskip
This modular treatment can be regarded as an application of Do\v{s}en's principle: calculi for versions of EAK with stronger and stronger propositional bases are obtained by progressively adding  structural rules, but keeping the same operational rules. As a consequence, cut elimination for the different versions will follow immediately from the cut-elimination metatheorem without having to verify condition C$_8$ again.

The following table shows the operational rules for the propositional base:

\smallskip
\begin{center}
{\scriptsize{
\begin{tabular}{rl}
\mc{2}{c}{\normalsize{\textbf{Operational Rules}}} \\
%\hline
 & \\

\AXC{\phantom{$\bot \fCenter \textrm{I}$}}
\LeftLabel{$\bot_L$}
\UI$\bot \fCenter \textrm{I}$
\DisplayProof &
\AX$X \fCenter \textrm{I}$
\RightLabel{$\bot_R$}
\UI$X \fCenter \bot$
\DisplayProof \\

 & \\

\AX$\textrm{I} \fCenter X$
\LeftLabel{$\top_L$}
\UI$\top \fCenter X$
\DisplayProof
 &
\AXC{\phantom{$\textrm{I} \fCenter \top$}}
\RightLabel{$\top_R$}
\UI$\textrm{I} \fCenter \top$
\DisplayProof\\

 & \\

\AX$A \,; B \fCenter Z$
\LeftLabel{$\pand_L$}
\UI$A \pand B \fCenter Z$
\DisplayProof &
\AX$X \fCenter A$
\AX$Y \fCenter B$
\RightLabel{$\pand_R$}
\BI$X \,; Y \fCenter A \pand B$
\DisplayProof
\ \, \\

 & \\
\AX$A \fCenter X$
\AX$B \fCenter Y$
\LeftLabel{$\vee_L$}
\BI$A \vee B \fCenter X \,; Y$
\DisplayProof &
\AX$Z \fCenter A \,; B$
\RightLabel{$\vee_R$}
\UI$Z \fCenter A \vee B $
\DisplayProof \\

 & \\
\ \,

\AX$B \fCenter Y  $
\AX$X \fCenter A  $
\LeftLabel{$\leftarrow_L$}
\BI$B \leftarrow A \fCenter Y<X$
\DisplayProof &
\AX$Z \fCenter B<A$
\RightLabel{$\leftarrow_R$}
\UI$Z \fCenter B \leftarrow A$

\DisplayProof \\

 & \\
\ \,
\AX$B <A \fCenter Z$
\LeftLabel{$\pdla_L$}
\UI$B \pdla A \fCenter Z$
\DisplayProof &
\AX$Y \fCenter B$
\AX$ A \fCenter X$
\RightLabel{$\pdla_R$}
\BI$ Y <  X  \fCenter B \pdla A $
\DisplayProof \\

 & \\
\ \,
\AX$X \fCenter A$
\AX$B \fCenter Y$
\LeftLabel{$\rightarrow_L$}
\BI$A \rightarrow B \fCenter X > Y$
\DisplayProof &
\AX$Z \fCenter A > B$
\RightLabel{$\rightarrow_R$}
\UI$Z \fCenter A \rightarrow B $
\DisplayProof \\

 & \\
\ \,
\AX$A> B\fCenter Z$
\LeftLabel{$\pdra_L$}
\UI$A \pdra B \fCenter Z$
\DisplayProof &
\AX$A \fCenter X$
\AX$ Y  \fCenter B$
\RightLabel{$\pdra_R$}
\BI$X>Y \fCenter A \pdra B $
\DisplayProof \\

 & \\

%\hline
\end{tabular}
}}
\end{center}
\smallskip

\noindent As is well known, in the presence of {\em exchange}, the connectives $\leftarrow$ and $\pdla$ are identified with $\rightarrow$ and $\pdra$, respectively.
Notice that the rules $\bot_R$ and $\top_L$ are derivable in the presence of {\em weakening} and the $\textrm{I}$-rules. An example of such a derivation is given below:

\smallskip
\begin{center}
\AX$X \fCenter \textrm{I}$
\UI$\textrm{I} > X \fCenter \bot $
\UI$ X \fCenter \textrm{I} \ ; \bot $
\UI$ X < \bot \fCenter \textrm{I} $
\UI$ X \fCenter \bot   $
\DisplayProof
\end{center}
\smallskip

%%%%%%%%%%%%%%%%%%

The rules for the normal epistemic modalities can be added to the system above or to any of its variants discussed early on. To this end, the language is now expanded with two contextual proxies and four operational connectives for every agent $\aga$, as follows:

\smallskip
\begin{center}

\begin{tabular}{|c|c|c|c|c|c|c|c|c|}
%\hline
\cline{1-5}
{\scriptsize{Structural symbols}} &\mc{2}{c|}{$\{\aga\}$}       & \mc{2}{c|}{$\RESagaProxy$}    \\
\hline
{\scriptsize{Operational symbols}} &  $\lc\aga\rc$ & $\ls\mkern1.8mu{\aga}\mkern1.8mu\rs$   & $\RESagaDia$  & $\RESagaBox  $ \\
%\hline
\cline{1-5}
\end{tabular}
\end{center}
\smallskip

\noindent The proxies $\{\aga\}$ and $\RESagaProxy$ are translated into  diamond-type modalities when occurring in precedent position and into  box-type modalities when occurring in succedent position. The structural rules, the display postulates, and the operational rules for the static modalities are respectively given in the following three tables:

\smallskip
\begin{center}
{\scriptsize{
\begin{tabular}{rl}

\mc{2}{c}{\normalsize{\textbf{Structural Rules}}} \\
%\hline
& \\

\AX$\textrm{I} \fCenter X$
\LeftLabel{$nec_L^{ep}$}
\UI$ \{\aga \}\, \textrm{I} \fCenter  X $
\DisplayProof
&
\AX$ X \fCenter  \textrm{I} $
\RightLabel{$nec_R^{ep}$}
\UI$ X \fCenter \{\aga\}\, \textrm{I} $
\DisplayProof  \\

 & \\

\AX$ \textrm{I}  \fCenter X$
\LeftLabel{${^{ep}nec_L}$}
\UI$\RESagaProxy\, \textrm{I}  \fCenter  X$
\DisplayProof &
\AX$X \fCenter  \textrm{I}  $
\RightLabel{${^{ep}nec_R}$}
\UI$ X \fCenter \RESagaProxy\, \textrm{I} $
\DisplayProof \\

 & \\

\AX$\{\aga\} Y > \{\aga\} Z \fCenter X$
\LeftLabel{$FS^{ep}_L$}
\UI$\{\aga\} (Y > Z) \fCenter X$
\DisplayProof &
\AX$Y \fCenter \{\aga\} X > \{\aga\} Z$
\RightLabel{$FS^{ep}_R$}
\UI$Y \fCenter \{\aga\} (X > Z)$
\DisplayProof \\

 & \\

\AX$\{\aga\} X\,; \{\aga\} Y \fCenter Z$
\LeftLabel{$mon^{ep}_L$}
\UI$\{\aga\} (X\,; Y) \fCenter Z$
\DisplayProof
&
\AX$Z \fCenter \{\aga\} Y\,; \{\aga\} X$
\RightLabel{$mon^{ep}_R$}
\UI$Z \fCenter \{\aga\} (Y\,; X)$
\DisplayProof \\

& \\
\AX$\RESagaProxy Y > \RESagaProxy X \fCenter Z$
\LeftLabel{$^{ep}FS_L$}
\UI$\RESagaProxy (Y > X) \fCenter Z$
\DisplayProof &
\AX$Y \fCenter \RESagaProxy X > \RESagaProxy Z$
\RightLabel{$^{ep}FS_R$}
\UI$Y \fCenter \RESagaProxy (X > Z)$
\DisplayProof \\

 & \\
\AX$\RESagaProxy X\,; \RESagaProxy Y \fCenter Z$
\LeftLabel{$^{ep}mon_L$}
\UI$\RESagaProxy (X\,; Y) \fCenter Z$
\DisplayProof &
\AX$Z \fCenter \RESagaProxy Y\,; \RESagaProxy X$
\RightLabel{$^{ep}mon_R$}
\UI$Z \fCenter \RESagaProxy (Y\,; X)$
\DisplayProof \\

 & \\
%\hline
\end{tabular}
}}
\end{center}
\smallskip

% \marginpar{\raggedright\tiny{A:  make the $\RESagaProxy$ smaller in the label of the rule}}
\noindent Notice that the  $mon$-rules (the soundness of which is due to  the monotonicity of $\lc\alpha\rc$ and $\ls{\mkern1.8mu{\alpha}}\mkern1.8mu\rs$) are derivable from the $FS$-rules in the presence of non restricted \emph{weakening} and \emph{contraction}.

\noindent The    $FS$-rules  above  encode the following Fischer Servi-type axioms:
\begin{eqnarray*}
 \lc \aga \rc A\rightarrow [\mkern1.8mu\aga\mkern1.8mu] B  \vdash [\mkern1.8mu\aga\mkern1.8mu] (A\rightarrow B) & &  \RESagaDia  A\rightarrow \RESagaBox  B    \vdash \RESagaBox  (A\rightarrow B)       \\
%\lc \alpha \rc A\rightarrow [\mkern1.8mu\alpha\mkern1.8mu] B  \vdash [\mkern1.8mu\alpha\mkern1.8mu] (A\rightarrow B) & &  \RESalphaDia  A\rightarrow \RESalphaBox  B    \vdash \RESalphaBox  (A\rightarrow B)       \\
  \lc\mkern1.8mu\aga\mkern1.8mu \rc (A\pdra B)
  \vdash
   [\mkern1.8mu\aga\mkern1.8mu] A\pdra  \lc \aga \rc B
 & &  \RESagaDia (A\pdra B) \vdash \RESagaBox  A\pdra \RESagaDia  B .
 \end{eqnarray*}
%
%\noindent As is well known, e.g.\ from correspondence theory \cite{ALBA} \marginpar{\raggedright\tiny{A: add reference o ALBA paper}},
These axioms encode the link between $\langle\aga\rangle$ and $[\aga]$ (and $\RESagaDia$ and $\RESagaBox$), namely, that they are interpreted semantically using the same relation in a Kripke frame.
This link can be alternatively expressed by  {\em conjugation axioms}, given below both in the diamond- and in the box-version:
\begin{eqnarray}
\lc\aga\rc A\wedge B\vdash \lc\aga\rc (A\wedge \RESagaDia B) &\quad \quad& \RESagaDia A\wedge B\vdash \RESagaDia (A\wedge \lc \aga \rc B), \label{eq:conjugation for diamond}\\
\ls\aga\rs(\RESagaBox A\vee B)\vdash (A\vee \ls\aga\rs B) &\quad \quad& \RESagaBox(\ls\aga\rs A\vee B)\vdash (A\vee \RESagaBox B),
\label{eq:conjugation for box}
\end{eqnarray}

\noindent which in  turn can be encoded in the following \emph{conjugation} rules:

{\scriptsize{
\[
\AX$\{\aga\} (X\,; \RESagaProxy Y) \fCenter Z$
\LeftLabel{\scriptsize{$conj$}}
\UI$ \{\aga\} X\,;Y \fCenter Z$
\DisplayProof
\quad
\quad
\AX$X \fCenter \{\aga\}  (Y\,; \RESagaProxy Z) $
\RightLabel{\scriptsize{$conj$}}
\UI$X \fCenter \{\aga\} Y\,;Z $
\DisplayProof
\]

\[
\AX$\RESagaProxy (X\,; \{\aga \} Y) \fCenter Z$
\LeftLabel{\scriptsize{$conj$}}
\UI$ \RESagaProxy X\,;Y \fCenter Z$
\DisplayProof
\quad
\quad
\AX$X \fCenter \RESagaProxy (Y\,; \{\aga \} Z) $
\RightLabel{\scriptsize{$conj$}}
\UI$X \fCenter  \RESagaProxy Y\,;Z $
\DisplayProof
\]
}}

\noindent
The $conj$-rules  and the $FS$-rules can be shown to be interderivable thanks to the following display postulates.

\smallskip
\begin{center}
{\scriptsize{
\begin{tabular}{rl}
\mc{2}{c}{\normalsize{\textbf{Display Postulates}}} \\

%\hline

& \\

\AX$\{\aga\} X \fCenter Y$
\LeftLabel{$(\{\aga\},  \RESagaProxy)$}
\doubleLine
\UI$X \fCenter  \RESagaProxy Y$
\DisplayProof &
\AX$X \fCenter \{\aga\} Y$
\RightLabel{$(\RESagaProxy, \{\aga\})$}
\doubleLine
\UI$ \RESagaProxy X \fCenter Y$
\DisplayProof \\

 & \\
%\hline
\end{tabular}
}}
\end{center}

\smallskip

\begin{center}
{\scriptsize{
\begin{tabular}{rl}
\mc{2}{c}{\normalsize{\textbf{Operational Rules}}}\\
%\hline
 & \\

\AX$\{\aga\} A \fCenter X$
\LeftLabel{$\lc\aga\rc_L$}
\UI$\lc\aga\rc A \fCenter X$
\DisplayProof
 &
\AX$X \fCenter A$
\RightLabel{$\lc\aga\rc_R$}
\UI$\{\aga\} X \fCenter \lc\aga\rc A$
\DisplayProof
 \\

 & \\

\AX$A \fCenter X$
\LeftLabel{$\ls\mkern0.6mu\aga\mkern0.6mu\rs_L$}
\UI$\ls\mkern1mu\aga\mkern1mu\rs A \fCenter \{\aga\} X$
\DisplayProof
 &
\AX$X \fCenter \{\aga\} A$
\RightLabel{$\ls\mkern0.6mu\aga\mkern0.6mu\rs_R$}
\UI$X \fCenter \ls\mkern1mu\aga\mkern1mu\rs A$
\DisplayProof
 \\

 & \\

\AX$\RESagaProxy A \fCenter X$
\LeftLabel{$\RESagaDia_L$}
\UI$\RESagaDia A \fCenter X$
\DisplayProof
 &
\AX$X \fCenter A$
\RightLabel{$\RESagaDia_R$}
\UI$\RESagaProxy X \fCenter \RESagaDia A$
\DisplayProof
 \\

 & \\

\AX$A \fCenter X$
\LeftLabel{$\RESagaBox_L$}
\UI$ \RESagaBox A \fCenter \RESagaProxy X$
\DisplayProof
 &
\AX$X \fCenter \RESagaProxy A$
\RightLabel{$\RESagaBox_R$}
\UI$X \fCenter \RESagaBox A$
\DisplayProof
 \\

 & \\
%\hline
\end{tabular}
}}
\end{center}
\smallskip

%\subsubsection{}

%What we aim to prove is that proof theoretic semantics can be given for special/new logical constants, such as dynamic modalities and also for different types such as agents and actions and their interaction.

\noindent The rules presented so far are essentially adaptations of display calculi of Gor\'e's \cite{Gore}. Let us   turn to the dynamic part of the calculus D'.EAK: the language is now expanded by adding, for each action $\alpha$:

- two contextual proxies, together with their  four corresponding  operational unary connectives;

- one constant symbol and its corresponding structural proxy:

\begin{center}
\begin{tabular}{|c|c|c|c|c|c|c|c|c|l}
%\aline
\cline{1-7}
\scriptsize{Structural symbols} &\mc{2}{c|}{$\{\alpha\}$}       & \mc{2}{c|}{$\RESalphaProxy$} & \mc{2}{c|}{$\Phi_\alpha$}   \\
\hline
\scriptsize{Operational symbols} &  $\lc\alpha\rc$ & $\ls\mkern2mu\alpha\mkern1mu\rs$  & $\RESalphaDia$ & $\RESalphaBox$  & $1_\alpha$ & $\phantom{1_\alpha}$\\
%\hline
\cline{1-7}
\end{tabular}
\end{center}
\noindent As in the previous version D.EAK, the proxies $\{\alpha\}$ and $\RESalphaProxy$ are translated into  diamond-type modalities when occurring in precedent position, and into  box-type modalities when occurring in succedent position. An important difference between D.EAK and D'.EAK is the introduction of the structural and operational constants $\Phi_\alpha$ and $1_\alpha$; indeed, the additional expressivity they provide is used to capture the proof-theoretic behaviour of the metalinguistic abbreviation $Pre(\alpha)$ at the object-level.   As was the case of $Pre(\alpha)$ in  D.EAK, the rules below will be such that the proxy $\Phi_\alpha$ can occur only in precedent position. Hence, the $\Phi_\alpha$ can never be interpreted as anything else than $1_\alpha$. However, a natural way to extend D'.EAK would be to introduce an operational constant $0_\alpha$, intuitively standing for the postconditions of $\alpha$ for each action $\alpha$, and dualize the relevant rules so as to capture the behaviour of postconditions. In the present paper, this expansion is not pursued any further.

The two tables below introduce the  structural rules for the dynamic modalities which are analogous to those for the static modalities given early on. %In what follows,  $X^{-\aga}$ and $Y^{-\aga}$ denote  structures restricted to the language in which neither the epistemic modalities nor their proxies occur.

\begin{center}
{\scriptsize{
\begin{tabular}{rl}
\mc{2}{c}{\normalsize{\textbf{Structural Rules}}} \\
%\hline
& \\

\AX$\textrm{I} \fCenter  X $
\LeftLabel{$nec_L^{dyn}$}
\UI$ \{\alpha \}\, \textrm{I} \fCenter  X $
\DisplayProof
&
\AX$ X \fCenter  \textrm{I} $
\RightLabel{$nec_R^{dyn}$}
\UI$ X \fCenter   \{\alpha\}\, \textrm{I} $
\DisplayProof  \\

 & \\

\AX$ \textrm{I}  \fCenter X$
\LeftLabel{$^{dyn}nec_L $}
\UI$\RESalphaProxy\, \textrm{I}  \fCenter  X$
\DisplayProof &
\AX$X \fCenter  \textrm{I}  $
\RightLabel{$^{dyn}nec_R$}
\UI$ X \fCenter \RESalphaProxy\, \textrm{I} $
\DisplayProof \\

 & \\

\AX$\{\alpha\} Y > \{\alpha\} Z \fCenter X$
\LeftLabel{$FS^{dyn}_L$}
\UI$\{\alpha\} (Y > Z) \fCenter X$
\DisplayProof &
\AX$Y \fCenter \{\alpha\} X > \{\alpha\} Z$
\RightLabel{$FS^{dyn}_R$}
\UI$Y \fCenter \{\alpha\} (X > Z)$
\DisplayProof \\

 & \\

\AX$\{\alpha\} X\,; \{\alpha\} Y \fCenter Z$
\LeftLabel{$mon^{dyn}_L$}
\UI$\{\alpha\} (X\,; Y) \fCenter Z$
\DisplayProof
&
\AX$Z \fCenter \{\alpha\} Y\,; \{\alpha\} X$
\RightLabel{$mon^{dyn}_R$}
\UI$Z \fCenter \{\alpha\} (Y\,; X)$
\DisplayProof \\

& \\
\AX$\RESalphaProxy Y > \RESalphaProxy X \fCenter Z$
\LeftLabel{$^{dyn}FS_L$}
\UI$\RESalphaProxy (Y > X) \fCenter Z$
\DisplayProof &
\AX$Y \fCenter \RESalphaProxy X > \RESalphaProxy Z$
\RightLabel{$^{dyn}FS_R$}
\UI$Y \fCenter \RESalphaProxy (X > Z)$
\DisplayProof \\

 & \\
\AX$\RESalphaProxy X\,; \RESalphaProxy Y \fCenter Z$
\LeftLabel{$^{dyn}mon_L$}
\UI$\RESalphaProxy (X\,; Y) \fCenter Z$
\DisplayProof &
\AX$Z \fCenter \RESalphaProxy Y\,; \RESalphaProxy X$
\RightLabel{$^{dyn}mon_R$}
\UI$Z \fCenter \RESalphaProxy (Y\,; X)$
\DisplayProof
\\
& \\
%\hline
\end{tabular}
}}
\end{center}

\noindent Analogous  considerations as those made for the epistemic $FS$- and $mon$-rules apply to  the dynamic $FS$- and $mon$-rules above, also in relation to analogous {\em conjugation} rules. %which  will be shown to be interderivable with the $FS$-rules in the appendix,   subsection \ref{}. \marginpar{\raggedright\tiny{A: refer to appropriate subsection or delete}}

\smallskip
\begin{center}
{\scriptsize{
\begin{tabular}{rl}
\mc{2}{c}{\normalsize{\textbf{Display Postulates}}} \\
%\hline

 & \\

\AX$\{\alpha\} X \fCenter Y$
\LeftLabel{$(\{\alpha\}, \RESalphaProxy)$}
\doubleLine
\UI$X \fCenter \RESalphaProxy Y$
\DisplayProof

 &

\AX$Y \fCenter \{\alpha\} X$
\RightLabel{$(\RESalphaProxy, \{\alpha\})$}
\doubleLine
\UI$\RESalphaProxy Y \fCenter X$
\DisplayProof \\

 & \\
%\hline
\end{tabular}
}}
\end{center}
\smallskip

Next, we introduce the structural rules which are to capture the specific behaviour of epistemic actions \label{pageref:atom}
\smallskip
\begin{center}
{\scriptsize{
\begin{tabular}{rl}

\mc{2}{c}{\normalsize{\textbf{Atom}}} \\
%\hline

& \\

\mc{2}{c}{
\AXC{}
\RightLabel{{$atom$}}
\LeftLabel{{$\phantom{atom}$}}
\UI$\Gamma p \fCenter \Delta p$
%\LeftLabel{\fns$atom^p_{L}$}
%\doubleLine
%\UI$\Phi_\alpha \ ;\ X^{-\aga}\fCenter Y^{-\aga}$
%\DisplayProof &
%
%\AX$X^{-\aga}\fCenter \{\alpha\} Y^{-\aga}$
%\RightLabel{\fns$atom^p_{R}$}
%\doubleLine
%\UI$ X^{-\aga} \fCenter \Phi_\alpha> Y^{-\aga}$
\DisplayProof
} \\

%& \\

%\AXC{\phantom{$p \fCenter p$}}
%\LeftLabel{\fns$atom_L$}
%\UI$\{\alpha\} \Phi_\alpha \fCenter 1_\alpha$
%\DisplayProof &
%\AXC{\phantom{$p \fCenter p$}}
%\LeftLabel{\fns$atom_R$}
%\UI$NO\Phi_\alpha \fCenter \{\alpha\} 1_\alpha$
%\DisplayProof \\

 & \\
%\hline
\end{tabular}
}}
\end{center}
\smallskip

\noindent where $\Gamma$ and $\Delta$ are arbitrary finite sequences of the form $(\alpha_1)\ldots (\alpha_n)$, such that each  $(\alpha_j)$ is of the form $\{\alpha_j\}$ or of the form $\RESalphajProxy$, for $1\leq j\leq n$. Intuitively, the {\em atom rules} capture the requirement that epistemic actions do not change the factual state of affairs (in the Hilbert-style presentation of EAK, this is encoded in the axiom \eqref{eq:facts} in section \ref{ssec:EAK}).

\smallskip
\begin{center}
{\scriptsize{
\begin{tabular}{@{}rl@{}}
\mc{2}{c}{\normalsize{\textbf{Structural Rules for Epistemic Actions}}} \\
%\hline

%\begin{center}
%\def\fCenter{\fns\emph{atom$_L$}}
%\AXC{$X \fCenter Y$}
%\noLine
%\def\fCenter{\vdash}
%\UIC{${\{\alpha\}\,p} \fCenter p$}
%\DisplayProof
 %&
%\!\!\!\!\!\!
%\def\fCenter{\fns\emph{atom$_R$}}
%\AXC{$\fCenter$}
%\noLine
%\def\fCenter{\vdash}
%\UI$p \fCenter {\{\alpha\}\,p}$
%\DisplayProof
% \\

 & \\

\mc{2}{c}{\ \ \ \
\AX$X \fCenter Y$
\RightLabel{\emph{balance}}
\UI$\{\alpha\} X\fCenter \{\alpha\} Y$
\DisplayProof}\\

 & \\

\AX$\{\alpha\} \RESalphaProxy X  \fCenter  Y$
\LeftLabel{$comp^{\alpha}_L$}
\UI$ \Phi_\alpha; X \fCenter  Y$
\DisplayProof
&
\AX$X \fCenter  \{\alpha\} \RESalphaProxy Y$
\RightLabel{$comp^{\alpha}_R$}
\UI$ X \fCenter \Phi_\alpha >Y$
\DisplayProof
\\

 & \\

\AX$ \Phi_\alpha;  \{\alpha\} X \fCenter  Y$
\LeftLabel{\emph{reduce'}$_{L}$}
\UI$ \{\alpha\} X \fCenter  Y$
\DisplayProof
&
\AX$ Y \fCenter \Phi_\alpha > \{ \alpha \} X$
\RightLabel{\emph{reduce'}$_{R}$}
\UI$ Y \fCenter \{ \alpha \} X$
\DisplayProof \\

& \\
\AX$ \{\alpha\} \{\aga\} X \fCenter Y$
\LeftLabel{\emph{swap-in'}$_{L}$}
\UI$\Phi_\alpha; {\{\aga\}\{\beta\}_{\alpha\aga\beta}\, X} \fCenter Y$
\DisplayProof
&
\AX$Y \fCenter \{\alpha\} \{\aga\} X$
\RightLabel{\emph{swap-in'}$_{R}$}
\UI$Y \fCenter \Phi_\alpha > {\{\aga\} \{\beta\}_{\alpha\aga\beta}\,X}$
\DisplayProof\\

& \\

\AX$\Big ( \{\aga\}\{\beta\} \,X \fCenter Y\mid \alpha\aga\beta\Big)$
\LeftLabel{\emph{swap-out'}$_{L}$}
\UI$ \{\alpha\} \{\aga\} X \fCenter \Bigsemic\Big(Y\mid \alpha\aga\beta\Big)$
\DisplayProof
&
\AX$\Big(Y \fCenter \{\aga\}\{\beta\}\,X \mid \alpha\aga\beta\Big)$
\RightLabel{\emph{swap-out'}$_{R}$}
\UI$\Bigsemic\Big(Y\mid \alpha\aga\beta\Big) \fCenter  \{\alpha\} \{\aga\} X$
\DisplayProof\\

&  \\

%\hline
\end{tabular}
%\end{center}
}}
\end{center}
\smallskip

\noindent The {\em swap-in'} rules are unary and should be read as follows: if the premise holds, then the conclusion holds relative to any action $\beta$ such that $\alpha\aga\beta$. The {\em swap-out'} rules do not have a fixed arity; they have as many premises\footnote{The {\em swap-out} rule could indeed be infinitary if action structures were allowed to be infinite, which in the present setting, as in \cite{BMS}, is not the case.} as there are actions $\beta$ such that $\alpha\aga\beta$. In the conclusion, the symbol $\Bigsemic \Big(Y \mid \alpha\aga\beta \Big)$ refers to a string $(\cdots(Y\,; Y) \,;\cdots \,; Y)$ with $n$ occurrences of $Y$, where $n = |\{\beta\mid \alpha\aga\beta\}|$. The {\em swap-in} and {\em swap-out} rules  encode the interaction between dynamic and epistemic modalities as it is captured by the interaction axioms in the Hilbert style presentation of EAK (cf.\ \eqref{eq:interact axiom} in section \ref{ssec:EAK} and similarly in section \ref{ssec:intEAK}). The {\em reduce} rules encode well-known EAK validities such as $\langle\alpha\rangle A \rightarrow (Pre(\alpha)\wedge \langle\alpha\rangle A )$.

%\marginpar{\raggedright\tiny{A: expand on what these rules intuitively mean}}

\noindent  Finally, the  operational rules for $\lc\alpha\rc$, $\ls\alpha\rs$, and $1_\alpha$ are given in the table below:

{\scriptsize{
\begin{center}
\begin{tabular}{rl}
\mc{2}{c}{\normalsize{\textbf{Operational Rules}}} \\
%\hline
 & \\
\AX$\{\alpha\} A \fCenter X$
\LeftLabel{\fns$\lc\alpha\rc_L$}
\UI$\lc\alpha\rc A \fCenter X$
\DisplayProof &
\AX$X \fCenter A$
\RightLabel{\fns$\lc\alpha\rc_R$}
\UI$\{\alpha\} X \fCenter \lc\alpha\rc A$
\DisplayProof \\
 & \\
\AX$A \fCenter X$
\LeftLabel{\fns$\ls\alpha\rs_L$}
\UI$\ls\alpha\rs A \fCenter \{\alpha\} X$
\DisplayProof &
\AX$X \fCenter \{\alpha\} A$
\RightLabel{\fns$\ls\alpha\rs_R$}
\UI$X \fCenter \ls\alpha\rs A$
\DisplayProof \\
 & \\
\AX$\Phi_\alpha  \fCenter X$
\LeftLabel{\scriptsize{${1_\alpha}_L$}}
\UI$ 1_\alpha \fCenter X$
\DisplayProof
 &
\AxiomC{\phantom{$\Phi_\alpha \vdash$}}
\RightLabel{\fns{${1_\alpha}_R$}}
\UnaryInfC{$\Phi_\alpha \vdash 1_\alpha $}
\DisplayProof \\
 & \\

%\hline
\end{tabular}
\end{center}
}
}

\subsection{Properties of D'.EAK}
\label{ssec:properties of D'.EAK}

%\noindent The  rules in the following two tables capture the specific features of EAK; all of them contain the formula $Pre(\alpha)$ as a side condition:
\paragraph{Soundness.}
The calculus D'.EAK can be readily shown to be sound with respect to the final coalgebra semantics. The general procedure has been outlined in section \ref{sec:final coalgebra sem}. The soundness of most of the rules of D'.EAK can be shown  entirely analogously  to the soundness of the corresponding rules in D.EAK, which is outlined in \cite{GKPLori}.

%
%Since, as discussed early on, the connectives $\RESalphaDia$ and $\RESalphaBox$ can only be interpreted in the final coalgebra, soundness will be proved in different ways for rules involving $\RESalphaProxy X$ and
As for rules not involving $\RESalphaProxy $,  we will rely on the following observation, which is based on the invariance of EAK-formulas under bisimulation (cf.\ Section \ref{ssec:EAK}):

\begin{lemma}
\label{lemma: final and standard sem}
The following are equivalent for all EAK-formulas $A$ and $B$:\\
(1) $\val{A}_Z\subseteq \val{B}_Z$;\\
(2) $\val{A}_M\subseteq \val{B}_M$ for every model $M$.
\end{lemma}
\begin{proof}
The direction from (2) to (1) is clear; conversely, fix a model $M$, and let $f: M\to Z$ be the unique arrow; then (1) immediately implies that $\val{A}_M  = f^{-1}(\val{A}_Z)\subseteq f^{-1}(\val{B}_Z) =  \val{B}_M$.
\end{proof}
In the light of the lemma above, and using the translations provided in Table \ref{pic:1}, the soundness of unary rules $A\vdash B/C\vdash D$ not involving $\RESalphaProxy$, such as {\em balance}, $\langle\alpha\rangle_R $ and $[\alpha]_L$, can be straightforwardly checked as implications of the form ``if $\val{A}_M\subseteq \val{B}_M$ on every model $M$, then $\val{C}_M\subseteq \val{D}_M$ on every model $M$''. As an example, let us check the soundness of {\em balance}: Let $A, B$ be EAK-formulas such that $\val{A}_M\subseteq \val{B}_M$ on every model $M$. Let us fix a model $M$, and show that $\val{\langle\alpha\rangle A}_M\subseteq \val{[\alpha] B}_M$. As discussed in \cite[Subsection 4.2]{KP13}, the following identities hold in any standard model:
\begin{gather}\label{eq:extension alpha phi}\val{ \langle \alpha \rangle A}_M = \val{Pre(\alpha)}_M \cap \iota_k^{-1} [i[\val{A}_{M^\alpha}]],\\
\label{eq:extension box alpha phi}\val{ [\alpha] A}_M = \val{Pre(\alpha)}_M \Rightarrow \iota_k^{-1} [i[\val{A}_{M^\alpha}]],\end{gather}
where the map $i: M^\alpha\to \coprod_\alpha M$ is the submodel embedding, and $\iota_k: M\to \coprod_\alpha M$ is the embedding of $M$ into its $k$-colored copy. Letting $g(-): = \iota_k^{-1} [i[-]]$, we need to show that

$$\val{Pre(\alpha)}_M\cap g(\val{A}_{M^\alpha})\subseteq \val{Pre(\alpha)}_M\Rightarrow g(\val{B}_{M^\alpha}).$$
This is a direct consequence of the Heyting-valid implication ``if $b\leq c$ then $a\wedge b\leq a\to c$'', the monotonicity of $g$, and the assumption that $\val{A}_M\subseteq \val{B}_M$ holds on every model, hence on $M^\alpha$. %Notice that this holds in any algebraic model   %and is best illustrated by the following lemma, the straightforw

%Given these stipulations,
Actually, for all rules $(A_i\vdash B_i\mid i\in I)/C\vdash D$  not involving $\RESalphaProxy $ except {\em balance}, $\langle\alpha\rangle_R $ and $[\alpha]_L$,   stronger soundness statements can be proven of the form ``for every model $M$, if $\val{A_i}_M\subseteq \val{B_i}_M$ for every $i\in I$, then $\val{C}_M\subseteq \val{D}_M$'' (this amounts to the soundness w.r.t. the standard semantics). %However, {\em only} for items 1, 4 and 5 of Proposition \ref{prop: soundness part 1} our result cannot be easily improved (cf.\ Remark \ref{rem:more on limitations}).}.
This is the case for all display postulates not involving $\RESalphaProxy$, the soundness of which boils down to the well known adjunction conditions holding in every model $M$. As to the remaining rules not involving $\RESalphaProxy$, thanks to the following general principle of {\em indirect (in)equality}, the stronger soundness condition above boils down to the verification of inclusions which interpret validities of IEAK \cite{KP13}, and hence, a fortiori, of EAK. Same arguments hold for the Grishin rules, except that their soundness boils down to classical but not intuitionistic validities.
%ard proof of which is omitted.
\begin{lemma}
{\bf (Principle of indirect inequality)} Tfae for any preorder $P$ and all $a, b\in P$:\\
(1) $a\leq b$;\\
(2) $x\leq a$ implies $x\leq b$ for every $x\in P$;\\
(3) $b\leq y$ implies $a\leq y$ for every $y\in P$.
\end{lemma}
As an example, let us verify \emph{s-out}$_{L}$: fix a model $M$, fix EAK-formulas $A$ and $B$, and assume that for every action $\beta$, if $\alpha\aga\beta$ then $\val{\langle\aga\rangle\lc\beta\rc A}_M\subseteq \val{B}_{M}$, i.e., that $\bigcup\{ \val{\langle\aga\rangle\lc\beta\rc A}_M\mid \alpha\aga\beta\}\subseteq \val{B}_{M}$; we need to show that $\val{\lc\alpha\rc\langle\aga\rangle A}_M\subseteq \val{B}_{M}$. By the principle of indirect inequality, it is enough to show that $ \val{\lc\alpha\rc\langle\aga\rangle A}_M \subseteq  \bigcup\{ \val{\langle\aga\rangle\lc\beta\rc A}_M\mid \alpha\aga\beta\}$. Indeed, since axiom \eqref{eq:interact axiom} is valid on any model, we have:
$$ \val{\lc\alpha\rc\langle\aga\rangle A}_M \subseteq  \val{Pre(\alpha)}_M\cap \bigcup\{ \val{\langle\aga\rangle\lc\beta\rc A}_M\mid \alpha\aga\beta\} \subseteq  \bigcup\{ \val{\langle\aga\rangle\lc\beta\rc A}_M\mid \alpha\aga\beta\}.$$
%\noindent which is true (cf.\ Proposition \ref{prop: classical PAL}).

The soundness of the operational rules of $1_\alpha$ is immediate; the soundness of \emph{atom}
can be proven directly on the final coalgebra by induction on the length of $\Gamma$ and $\Delta$ using the fact, mentioned on page \pageref{par : Specific desiderata for epistemic actions},
that
epistemic actions do not change the valuations of atomic formulae.
For instance, as to the base case of this induction,
let us argue for the soundness of $\RESalphaProxy p \vdash p$ and $ p \vdash \RESalphaProxy p$:
indeed, let $\alpha_Z \subseteq Z\times Z$ be the interpretation of the epistemic action $\alpha$ on the final coalgebra, then the left-hand side of the atom-sequent above is interpreted as the set $\alpha_Z [\val{p}_Z]$.
Because of the assumption on $\alpha_Z$ mentioned above
it immediately follows that $\alpha_Z [\val{p}_Z] \subseteq \val{p}_Z$, and $\alpha_Z [\val{p}^c_Z] \subseteq \val{p}^c_Z$. The former inclusion gives the soundness of
$\RESalphaProxy p \vdash p$, while the latter is equivalent to $ \val{p}_Z \subseteq (\alpha_Z [\val{p}^c_Z])^c$,
which gives the soundness of $ p \vdash \RESalphaProxy p$.

The soundness of the $comp$ rules is given in the appendix (cf.\ subsection \ref{Soundness of comp}). %\marginpar{\raggedright\tiny{A: refer to appropriate subsection}}

%The reader is referred to \cite{GKP13} for further details.
Finally, the soundness of the rules which do involve $\RESalphaProxy$ remains to be shown. The soundness of the display postulates immediately follows from Proposition \ref{prop:finalsemantics}. As an example, let us verify the soundness of $^{dyn}FS_L$: translating the structures into formulas,  it boils down to verifying that, for all EAK-formulas $A, B$ and $C$,
if $\val{\RESalphaBox A}_Z \pdra \val{\RESalphaDia B}_Z \subseteq \val{C}_{Z}$, then $\val{\RESalphaDia(A \pdra B)}_{Z}\subseteq \val{C}_{Z}$. By applying the appropriate adjunction rules, the implication above is equivalent to the following implication: if $\val{B}_Z \subseteq \val{[\alpha](\RESalphaBox A\vee C)}_{Z}$ then $\val{B}_Z \subseteq \val{A\vee [\alpha] C}_{Z}$. By applying the principle of indirect inequality,  we are reduced to showing the inclusion
$$\val{[\alpha](\RESalphaBox A\vee C)}_Z\subseteq \val{A\vee [\alpha] C}_{Z},$$
which is the soundness of the box-version of a conjugation condition (see the shape of \eqref{eq:conjugation for box} for epistemic modalities), and is true in $Z$ since  $\RESalphaBox$ is interpreted as  $\pBox{\alpha}$.

\paragraph{Completeness and  conservativity.} The completeness of D'.EAK w.r.t.\ the Hilbert presentation of EAK (cf.\ subsections \ref{ssec:EAK} and \ref{ssec:intEAK}) is achieved by showing that the axioms of (the intuitionistic version of) EAK are derivable in D'.EAK. These derivations are collected in subsection \ref{ssec: completeness}.

Again, as was the case for D.EAK, the fact that D'.EAK is a conservative extension of EAK can be argued as follows:   let $A, B$ be EAK-formulas
such that $A\vdash_{D'.EAK} B$, and let $\mathbb{Z}$ be the final coalgebra. By the soundness of D'.EAK w.r.t.\ the $\mathbb{Z}$, this implies that $\val{A}_Z\subseteq\val{B}_Z$, which, by the bisimulation invariance of EAK (cf.\ \cite[Lemma 1]{GKPLori}), %\marginpar{\raggedright\tiny{A: refer to lemma}}
implies that $\val{A}_M\subseteq\val{B}_M$  for every Kripke model $M$, which, by the completeness of EAK w.r.t.\ the standard Kripke semantics, implies that $A\vdash_{EAK} B$.

%lemma: final and standard sem

\paragraph{Adequacy of D'.EAK w.r.t.\ Wansing's criteria.}
It is easy to see that the calculus D'.EAK  enjoys the display property (cf.\ Definition \ref{def: display prop}).
Like its previous version, D'.EAK is defined independently of the relational semantics of EAK, and therefore is suitable for a fine-grained  proof-theoretic semantic analysis. It can be readily verified by inspection that all  operational rules satisfy Wansing's criteria of \emph{separation}, \emph{symmetry} and \emph{explicitness} (cf.\ subsection \ref{ssec:Wansing}). %In fact, they satisfy the \emph{segregation} condition,  i.e., all active formulas occurring in them are isolated. Since D'.EAK enjoy the display property, this very restrictive requirement does not affect the proof power of the system.

%Finally, all rules satisfy Wansing's and Belnap's requirement that each rule is closed under simultaneous substitution of arbitrary structures for congruent parameters.

Moreover, a clear-cut division of labour has been achieved between the operational rules, which are to encode the proof-theoretic meaning of the new connectives, and %the structural rules; in particular, the operational rule \emph{reverse} and
the structural rules, which are to express the relations entertained between the different connectives by way of their proxies.

Another important proof-theoretic feature of D'.EAK is modularity. %\marginpar{\raggedright\tiny{A: edit paragraph; this paragraph  is very important for the issue; can we claim that we get to the Lambek calculus???}}
As discussed in subsection \ref{D'.EAK},  %\marginpar{\raggedright\tiny{A: add reference}}
%the space of (displayable) normal modal logics can be reconstructed from the rules for static fragment, therefore
by suitably removing structural rules  for the propositional base of D'.EAK, the substructural versions of  EAK can be modularly defined. Moreover, by adding structural rules corresponding to properly displayable modal logics (cf.\ \cite{Kracht}), different assumptions can be captured on the behaviour of the epistemic modalities.\footnote{Note that \emph{Balance}, \emph{comp}, \emph{reduce}, \emph{swap-in} and \emph{swap-out} are the only specific structural rules for epistemic actions; the \emph{monotonicity}  and \emph{Fischer-Servi} rules respectively encode the conditions that box and diamond are monotone and interpreted by means of the same relation; the \emph{necessitation} can be considered as a special case of atom and $IW$ can be eliminated if, e.g., $\bot \vdash \lc\alpha\rc \bot$ and $[\alpha] \top \vdash \top$ are introduced as zeroary rules.}

Notwithstanding the fact that the old \emph{reverse} rules, offending {\em segregation}, are derived rules in D'.EAK, still the system D'.EAK does not satisfy \emph{segregation}. However, the only rule in D'.EAK offending {\em segregation} is \emph{atom} because one of the two principal formulas in each {\em atom} axioms might not occur in display. Even if the most rigid proof-theoretic semantic principle is not met, D'.EAK is a quasi-proper display calculus, and hence it
%there is enough evidence that D'.EAK  provides an adequate proof-theoretic semantics for all the connectives occurring in it, given that this calculus
enjoys Belnap-style cut elimination, as will be shown in the next subsection.

\subsection{Belnap-style cut-elimination for D'.EAK}\label{sec:C_1-C7for D'.EAK}

%As mentioned earlier on, \marginpar{\raggedright\tiny{A: edit this paragraph}}  the display calculus D.EAK of \cite{GKPLori} %only a Gentzen style cut elimination has been proven, and some of the rules did not satisfy the criteria of proof-theoretic semantics (For more on this discussion, the reader is referred to \cite[section3.4?]{thesis}) \marginpar{\raggedright\tiny{A: add reference}}.
In the present subsection, we prove that  D'.EAK is a quasi proper display calculus (cf.\ Subsection \ref{ssec:quasi-def}), that is, the rules of D'.EAK  satisfy  conditions C$_1$, C$_2$, C$_3$, C$_4$, C$_5'$, C$_5''$, C$_6$, C$_7$, C$_8$. By Theorem \ref{thm:meta}, this is enough to establish that the calculus enjoys the cut elimination and the subformula property.
%In the present subsection, we show that %.
 %which, as discussed in subsection \ref{ssec:DisplayLogic}, %\marginpar{\raggedright\tiny{A: add reference}}
%are sufficient for D'.EAK to enjoy the Belnap-style cut elimination and the subformula property \cite{Belnap, Restall, Wansing}.

The rules \emph{reverse} are now derivable, and all the rules with the side condition $Pre(\alpha)$ have been reformulated so as to either remove $Pre(\alpha)$ altogether, or to replace it with its structural counterpart.  This has been achieved by expanding the language so that the meta-linguistic abbreviation $Pre (\alpha)$ can be replaced by an operational constant and its corresponding structural connective. Hence, it can be readily verified that all rules are closed under simultaneous substitution of arbitrary structures for congruent parameters, which   %For this reason, it is immediate to see that the new \emph{swap-in} rules
satisfies  conditions C$_6$ and C$_7$.
It is easy to see that the operational rules for $1_\alpha$ and the $comp$ rules satisfy the criteria C$_1$--C$_7$.
The  \emph{atom} axioms can be readily seen to  verify condition $C''_5$ as given in subsection \ref{ssec:quasi-def}.
Finally, as to condition $C_8$, let us show the cases involving the new connective $1_\alpha$. All the other cases %have been already treated in the Gentzen-style cut-elimination proof for D.EAK (although they do not appear in \cite{GKPLori}) and
are reported in appendix \ref{cut in D'.EAK}.

{\scriptsize{
\begin{center}
\begin{tabular}{@{}rcl@{}}
\bottomAlignProof
\AXC{ $\Phi_\alpha$ \fCenter {$1_\alpha$}}
\AXC{\ \ \ $\vdots$ \raisebox{1mm}{$\pi$}}
\noLine
\UI$\Phi_\alpha \fCenter\ X$
\UI$ 1_\alpha \fCenter\ X$
\BI$\Phi_\alpha \fCenter X$
\DisplayProof
 & $\rightsquigarrow$ &
\bottomAlignProof
\AXC{\ \ \ $\vdots$ \raisebox{1mm}{$\pi$}}
\noLine
\UI$\Phi_\alpha \fCenter X$
\DisplayProof
 \\
\end{tabular}
\end{center}
}}

% !TEX root = main.tex
\section{Conclusions and further directions}

\subsection{Conclusions} In the present paper, we provide an analysis, conducted adopting the viewpoint of proof-theoretic semantics, of the state-of-the-art deductive systems for dynamic epistemic logic, focusing mainly on Baltag-Moss-Solecki's logic of epistemic actions and knowledge (EAK). We start with an overview of the general research agenda in proof-theoretic semantics, and then we focus on display calculi, as a proof-theoretic paradigm which has been successful in accounting for difficult logics, such as modal logics and substructural logics. We discuss the requirements which a proof system should satisfy to provide adequate proof-theoretic semantics to logical constants, and, as an original contribution, we introduce the notion of quasi proper display calculus, and prove its corresponding Belnap-style cut elimination metatheorem.  We then  evaluate the main existing proof systems for PAL/EAK according to the previously discussed requirements. As the second original contribution, we propose a revised version of one such system, namely of the system  D.EAK (cf.~section \ref{D.EAK}), and we argue that our revised system D'.EAK adequately meets the proof-theoretic semantic requirements  for all the logical constants involved. We also show that D'.EAK is sound w.r.t.\ the final coalgebra semantics, complete w.r.t.\ EAK, of which it is a conservative extension. These three facts together guarantee that D'.EAK exactly captures EAK. Finally, we  verify that D'.EAK is a quasi proper display calculus. Hence, the generalized metatheorem applies, and D'.EAK is thus shown to enjoy Belnap-style cut elimination (which was not argued for in the case of the original system D.EAK) and the subformula property. The main ingredient of this revision is an expansion of the language of the original system, aimed at achieving an {\em independent} proof-theoretic account of the preconditions $Pre(\alpha)$. This account is independent  both in the sense that it is given purely in terms of the resources of D'.EAK, and in the sense that the metalinguistic abbreviation $Pre(\alpha)$ is treated as a first-class citizen of the revised system. Indeed, $Pre(\alpha)$ is endowed with both an operational and a structural representation, both of which well-behaving.

\subsection{Further directions}
\label{ssec:further directions}
\paragraph{Uniform proof-theoretic account for dynamic logics. } The present paper is part of a larger research program aimed at giving a uniform proof-theoretic account to a wide class of logics which includes dynamic logics. In \cite{SabGre} and \cite{PDL}, this treatment has been extended to monotone modal logics and, respectively, to the full language of Propositional Dynamic Logic.  %are also incorporated in this   % As far as the authors know, together with PDL without the Kleene star (cf.\ \cite{Wansing}), EAK is the only dynamic logic  which has been given a proof-theoretic account of the kind devised in the present paper.
Another interesting case study is Parikh's Game Logic \cite{Parikh}, where the dynamic modalities are non normal and the set of agents is endowed with algebraic structure, which is treated in a paper \cite{GL} in preparation.

\paragraph{Multi-type display-style calculi. } The metatheorem proven in the present paper applies to a class of display calculi (the quasi-proper display calculi) which generalize Wansing's notion of proper display calculi by relaxing the property of isolation.  However, in both quasi proper and proper display calculi, rules are required to be closed under simultaneous substitution of  {\em arbitrary} structures for congruent formulas. %are schematic, in the sense that all parametric variables occur {\em unrestricted}.
This requirement occurs in a weaker form in both the original  \cite[Theorem 4.4]{Belnap} and in some of its subsequent versions \cite{Be2, Restall, Wansing}. Indeed, these metatheorems  %do not require all rules to be schematic, and
apply to display calculi
admitting  rules  %which are not closed under simultaneous substitution of  some parametric formulas for {\em arbitrary} structures, but
for which the closure under substitution may be {\em not  arbitrary}, but restricted to structures satisfying  certain conditions. This weaker requirement  primarily concerns rules; however, it is encoded  in the notion of {\em regular formula} and  asks every formula to be regular. The condition given in terms of regular formulas is key to accounting for important logics such as linear logic. On the other hand, it ingeniously relies on very special features of the signature of linear logic, and hence it is of difficult application outside that setting. We conjecture that logics such as linear logic can be alternatively accounted for by display-type calculi  all the rules of which are closed under simultaneous substitution of arbitrary structures for parametric operational terms (formulas).  %the notion of (quasi-)proper display calculi can be suitably logics that are displayable but not properly displayable  can be alternatively captured by   (quasi-)proper display calculi, %this weaker closure condition can be alternatively accounted for in the setting of (quasi) proper display calculi,
We conjecture that this is possible thanks to the introduction of  a suitable {\em multi-type} environment, in which every derivable sequent/consecution is required to be {\em type-uniform} (i.e., both the antecedent and the  consequent of any sequent/consecution must belong to the same type). The requirement formulated in terms of regular formulas would then be encoded in the multi-type setting in terms of the condition that, in each given rule,  parametric constituents (of a given and  unambiguously determined type) can be uniformly replaced by structures which are {\em arbitrary within that same type}, so as to obtain instances of the same rule. An example of such a multi-type environment is introduced in \cite{Multitype}. The adaptation of the multi-type setting to the case of linear logic is work in progress.

\appendix
\section{Special rules}\label{Special rules}

\subsection{Derived rules in D'.EAK}

In the presence of the display postulates, the $conj$-rules are interderivable with the Fischer Servi rules. Indeed, let us show that the following rules

{\scriptsize{
\[
\AX$\{\alpha\} (X~; \RESalphaProxy Y) \fCenter Z$
\LeftLabel{\scriptsize{$conj$}}
\UI$ \{\alpha\} X~;Y \fCenter Z$
\DisplayProof
\quad
\quad
\AX$Y \fCenter \RESalphaProxy X > \RESalphaProxy Z$
\RightLabel{\scriptsize{$FS$}}
\UI$Y \fCenter \RESalphaProxy (X> Z)$
\DisplayProof
\]
}
}
\noindent are interderivable:\footnote{Note that we are using exchange, but this rule is not required if we add the corresponding \emph{Fisher-Servi} rule for the right-residuum of `$;$' and the obvious \emph{conjugation} rule with `$X\,; \{\alpha\} Y$' in a reversed order.}
{\scriptsize{
\[
\AX$\{\alpha\} (X\,; \RESalphaProxy Y) \fCenter Z$
\UI$X\,;\RESalphaProxy Y \fCenter \RESalphaProxy Z$
\UI$\RESalphaProxy Y\,; X \fCenter \RESalphaProxy Z$
\UI$X \fCenter \RESalphaProxy Y > \RESalphaProxy Z$
\RightLabel{\scriptsize{$FS$}}
\UI$X \fCenter \RESalphaProxy (Y>Z)$
\UI$\{\alpha\} X \fCenter Y>Z$
\UI$Y\,; \{\alpha\} X \fCenter Z$
\UI$ \{\alpha\} X\,;Y \fCenter Z$
\DisplayProof
\quad
\quad
\AX$Y \fCenter \RESalphaProxy X > \RESalphaProxy Z$
\UI$\RESalphaProxy X\,;Y \fCenter \RESalphaProxy Z$
\UI$Y\,; \RESalphaProxy X \fCenter \RESalphaProxy Z$
\UI$\{\alpha\}(Y\,; \RESalphaProxy X) \fCenter Z$
\LeftLabel{$conj$}
\UI$ \{\alpha\}Y \,;  X \fCenter Z$
\UI$X\,; \{\alpha\}Y \fCenter Z$
\UI$\{\alpha\} Y \fCenter X > Z$
\UI$Y \fCenter \RESalphaProxy (X > Z)$
\DisplayProof
\]
}
}

\noindent Analogous derivations show that  the pairs of rules in each row of the table below are interderivable:

\begin{center}
{\scriptsize{
\begin{tabular}{lcr}
\AX$\RESalphaProxy (X\,; \{\alpha \} Y) \fCenter Z$
\LeftLabel{\scriptsize{$conj$}}
\UI$ \RESalphaProxy X\,;Y \fCenter Z$
\DisplayProof
& & 
\AX$Y \fCenter \{\alpha\} X > \{\alpha\} Z$
\RightLabel{\scriptsize{$FS$}}
\UI$Y \fCenter \{\alpha\} (X > Z)$
\DisplayProof
 \\
 
&&\\

\AX$X \fCenter \{\alpha\}  (Y\,; \RESalphaProxy Z) $
\RightLabel{\scriptsize{$conj$}}
\UI$X \fCenter \{\alpha\} Y\,;Z $
\DisplayProof
&&
\AX$\RESalphaProxy Y > \RESalphaProxy X \fCenter Z$
\LeftLabel{$FS$}
\UI$\RESalphaProxy (Y > X) \fCenter Z$
\DisplayProof 
\\

&&\\

\AX$X \fCenter \RESalphaProxy (Y\,; \{\alpha \} Z) $
\RightLabel{\scriptsize{$conj$}}
\UI$X \fCenter  \RESalphaProxy Y\,;Z $
\DisplayProof

&&

\AX$\{\alpha\} Y > \{\alpha\} Z \fCenter X$
\LeftLabel{$FS$}
\UI$\{\alpha\} (Y > Z) \fCenter X$
\DisplayProof 
\\
\end{tabular}
}}
\end{center}
%\noindent

\noindent Let us show that the  rules ``with side conditions'' in D.EAK (cf.\ subsection \ref{D.EAK}) can be derived from their corresponding  rules in D'.EAK and the remaining part of the calculus.

\noindent An important benefit of the revised system is that the operational rules \emph{reverse} (or more precisely their rewritings in the new notation), which were primitive in the old system, are now derivable using the new rules for $\Phi_\alpha$ and $1_\alpha$ and the new \emph{reduce}. This supports our intuition that the rules \emph{reverse} do not participate in the proof-theoretic meaning of the connectives $\lc \alpha \rc$ and $[\alpha]$.

{\scriptsize{
\[
\AX$\Phi_{\alpha} \fCenter 1_\alpha$
\AX$1_\alpha\,; \{\alpha\} A \fCenter X$
%\UI$\{\alpha\} A\,; 1_\alpha \fCenter X$
\UI$1_\alpha \fCenter X < \{\alpha\} A$
\BI$\Phi_{\alpha} \fCenter X < \{\alpha\} A$
%\UI$\{\alpha\} A\,;\Phi_\alpha \fCenter X$
\UI$\Phi_\alpha\,; \{\alpha\} A \fCenter X$
\UI$\{\alpha\} A \fCenter X$
\UI$A \fCenter \RESalphaProxy X$
\UI$ [\alpha] A \fCenter \{\alpha\} \RESalphaProxy X$
\RightLabel{$comp$}
\UI$[\alpha] A \fCenter \Phi_{\alpha} > X$
\UI$\Phi_{\alpha}\,; [\alpha] A \fCenter X$
\UI$ [\alpha] A\,; \Phi_{\alpha} \fCenter X$
\UI$\Phi_{\alpha} \fCenter [\alpha] A > X$
\UI$1_{\alpha} \fCenter [\alpha] A > X$
\UI$[\alpha] A\,; 1_{\alpha} \fCenter X$
\UI$1_{\alpha}\,; [\alpha] A \fCenter X$
\DisplayProof
\qquad
\qquad
\AX$\Phi_{\alpha} \fCenter 1_\alpha$
\AX$X \fCenter 1_\alpha > \{\alpha\} A$
\UI$1_\alpha\,; X \fCenter \{\alpha\} A$
%\UI$X\,; 1_\alpha \fCenter \{\alpha\} A$
\UI$ 1_\alpha \fCenter \{\alpha\} A < X$
\BI$\Phi_{\alpha} \fCenter \{\alpha\} A < X$
%\UI$X\,; \Phi_{\alpha} \fCenter \{\alpha\} A$
\UI$\Phi_{\alpha}\,; X \fCenter \{\alpha\} A$
\UI$X \fCenter \Phi_{\alpha} > \{\alpha\} A$
\RightLabel{$reduce'$}
\UI$ X \fCenter \{ \alpha \} A$
\UI$\RESalphaProxy X \fCenter A$
\UI$\{\alpha\} \RESalphaProxy X \fCenter \langle \alpha \rangle A$
\LeftLabel{$comp$}
\UI$\Phi_\alpha\,; X \fCenter \langle \alpha \rangle A$
\UI$X\,; \Phi_\alpha \fCenter \langle \alpha \rangle A$
\UI$\Phi_\alpha \fCenter X > \langle \alpha \rangle A$
\UI$1_\alpha \fCenter X > \langle \alpha \rangle A$
\UI$X\,; 1_\alpha \fCenter \langle \alpha \rangle A$
\UI$1_\alpha\,; X \fCenter \langle \alpha \rangle A$
\UI$ X \fCenter 1_\alpha > \langle\alpha\rangle A$
\DisplayProof
\]
}}

\noindent The old rules \emph{reduce} are derivable as follows.
{\scriptsize{
\[
\AX$\Phi_{\alpha} \fCenter 1_\alpha$
\AX$1_\alpha\,; \{\alpha\} A \fCenter X$
%\UI$\{\alpha\} A\,; 1_\alpha \fCenter X$
\UI$1_\alpha \fCenter X < \{\alpha\} A$
\BI$\Phi_{\alpha} \fCenter X < \{\alpha\} A$
%\UI$\{\alpha\} A\,;\Phi_\alpha \fCenter X$
\UI$\Phi_\alpha\,; \{\alpha\} A \fCenter X$
\LeftLabel{$reduce'$}
\UI$\{\alpha\} A \fCenter X$
\DisplayProof
\qquad
\qquad
\AX$\Phi_{\alpha} \fCenter 1_\alpha$
\AX$ X \fCenter 1_\alpha > \{\alpha\} A$
\UI$ 1_\alpha\,; X \fCenter \{\alpha\} A$
\UI$X\,; 1_\alpha \fCenter \{\alpha\} A$
\UI$1_\alpha \fCenter X > \{\alpha\} A$
\BI$\Phi_{\alpha} \fCenter X> \{\alpha\} A$
\UI$X\,; \Phi_{\alpha} \fCenter \{\alpha\} A$
\UI$\Phi_{\alpha}\,; X \fCenter \{\alpha\} A$
\UI$X \fCenter \Phi_{\alpha} > \{\alpha\} A$
\RightLabel{$reduce'$}
\UI$X \fCenter \{\alpha\} A$
\DisplayProof
\]
}}

%\AX$Y \fCenter \{\alpha\} \{\aga\} X$
%\RightLabel{\fns\emph{swap-in'}$_{R}$}
%\UI$Y \fCenter \Phi_\alpha > {\{\aga\} \{\beta\}_{\alpha\aga\beta}\,X}$

\noindent The old \emph{swap-in} rules are derivable in the revised calculus from the new \emph{swap-in} rules as follows.
{\scriptsize
\[
\AX$\Phi_\alpha \fCenter 1_\alpha$
\AX$1_\alpha\,; \{\alpha\} \{\aga\} X \fCenter Y$
%\UI$\{\alpha\}\,; 1_\alpha \{\aga\} X \fCenter Y$
\UI$1_\alpha \fCenter Y < \{\alpha\} \{\aga\} X$
\BI$\Phi_\alpha \fCenter Y < \{\alpha\} \{\aga\} X$
%\UI$\{\alpha\} \{\aga\} X\,; \Phi_\alpha \fCenter Y$
\UI$\Phi_\alpha\,; \{\alpha\} \{\aga\} X \fCenter Y$
\LeftLabel{\fns{$reduce'$}}
\UI$\{\alpha\} \{\aga\} X \fCenter Y$
\LeftLabel{\fns{$swap\textrm{-}in'$}}
\UI$\Phi_\alpha\,; \{\aga\}\{\beta\}_{\alpha\aga\beta}\, X \fCenter Y$
%\UI$\{\aga\}\{\beta\}_{\alpha\aga\beta}\, X\,; \Phi_\alpha \fCenter Y$
\UI$\Phi_\alpha \fCenter Y < \{\aga\}\{\beta\}_{\alpha\aga\beta}\, X$
\UI$1_\alpha \fCenter Y < \{\aga\}\{\beta\}_{\alpha\aga\beta}\, X$
%\UI$\{\aga\}\{\beta\}_{\alpha\aga\beta}\, X\,; 1_\alpha \fCenter Y$
\UI$1_\alpha\,; \{\aga\}\{\beta\}_{\alpha\aga\beta}\, X \fCenter Y$
\DisplayProof
\AX$\Phi_\alpha \fCenter 1_\alpha$
\AX$Y \fCenter 1_\alpha > \{\alpha\} \{\aga\} X$
\UI$ 1_\alpha ; Y \fCenter \{\alpha\} \{\aga\} X$
%\UI$ Y ; 1_\alpha \fCenter \{\alpha\} \{\aga\} X$
\UI$ 1_\alpha \fCenter \{\alpha\} \{\aga\} X < Y$
\BI$\Phi_\alpha \fCenter \{\alpha\} \{\aga\} X < Y$
%\UI$ Y ; \Phi_\alpha   \fCenter \{\alpha\} \{\aga\} X $
\UI$\Phi_\alpha\,; Y \fCenter \{\alpha\} \{\aga\} X$
\UI$Y \fCenter \Phi_\alpha > \{\alpha\} \{\aga\} X$
\RightLabel{\fns{$reduce'$}}
\UI$Y \fCenter \{\alpha\} \{\aga\} X$
%\UI$\{\alpha\} \{\aga\} X \fCenter Y$
\RightLabel{\fns{$swap\textrm{-}in'$}}
\UI$Y \fCenter \{\alpha\} \{\aga\} X$
\UI$Y \fCenter \Phi_\alpha > \{\aga\}\{\beta\}_{\alpha\aga\beta}\, X$
\DisplayProof
\]
}

\vspace{5px}

\noindent The old \emph{swap-out} rules (translated into D'.EAK) are derivable using the new \emph{swap-out} rules:

\begin{center}
{\scriptsize
\begin{tabular}{c}
\AXC{$\Phi_\alpha \vdash 1_\alpha$}
\AXC{$1_\alpha\,; \{\aga\}\{\beta_1\}\, X \vdash Y\mid \alpha\aga\beta_1$}
%\UIC{$\{\aga\}\{\beta_1\}\, X\,; 1_\alpha \vdash Y\mid \alpha\aga\beta_1$}
\UIC{$1_\alpha \vdash Y < \{\aga\}\{\beta_1\}\, X\mid \alpha\aga\beta_1$}
\BIC{$\Phi_\alpha \vdash Y < \{\aga\}\{\beta_1\}\, X\mid \alpha\aga\beta_1$}
%\UIC{$\{\aga\}\{\beta_1\}\, X\,; \Phi_\alpha \vdash Y\mid \alpha\aga\beta_1$}
\UIC{$\Phi_\alpha\,; \{\aga\}\{\beta_1\} \,X \vdash Y\mid \alpha\aga\beta_1$}
\LeftLabel{\scriptsize{$reduce'$}}
\UIC{$\{\aga\}\{\beta_1\}\, X \vdash Y\mid \alpha\aga\beta_1$}
%%%%%%%%%%%%%%%%%%%%%%%%%%%%%%%%%%%%%%%%%%%%%%%%%%%%%%%%%%%%%%%%%%%%%%%%%%%%%%%%%
%\def\fCenter{\mbox{$\phantom{\{}\cdots\phantom{\}}$}}
\AXC{$\phantom{\{}\cdots\phantom{\}}$}
\noLine
\UIC{$\phantom{\{}\cdots\phantom{\}}$}
\noLine
\UIC{$\phantom{\{}\cdots\phantom{\}}$}
\noLine
\UIC{$\phantom{\{}\cdots\phantom{\}}$}
\noLine
\UIC{$\phantom{\{}\cdots\phantom{\}}$}
%%%%%%%%%%%%%%%%%%%%%%%%%%%%%%%%%%%%%%%%%%%%%%%%%%%%%%%%%%%%%%%%%%%%%%%%%%%%%%%%%%
\AXC{$\Phi_\alpha \vdash 1_\alpha$}
\AXC{$1_\alpha\,; \{\aga\}\{\beta_n\}\, X \vdash Y\mid \alpha\aga\beta_n$}
%\UIC{$\{\aga\}\{\beta_n\}\, X\,; 1_\alpha \vdash Y\mid \alpha\aga\beta_n$}
\UIC{$1_\alpha \vdash Y < \{\aga\}\{\beta_n\}\, X\mid \alpha\aga\beta_n$}
\BIC{$\Phi_\alpha \vdash Y < \{\aga\}\{\beta_n\}\, X\mid \alpha\aga\beta_n$}
%\UIC{$\{\aga\}\{\beta_n\}\, X\,; \Phi_\alpha \vdash Y\mid \alpha\aga\beta_n$}
\UIC{$\Phi_\alpha\,; \{\aga\}\{\beta_n\} \,X \vdash Y\mid \alpha\aga\beta_n$}
\LeftLabel{\scriptsize{$reduce'$}}
\UIC{$\{\aga\}\{\beta_n\}\, X \vdash Y\mid \alpha\aga\beta_n$}
%%%%%%%%%%%%%%%%%%%%%%%%%%%%%%%%%%%%%%%%%%%%%%%%%%%%%%%%%%%%%%%%%%%%%%%%%%%%%%%%%
\RightLabel{\scriptsize{$swap\textrm{-}out'$}}
\TI$\{\alpha\} \{\aga\} X \fCenter \Bigsemic\Big(Y\mid \alpha\aga\beta\Big)$
\UI$1_\alpha \fCenter \{\alpha\} \{\aga\} X > \Bigsemic\Big(Y\mid \alpha\aga\beta\Big)$
\UI$ \{\alpha\} \{\aga\} X; 1_\alpha \fCenter  \Bigsemic\Big(Y\mid \alpha\aga\beta\Big)$
\UI$1_\alpha ; \{\alpha\} \{\aga\} X \fCenter \Bigsemic\Big(Y\mid \alpha\aga\beta\Big)$
\DisplayProof
\end{tabular}
}
\end{center}
\medskip

\begin{center}
{\scriptsize{
\begin{tabular}{c}
\AXC{$\Phi_\alpha \vdash 1_\alpha$}
\AXC{$Y \vdash 1_\alpha\,> \{\aga\}\{\beta_1\}\, X \mid \alpha\aga\beta_1$}
\UIC{$1_\alpha\, ; Y \vdash \{\aga\}\{\beta_1\}\, X \mid \alpha\aga\beta_1$}
%\UIC{$Y ;1_\alpha \vdash \{\aga\}\{\beta_1\}\, X \mid \alpha\aga\beta_1$}
\UIC{$1_\alpha \vdash \{\aga\}\{\beta_1\}\, X \mid \alpha\aga\beta_1 < Y$}
\BIC{$\Phi_\alpha \vdash \{\aga\}\{\beta_1\}\, X \mid \alpha\aga\beta_1 < Y$}
%\UIC{$Y ; \Phi_\alpha \vdash \{\aga\}\{\beta_1\}\, X \mid \alpha\aga\beta_1$}
\UIC{$\Phi_\alpha\,; Y \vdash \{\aga\}\{\beta_1\}\, X \mid \alpha\aga\beta_1$}
\UIC{$Y \vdash \Phi_\alpha > \{\aga\}\{\beta_1\}\, X \mid \alpha\aga\beta_1$}
\LeftLabel{\scriptsize{$reduce'$}}
\UIC{$Y \vdash \{\aga\}\{\beta_1\}\, X \mid \alpha\aga\beta_1$}
%%%%%%%%%%%%%%%%%%%%%%%%%%%%%%%%%%%%%%%%%%%%%%%%%%%%%%%%%%%%%%%%%%%%%%%%%%%%%%%%%
\AXC{$\phantom{\{}\cdots\phantom{\}}$}
\noLine
\UIC{$\phantom{\{}\cdots\phantom{\}}$}
\noLine
\UIC{$\phantom{\{}\cdots\phantom{\}}$}
\noLine
\UIC{$\phantom{\{}\cdots\phantom{\}}$}
\noLine
\UIC{$\phantom{\{}\cdots\phantom{\}}$}
\noLine
\UIC{$\phantom{\{}\cdots\phantom{\}}$}
\noLine
\UIC{$\phantom{\{}\cdots\phantom{\}}$}
%%%%%%%%%%%%%%%%%%%%%%%%%%%%%%%%%%%%%%%%%%%%%%%%%%%%%%%%%%%%%%%%%%%%%%%%%%%%%%%%%%
\AXC{$\Phi_\alpha \vdash 1_\alpha$}
\AXC{$Y \vdash 1_\alpha\,> \{\aga\}\{\beta_n\}\, X \mid \alpha\aga\beta_n$}
\UIC{$1_\alpha\, ; Y \vdash \{\aga\}\{\beta_n\}\, X \mid \alpha\aga\beta_n$}
%\UIC{$Y ;1_\alpha \vdash \{\aga\}\{\beta_1\}\, X \mid \alpha\aga\beta_1$}
\UIC{$1_\alpha \vdash \{\aga\}\{\beta_n\}\, X \mid \alpha\aga\beta_n < Y$}
\BIC{$\Phi_\alpha \vdash \{\aga\}\{\beta_n\}\, X \mid \alpha\aga\beta_n < Y$}
%\UIC{$Y ; \Phi_\alpha \vdash \{\aga\}\{\beta_1\}\, X \mid \alpha\aga\beta_1$}
\UIC{$\Phi_\alpha\,; Y \vdash \{\aga\}\{\beta_n\}\, X \mid \alpha\aga\beta_n$}
\UIC{$Y \vdash \Phi_\alpha > \{\aga\}\{\beta_n\}\, X \mid \alpha\aga\beta_n$}
\LeftLabel{\scriptsize{$reduce'$}}
\UIC{$Y \vdash \{\aga\}\{\beta_n\}\, X \mid \alpha\aga\beta_n$}
%%%%%%%%%%%%%%%%%%%%%%%%%%%%%%%%%%%%%%%%%%%%%%%%%%%%%%%%%%%%%%%%%%%%%%%%%%%%%%%%%
\RightLabel{\scriptsize{$swap\textrm{-}out'$}}
\TI$\Bigsemic\Big(Y\mid \alpha\aga\beta\Big) \fCenter  \{\alpha\} \{\aga\} X$
\UI$1_\alpha \fCenter \Bigsemic\Big(Y\mid \alpha\aga\beta\Big) >  \{\alpha\} \{\aga\} X $
\UI$ \Bigsemic\Big(Y\mid \alpha\aga\beta\Big) ; 1_\alpha \fCenter    \{\alpha\} \{\aga\} X$
\UI$\Bigsemic\Big(Y\mid \alpha\aga\beta\Big) ; 1_\alpha \fCenter  \{\alpha\} \{\aga\} X$
\UI$\Bigsemic\Big(Y\mid \alpha\aga\beta\Big)  \fCenter   1_\alpha > \{\alpha\} \{\aga\} X$
\DisplayProof
\end{tabular}
}}
\end{center}

%%%%%%%%%%%%%%%%%%%%%%%%%%%%%%%%%%%%%%%%%%%%%%%%%%%%%%%%%%%%%%%%%%%%%%%%%%%%%%%%%%

\subsection{Soundness of $comp$ rules in the final coalgebra}\label{Soundness of comp}

We address the reader to \cite{GKPLori} for details on the final coalgebra semantics for dynamic epistemic logic.

%The none restricted weakening is sound w.r.t.~final coalgebra semantics, because
%if $\phi \subseteq \psi $ and since $\alpha \subseteq \phi \rightarrow \psi$ which is the same as $\alpha \cap \phi \subseteq \psi$. Now $\alpha \cap \phi \subseteq \psi$.

To prove the soundness of the rules above in the final coalgebra it suffices to check that for every formula $A$,

$$[\alpha][\alpha^{-1}] \val{A}_{\bbZ}\subseteq \val{Pre(\alpha)\rightarrow A}_\bbZ \ \mbox{ and }\ \val{Pre(\alpha) \ ;\  A}_\bbZ\subseteq \langle\alpha\rangle\langle\alpha^{-1}\rangle \val{A}_{\bbZ}.$$

We will make use of the following  general fact:
\begin{fact}
\label{fact on rels}
Let $R$ be a binary relation on a set $X$ and let $R^{-1}$ be its converse. Then,
$$[Dom(R)\times Dom(R)]\cap \Delta_X\subseteq R;R^{-1},$$
where $Dom(R) = \{x\in X\mid xRy$ for some $y\in X\}$, and $\Delta_X = \{(x, x)\mid x\in X\}$.
\end{fact}
\begin{proof}
Straightforward.
\end{proof}

\begin{fact}
The following $comp$ rules:
\[
\AxiomC{$Y \vdash \{\alpha\} \RESalphaProxy X$}
\UnaryInfC{$ Y \vdash \underline{Pre} (\alpha) >X$}
\DisplayProof
\qquad
\AxiomC{$\{\alpha\} \RESalphaProxy X \vdash Y$}
\UnaryInfC{$ \underline{Pre} (\alpha)\,; X \vdash Y$}
\DisplayProof
\]

 are sound in the final coalgebra.
\end{fact}
\begin{proof} $\quad$
\begin{center}
\begin{tabular}{r c l l}
$\langle\alpha\rangle\langle\alpha^{-1}\rangle \val{A}_{\bbZ} $ & $=$   &$\alpha^{-1}[\alpha [\val{A}_{\bbZ}] ]$ &\\
& $=$  &$(\alpha;\alpha^{-1}) [\val{A}_{\bbZ}]$ &\\
& $\supseteq$  &$S [\val{A}_{\bbZ}]$ & Fact \ref{fact on rels}\\
& $=$  &$Dom (\alpha)\cap \val{A}_{\bbZ}$ & \\
& $=$  &$\val{Pre (\alpha)\ ;\  A}_{\bbZ},$ & \\

\end{tabular}
\end{center}

\begin{center}
\begin{tabular}{r c l l}
$[\alpha][\alpha^{-1}] \val{A}_{\bbZ}$ & $=$  &$(\alpha^{-1}[  ( [\alpha^{-1}] \val{A}_{\bbZ})^c])^c$ &\\
& $=$  &$(\alpha^{-1}[\alpha [\val{A}_{\bbZ}^c]])^c$ &\\
& $=$  &$((\alpha;\alpha^{-1}) [\val{A}_{\bbZ}^c])^c$ &\\
& $\subseteq$  &$(S [\val{A}_{\bbZ}^c])^c$ & Fact \ref{fact on rels}\\
& $=$  &$(Dom (\alpha)\cap \val{A}_{\bbZ}^c])^c$ & \\
& $=$  &$Dom (\alpha)^c\cup \val{A}_{\bbZ}$ & \\
& $=$  &$\val{Pre (\alpha)\rightarrow A}_{\bbZ},$ & \\

\end{tabular}
\end{center}

where $S = [Dom(R)\times Dom(R)]\cap \Delta_X$.
\end{proof}

% !TEX root =  main.tex 
\section{Cut elimination for D'.EAK}\label{cut in D'.EAK}

In the present section, we report on the remaining cases for the verification of condition $C_8$ for D'.EAK; these cases are needed already for the cut elimination \'a la Gentzen for D.EAK, but do not appear in \cite{GKPLori}.

First we consider the $atom$ rule (see page~\pageref{pageref:atom}).

\begin{center}
{\footnotesize{
\bottomAlignProof
\begin{tabular}{lcr}
\AX$\Gamma p\fCenter p$
\AX$p \fCenter \Delta p$
\BI$\Gamma p \fCenter \Delta p$
\DisplayProof
 & $\rightsquigarrow$ &
\bottomAlignProof
\AX$\Gamma p \fCenter \Delta p$
\DisplayProof
 \\
\end{tabular}
}}
\end{center}

\noindent Now we treat the introductions of the connectives of the propositional base (we also treat here the cases relative to the two additional arrows $\leftarrow$ and $\pdra$ added to our presentation of D.EAK):

\begin{center}
{\scriptsize{
\bottomAlignProof
\begin{tabular}{lcr}
\AX$\textrm{I}\ \fCenter\ \top$
\AXC{\ \ $\vdots$ \raisebox{1mm}{$\pi$}}
\noLine
\UI$\textrm{I}\ \fCenter\ X$
\UI$\top\ \fCenter\ X$
\BI$\textrm{I}\ \fCenter\ X$
\DisplayProof
 & $\rightsquigarrow$ &
\bottomAlignProof
\AXC{\ \ $\vdots$ \raisebox{1mm}{$\pi$}}
\noLine
\UI$\textrm{I}\ \fCenter\ X$
\DisplayProof
 \\
\end{tabular}
}}
\end{center}

\begin{center}
{\scriptsize{
\bottomAlignProof
\begin{tabular}{lcr}
\AXC{\ \ $\vdots$ \raisebox{1mm}{$\pi$}}
\noLine
\UI$ X \ \fCenter\ \textrm{I}$
\UI$ X \fCenter\ \bot$
\AX$\bot \fCenter\ \textrm{I}$
\BI$ X\ \fCenter\ \textrm{I}$
\DisplayProof
 & $\rightsquigarrow$ &
\bottomAlignProof
\AXC{\ \ $\vdots$ \raisebox{1mm}{$\pi$}}
\noLine
\UI$X \fCenter\ \textrm{I}$
\DisplayProof
 \\
\end{tabular}
}}
\end{center}

\begin{center}
{\scriptsize{
\bottomAlignProof
\begin{tabular}{lcr}
\AXC{\ \ \ $\vdots$ \raisebox{1mm}{$\pi_1$}}
\noLine
\UI$X\ \fCenter\ A$
\AXC{\ \ \ $\vdots$ \raisebox{1mm}{$\pi_2$}}
\noLine
\UI$Y\ \fCenter\ B$
\BI$X\,; Y\ \fCenter\ A\wedge B$
\AXC{\ \ \ $\vdots$ \raisebox{1mm}{$\pi_3$}}
\noLine
\UI$A\,; B\ \fCenter\ Z$
\UI$A\wedge B\ \fCenter\ Z$
\BI$X\,; Y\ \fCenter\ Z$
\DisplayProof
 & $\rightsquigarrow$ &
\bottomAlignProof
\AXC{\ \ \ $\vdots$ \raisebox{1mm}{$\pi_2$}}
\noLine
\UI$Y\ \fCenter\ B$
\AXC{\ \ \ $\vdots$ \raisebox{1mm}{$\pi_1$}}
\noLine
\UI$X\ \fCenter\ A$
\AXC{\ \ \ $\vdots$ \raisebox{1mm}{$\pi_3$}}
\noLine
\UI$A\,; B\ \fCenter\ Z$
\UI$B\,; A\ \fCenter\ Z$
\UI$A\ \fCenter\ B > Z$
\BI$X\ \fCenter\ B > Z$
\UI$B\,; X\ \fCenter\ Z$
\UI$X\,; B\ \fCenter\ Z$
\UI$B\ \fCenter\ X > Z$
\BI$Y\ \fCenter\ X > Z$
\UI$X\,; Y\ \fCenter\ Z$
\DisplayProof
 \\
\end{tabular}
}}
\end{center}

%%%%%%%%%%%%%%%%%%%%%%%%%%%%%%%%%%%%%%%%%%%%%%%%%%%%%%%

\begin{center}
{\scriptsize{
\bottomAlignProof
\begin{tabular}{@{}ccc@{}}
\AXC{\ \ \ $\vdots$ \raisebox{1mm}{$\pi_3$}}
\noLine
\UI$Z\ \fCenter\ B\,; A$
\UI$Z\ \fCenter\ B\vee A $
\AXC{\ \ \ $\vdots$ \raisebox{1mm}{$\pi_1$}}
\noLine
\UI$B\ \fCenter\ Y$
\AXC{\ \ \ $\vdots$ \raisebox{1mm}{$\pi_2$}}
\noLine
\UI$A\ \fCenter\ X$
\BI$B\vee A\ \fCenter\ Y\,; X$
\BI$Z\ \fCenter\ Y\,; X$
\DisplayProof
& $\rightsquigarrow$ &
\bottomAlignProof
\AXC{\ \ \ $\vdots$ \raisebox{1mm}{$\pi_3$}}
\noLine
\UI$Z\ \fCenter\ B\,; A$
\UI$Z\ \fCenter\ A\,; B$
\UI$A > Z\ \fCenter\ B$
\AXC{\ \ \ $\vdots$ \raisebox{1mm}{$\pi_1$}}
\noLine
\UI$B\ \fCenter\ Y$
\BI$A > Z\ \fCenter\ Y$
\UI$Z\ \fCenter\ A\,; Y$
\UI$Z\ \fCenter\ Y\,; A$
\UI$Y > Z\ \fCenter\ A$
\AXC{\ \ \ $\vdots$ \raisebox{1mm}{$\pi_2$}}
\noLine
\UI$A\ \fCenter\ X$
\BI$Y > Z\ \fCenter\ X$
\UI$Z\ \fCenter\ Y\,; X$
\DisplayProof
 \\
\end{tabular}
}}
\end{center}
%%%%%%%%%%%%%%%%%%%%%%%%%%%%%%%%%%%%%%%%%%%%%%%%%%%%%%%%%%%%%%%%%%%%%%%%%%%%%%%%%%%%%%%%%%%
%%%%%%%%%%%%%%%%%%%%%%%%%%%%%%%%%%%%%%%%%%%%%%%%%%%%%%%%%%%%%%%%%%%%%%%%%%%%%%%%%%%%%%%%%%%%

\begin{center}
{\scriptsize{
\bottomAlignProof
\begin{tabular}{@{}ccc@{}}
\AXC{\ \ \ $\vdots$ \raisebox{1mm}{$\pi_1$}}
\noLine
\UI$Y\ \fCenter\ A > B$
\UI$Y\ \fCenter\ A\rightarrow B$
\AXC{\ \ \ $\vdots$ \raisebox{1mm}{$\pi_2$}}
\noLine
\UI$X\ \fCenter\ A$
\AXC{\ \ \ $\vdots$ \raisebox{1mm}{$\pi_3$}}
\noLine
\UI$B\ \fCenter\ Z$
\BI$A\rightarrow B\ \fCenter\ X > Z$
\BI$Y\ \fCenter\ X > Z$
\DisplayProof
& $\rightsquigarrow$&
\bottomAlignProof
\AXC{\ \ \ $\vdots$ \raisebox{1mm}{$\pi_2$}}
\noLine
\UI$X\ \fCenter\ A$
\AXC{\ \ \ $\vdots$ \raisebox{1mm}{$\pi_1$}}
\noLine
\UI$Y\ \fCenter\ A > B$
\UI$A\,; Y\ \fCenter\ B$
\AXC{\ \ \ $\vdots$ \raisebox{1mm}{$\pi_3$}}
\noLine
\UI$B\ \fCenter\ Z$
\BI$A\,; Y\ \fCenter\ Z$
\UI$Y\,; A\ \fCenter\ Z$
\UI$A\ \fCenter\ Y > Z$
\BI$X\ \fCenter\ Y > Z$
\UI$Y\,; X\ \fCenter\ Z$
\UI$X\,; Y\ \fCenter\ Z$
\UI$Y\ \fCenter\ X > Z$
\DisplayProof
 \\
\end{tabular}
}}
\end{center}

%%%%%%%%%%%%%%%%%%%%%%%%%%%%%%%%%%%%%%%%%%%%%%%%%%%%%%%

\begin{center}
{\scriptsize{
\bottomAlignProof
\begin{tabular}{@{}ccc@{}}
\AXC{\ \ \ $\vdots$ \raisebox{1mm}{$\pi_1$}}
\noLine
\UI$Y\ \fCenter\ B < A$
\UI$Y\ \fCenter\ B\leftarrow A$
\AXC{\ \ \ $\vdots$ \raisebox{1mm}{$\pi_2$}}
\noLine
\UI$B\ \fCenter\ Z$
\AXC{\ \ \ $\vdots$ \raisebox{1mm}{$\pi_3$}}
\noLine
\UI$X\ \fCenter\ A$
\BI$B\leftarrow A\ \fCenter\ Z < X$
\BI$Y\ \fCenter\ Z< X$
\DisplayProof
&$\rightsquigarrow$ &
\bottomAlignProof
\AXC{\ \ \ $\vdots$ \raisebox{1mm}{$\pi_2$}}
\noLine
\UI$X\ \fCenter\ A$
\AXC{\ \ \ $\vdots$ \raisebox{1mm}{$\pi_1$}}
\noLine
\UI$Y\ \fCenter\ B< A$
\UI$ Y; A\ \fCenter\ B$
\AXC{\ \ \ $\vdots$ \raisebox{1mm}{$\pi_3$}}
\noLine
\UI$B\ \fCenter\ Z$
\BI$ Y; A\ \fCenter\ Z$
\UI$ A; Y\ \fCenter\ Z$
\UI$A\ \fCenter\ Z < Y$
\BI$X\ \fCenter\  Z <Y$
\UI$ X; Y\ \fCenter\ Z$
\UI$ Y; X\ \fCenter\ Z$
\UI$Y\ \fCenter\  Z < X$
\DisplayProof
 \\
\end{tabular}
}}
\end{center}

%%%%%%%%%%%%%%%%%%%%%%%%%%

\begin{center}
{\scriptsize{
\bottomAlignProof
\begin{tabular}{@{}ccc@{}}
\AXC{\ \ \ $\vdots$ \raisebox{1mm}{$\pi_2$}}
\noLine
\UI$A\ \fCenter\ Y$
\AXC{\ \ \ $\vdots$ \raisebox{1mm}{$\pi_3$}}
\noLine
\UI$Z\ \fCenter\ B$
\BI$Y > Z\ \fCenter\ A \pdra B$
\AXC{\ \ \ $\vdots$ \raisebox{1mm}{$\pi_1$}}
\noLine
\UI$A > B\ \fCenter\ X$
\UI$A \pdra B\ \fCenter\ X$
\BI$Y > Z\ \fCenter\ X$
\DisplayProof
&$\rightsquigarrow$&
\bottomAlignProof
\AXC{\ \ \ $\vdots$ \raisebox{1mm}{$\pi_3$}}
\noLine
\UI$Z\ \fCenter\ B$
\AXC{\ \ \ $\vdots$ \raisebox{1mm}{$\pi_1$}}
\noLine
\UI$A > B\ \fCenter\ X$
\UI$B\ \fCenter\ A\,; X$
\BI$Z\ \fCenter\ A\,; X$
\UI$Z\ \fCenter\ X\,; A$
\UI$X > Z\ \fCenter\ A$
\AXC{\ \ \ $\vdots$ \raisebox{1mm}{$\pi_2$}}
\noLine
\UI$A\ \fCenter\ Y$
\BI$X > Z\ \fCenter\ Y$
\UI$Z\ \fCenter\ X\,; Y$
\UI$Z\ \fCenter\ Y\,; X$
\UI$Y > Z\ \fCenter\ X$
\DisplayProof
 \\
\end{tabular}
}}
\end{center}

\begin{center}
{\scriptsize{
\bottomAlignProof
\begin{tabular}{@{}ccc@{}}
\AXC{\ \ \ $\vdots$ \raisebox{1mm}{$\pi_2$}}
\noLine
\UI$Y\ \fCenter\ B$
\AXC{\ \ \ $\vdots$ \raisebox{1mm}{$\pi_3$}}
\noLine
\UI$A\ \fCenter\ Z$
\BI$Y < Z\ \fCenter\ B \pdla A$
\AXC{\ \ \ $\vdots$ \raisebox{1mm}{$\pi_1$}}
\noLine
\UI$B < A \ \fCenter\ X$
\UI$B \pdla A\ \fCenter\ X$
\BI$Y < Z\ \fCenter\ X$
\DisplayProof
 &$\rightsquigarrow$&
\bottomAlignProof
\AXC{\ \ \ $\vdots$ \raisebox{1mm}{$\pi_3$}}
\noLine
\UI$Y\ \fCenter\ B$
\AXC{\ \ \ $\vdots$ \raisebox{1mm}{$\pi_1$}}
\noLine
\UI$B < A \ \fCenter\ X$
\UI$B\ \fCenter\  X; A$
\BI$Y\ \fCenter\  X; A$
\UI$Y\ \fCenter\ A; X$
\UI$ Y < X\ \fCenter\ A$
\AXC{\ \ \ $\vdots$ \raisebox{1mm}{$\pi_2$}}
\noLine
\UI$A\ \fCenter\ Z$
\BI$Y < X\ \fCenter\ Z$
\UI$Y\ \fCenter\ Z\,; X$
\UI$Y\ \fCenter\ X\,; Z$
\UI$Y < Z\ \fCenter\ Y$
\DisplayProof
 \\
\end{tabular}
}}
\end{center}

\noindent Now we turn to the part of D'.EAK with static modalities. We omit  the proofs for  $\RESagaDia$ and $\RESagaBox$ , because they analogous to the transformations of  $\langle \aga \rangle$ and $[\aga]$.

\begin{center}
{\scriptsize{
\begin{tabular}{@{}rcl@{}}
\bottomAlignProof
\AXC{\ \ \ $\vdots$ \raisebox{1mm}{$\pi_1$}}
\noLine
\UI$X\ \fCenter\ A$
\UI$\{ \aga\} X\ \fCenter\ \langle \aga \rangle A$
\AXC{\ \ \ $\vdots$ \raisebox{1mm}{$\pi_2$}}
\noLine
\UI$\{ \aga\} A\ \fCenter\ Y$
\UI$\langle \aga \rangle A\ \fCenter\ Y$
\BI$\{\aga\} X\ \fCenter\ Y$
\DisplayProof
 & $\rightsquigarrow$ &
\bottomAlignProof
\AXC{\ \ \ $\vdots$ \raisebox{1mm}{$\pi_1$}}
\noLine
\UI$X\ \fCenter\ A$
\AXC{\ \ \,$\vdots$ \raisebox{1mm}{$\pi_2$}}
\noLine
\UI$\{ \aga \} A\ \fCenter\ Y$
\UI$A\ \fCenter\ \RESagaProxy Y$
\BI$X\ \fCenter\ \RESagaProxy Y$
\UI$\{\aga\} X\ \fCenter\ Y$
\DisplayProof
 \\
\end{tabular}
}}
\end{center}

\begin{center}
{\scriptsize{
\begin{tabular}{@{}rcl@{}}
\bottomAlignProof
\AXC{\ \ \ $\vdots$ \raisebox{1mm}{$\pi_1$}}
\noLine
\UI$X\ \fCenter\ \{\aga\} A$
\UI$X\ \fCenter\ [\aga] A$
\AXC{\ \ \ $\vdots$ \raisebox{1mm}{$\pi_2$}}
\noLine
\UI$A\ \fCenter\ Y$
\UI$[\aga] A\ \fCenter\ \{\aga\} Y$
\BI$X \fCenter \{\aga\} Y$
\DisplayProof
 & $\rightsquigarrow$ &
\bottomAlignProof
\AXC{\ \ \ $\vdots$ \raisebox{1mm}{$\pi_1$}}
\noLine
\UI$X\ \fCenter\ \{\aga\} A$
\UI$\RESagaProxy X \fCenter A$
\AXC{\ \ \ $\vdots$ \raisebox{1mm}{$\pi_2$}}
\noLine
\UI$A\ \fCenter\ Y$
\BI$\RESagaProxy X\ \fCenter\ Y$
\UI$X\ \fCenter\ \{\aga\}Y$
\DisplayProof
 \\
\end{tabular}
}}
\end{center}

\noindent The transformations of the dynamic modalities are analogous to the ones of static modalities and, again, we only show them for $\langle \alpha \rangle$ and $[\alpha]$.
\begin{center}
{\scriptsize{
\begin{tabular}{@{}rcl@{}}
\bottomAlignProof
\AXC{\ \ \ $\vdots$ \raisebox{1mm}{$\pi_1$}}
\noLine
\UI$X \fCenter A$
\UI$\{\alpha\} X\ \fCenter\ \lc\alpha\rc A$
\AXC{\ \ \ $\vdots$ \raisebox{1mm}{$\pi_2$}}
\noLine
\UI$\{\alpha\} A\ \fCenter\ Y$
\UI$\lc\alpha\rc A\ \fCenter\ Y$
\BI$\{\alpha\} X\ \fCenter\ Y$
\DisplayProof
 & $\rightsquigarrow$ &
\!\!\!\!\!\!
\bottomAlignProof
\AXC{\ \ \ $\vdots$ \raisebox{1mm}{$\pi_1$}}
\noLine
\UI$X\ \fCenter\ A$
\AXC{\ \ \ $\vdots$ \raisebox{1mm}{$\pi_2$}}
\noLine
\UI$\{\alpha\} A\ \fCenter\ Y$
\UI$A\ \fCenter\ \RESalphaProxy Y$
\BI$X\ \fCenter\ \RESalphaProxy Y$
\UI$\{\alpha\} X\ \fCenter\ Y$
\DisplayProof
 \\
\end{tabular}
}}
\end{center}

\begin{center}
{\scriptsize{
\begin{tabular}{@{}rcl@{}}
\bottomAlignProof
\AXC{\ \ \ $\vdots$ \raisebox{1mm}{$\pi_1$}}
\noLine
\UI$X\ \fCenter\ \{\alpha\} A$
\UI$X\ \fCenter\ [\alpha] A$
\AXC{\ \ \ $\vdots$ \raisebox{1mm}{$\pi_2$}}
\noLine
\UI$A\ \fCenter\ Y$
\UI$[\alpha] A\ \fCenter\ \{\alpha\} Y$
\BI$X \fCenter \{\alpha\} Y$
\DisplayProof
 & $\rightsquigarrow$ &
\bottomAlignProof
\AXC{\ \ \ $\vdots$ \raisebox{1mm}{$\pi_1$}}
\noLine
\UI$X\ \fCenter\ \{\alpha\} A$
\UI$\RESalphaProxy X \fCenter A$
\AXC{\ \ \ $\vdots$ \raisebox{1mm}{$\pi_2$}}
\noLine
\UI$A\ \fCenter\ Y$
\BI$\RESalphaProxy X\ \fCenter\ Y$
\UI$X\ \fCenter\ \{\alpha\}Y$
\DisplayProof
 \\
\end{tabular}
}}
\end{center}

% !TEX root =  main.tex
\section{Completeness of D'.EAK}\label{ssec: completeness}

To prove, indirectly, the completeness of D'.EAK it is enough to show that all the axioms and rules of IEAK are theorems and, respectively, derived or admissible rules of D'.EAK. Below we show the derivations of the dynamic axioms and we leave the remaining axioms and rules to the reader.

\noindent
$\rule{126.2mm}{0.5pt}$
$\lc\alpha\rc\, p \dashv\vdash 1_\alpha \pand p$

\begin{center}
{\scriptsize{
\begin{tabular}{@{}lr@{}}
\AX$\Phi_\alpha \fCenter 1_\alpha$
\AX$\{\alpha\}\, p \fCenter p$
\BI$\Phi_\alpha\,; \{\alpha\}\, p\fCenter 1_\alpha \pand p$
\LeftLabel{\scriptsize{$reduce'$}}
\UI$\{\alpha\}\, p \fCenter 1_\alpha \pand p$
\UI$\lc\alpha\rc\, p \fCenter 1_\alpha \pand p$
\DisplayProof

 & 

\AX$\RESalphaProxy\, p \fCenter p$
\UI$\{\alpha\} \RESalphaProxy\, p \fCenter \lc\alpha\rc\, p$
\LeftLabel{\scriptsize{$comp$}}
\UI$\Phi_\alpha\,; p \fCenter \lc\alpha\rc\, p$
\UI$\Phi_\alpha \fCenter \lc\alpha\rc\, p < p$
\UI$1_\alpha \fCenter \lc\alpha\rc\, p < p$
\UI$1_\alpha\,; p \fCenter \lc\alpha\rc\, p$
\UI$1_\alpha \pand p \fCenter \lc\alpha\rc\, p$
\DisplayProof
 \\
\end{tabular}
}}
\end{center}

\noindent
$\rule{126.2mm}{0.5pt}$
$\ls\alpha\rs\, p \dashv\vdash 1_\alpha \pra p$

\begin{center}
{\scriptsize{
\begin{tabular}{@{}lr@{}}
\AX$p \fCenter \RESalphaProxy\, p$
\UI$\ls\alpha\rs\, p \fCenter \{\alpha\} \RESalphaProxy\, p$
\RightLabel{\scriptsize{$comp$}}
\UI$\ls\alpha\rs\, p \fCenter \Phi_{\alpha} > p$
\UI$\Phi_{\alpha}\,; \ls\alpha\rs\, p \fCenter p$
\UI$\Phi_{\alpha} \fCenter p < \ls\alpha\rs\, p$
\UI$1_{\alpha} \fCenter p < \ls\alpha\rs\, p$
\UI$1_{\alpha}\,; \ls\alpha\rs\, p \fCenter p$
\UI$\ls\alpha\rs\, p \fCenter 1_\alpha > p$
\UI$\ls\alpha\rs\, p \fCenter 1_\alpha \pra p$
\DisplayProof

\ \ \ \ \  & \ \ \ \ \ 

\AX$\Phi_\alpha \fCenter 1_\alpha$
\AX$p \fCenter \{\alpha\}\, p$
\BI$1_\alpha \rightarrow p \fCenter \Phi_\alpha > \{\alpha\}\, p$
\RightLabel{\scriptsize{$reduce'$}}
\UI$1_\alpha \rightarrow p \fCenter \{\alpha\}\, p$
\UI$1_\alpha\rightarrow p \fCenter \ls\alpha\rs\, p$
\DisplayProof
 \\
\end{tabular}
}}
\end{center}

\noindent
$\rule{126.2mm}{0.5pt}$
$\lc\alpha\rc \top \dashv\vdash 1_\alpha$

\begin{center}
{\scriptsize{
\begin{tabular}{@{}lcr@{}}
\AX$\Phi_\alpha \fCenter 1_\alpha$
\UI$\textrm{I} \fCenter 1_\alpha < \Phi_\alpha$
\LeftLabel{\scriptsize{$nec$}}
\UI$\{\alpha\}\, \textrm{I} \fCenter 1_\alpha < \Phi_\alpha$

\UI$\textrm{I} \fCenter \RESalphaProxy\, ( 1_\alpha < \Phi_\alpha )$
\UI$\top \fCenter \RESalphaProxy\, ( 1_\alpha < \Phi_\alpha)$
\UI$\{\alpha\} \top \fCenter 1_\alpha < \Phi_\alpha$
\UI$\Phi_\alpha\,; \{\alpha\} \top \fCenter 1_\alpha$
\LeftLabel{\scriptsize{$reduce'$}}
\UI$\{\alpha\} \top \fCenter 1_\alpha$
\UI$\lc\alpha\rc \top \fCenter 1_\alpha$
\DisplayProof

 & &

\AX$\textrm{I} \fCenter \top$
\LeftLabel{\scriptsize{$nec$}}
\UI$\RESalphaProxy\, \textrm{I} \fCenter \top$
\UI$\{\alpha\} \RESalphaProxy\, \textrm{I} \fCenter \lc\alpha\rc \top$
\LeftLabel{\scriptsize{$comp$}}
\UI$\Phi_{\alpha}\,; \textrm{I} \fCenter \lc\alpha\rc \top$
\UI$\Phi_{\alpha} \fCenter \lc\alpha\rc \top$
\UI$1_{\alpha} \fCenter \lc\alpha\rc \top$
\DisplayProof
 \\
\end{tabular}
}}
\end{center}

\newpage

\noindent
$\rule{126.2mm}{0.5pt}$
$\ls\alpha\rs \bot \dashv\vdash \neg 1_\alpha$

\begin{center}
{\scriptsize{
\begin{tabular}{@{}lr@{}}
\AX$\bot \fCenter \textrm{I}$
\RightLabel{\scriptsize{$nec$}}
\UI$\bot \fCenter \RESalphaProxy\, \textrm{I}$
\UI$\ls\alpha\rs \bot \fCenter \{\alpha\} \RESalphaProxy\, \textrm{I}$
\RightLabel{\scriptsize{$comp$}}
\UI$\ls\alpha\rs \bot \fCenter \Phi_{\alpha} > \textrm{I}$
\UI$\Phi_{\alpha}\,; \ls\alpha\rs \bot \fCenter \textrm{I}$
\UI$\Phi_{\alpha}\,; \ls\alpha\rs \bot \fCenter \bot$
\UI$\Phi_{\alpha} \fCenter \bot < \ls\alpha\rs \bot$
\UI$1_{\alpha} \fCenter \bot < \ls\alpha\rs \bot$
\UI$1_{\alpha}\,; \ls\alpha\rs \bot \fCenter \bot$
\UI$\ls\alpha\rs \bot \fCenter 1_\alpha > \bot$
\UI$\ls\alpha\rs \bot \fCenter 1_\alpha \pra \bot$
\UI$\ls\alpha\rs \bot \fCenter \neg 1_\alpha$
\DisplayProof

 & 

\AX$\Phi_\alpha \fCenter 1_\alpha$
\AX$\bot \fCenter \textrm{I}$
\RightLabel{\scriptsize{$nec$}}
\UI$\bot \fCenter \{\alpha\} \textrm{I}$
\UI$\RESalphaProxy \bot \fCenter \textrm{I}$
\UI$\RESalphaProxy \bot \fCenter \bot$
\UI$\bot \fCenter \{\alpha\} \bot$
\BI$1_\alpha\rightarrow \bot \fCenter \Phi_\alpha > \{\alpha\} \bot$
\UI$\neg 1_\alpha \fCenter \Phi_\alpha > \{\alpha\} \bot$
\RightLabel{\scriptsize{$reduce'$}}
\UI$\neg 1_\alpha \fCenter \{\alpha\} \bot$
\UI$\neg 1_\alpha \fCenter \ls\alpha\rs \bot$
\DisplayProof
 \\
\end{tabular}
}}
\end{center}

\noindent
$\rule{126.2mm}{0.5pt}$
$\lc\alpha\rc \bot \dashv\vdash \bot$

\begin{center}
{\scriptsize{
\begin{tabular}{@{}lr@{}}
\AX$\bot \fCenter \textrm{I}$
\RightLabel{\scriptsize{$nec$}}
\UI$\bot \fCenter \RESalphaProxy\, \textrm{I}$
\UI$\{\alpha\} \bot \fCenter \textrm{I}$
\UI$\{\alpha\} \bot \fCenter \bot$
\UI$\lc\alpha\rc \bot \fCenter \bot$
\DisplayProof

 & 

\AX$\bot \fCenter \textrm{I}$
\UI$\bot \fCenter \lc\alpha\rc \bot$
\DisplayProof
 \\
\end{tabular}
}}
\end{center}

\noindent
$\rule{126.2mm}{0.5pt}$
$\ls\alpha\rs \top \dashv\vdash \top$

\begin{center}
{\scriptsize{
\begin{tabular}{@{}lr@{}}
\AX$\textrm{I} \fCenter \top$
\UI$\ls\alpha\rs \top \fCenter \top$
\DisplayProof

 & 

\AX$\textrm{I} \fCenter \top$
\LeftLabel{\scriptsize{$nec$}}
\UI$\RESalphaProxy \textrm{I} \fCenter \top$
\UI$\textrm{I} \fCenter  \{\alpha\}\top$
\UI$\top \fCenter \{\alpha\} \top$
\UI$\top \fCenter \ls\alpha\rs \top$
\DisplayProof
 \\
\end{tabular}
}}
\end{center}

\noindent
$\rule{126.2mm}{0.5pt}$
$[\alpha] (A\pand B) \dashv\vdash [\alpha] A \pand [\alpha] B$

\begin{center}
{\scriptsize{
\begin{tabular}{@{}lr@{}}
\AX$A \fCenter A$
\UI$A\,; B \fCenter A$
\UI$A\pand B \fCenter A$
\UI$[\alpha] (A\pand B) \fCenter \{\alpha\} A$
\UI$[\alpha] (A\pand B) \fCenter [\alpha] A$
\AX$B \fCenter B$
\UI$A\,; B \fCenter B$
\UI$A\pand B \fCenter B$
\UI$[\alpha] (A\pand B) \fCenter \{\alpha\} B$
\UI$[\alpha] (A\pand B) \fCenter [\alpha] B$
\BI$[\alpha] (A\pand B)\,; [\alpha] (A\pand B) \fCenter [\alpha] A \pand [\alpha] B$
\UI$[\alpha] (A\pand B) \fCenter [\alpha] A \pand [\alpha] B$
\DisplayProof

 & 

\AX$A \fCenter A$
\UI$[\alpha] A \fCenter \{\alpha\} A$
%\LeftLabel{\tiny{adj}}
\UI$\RESalphaProxy [\alpha] A \fCenter A$
\AX$B \fCenter B$
\UI$[\alpha] B \fCenter \{\alpha\} B$
%\LeftLabel{\tiny{adj}}
\UI$\RESalphaProxy [\alpha] B \fCenter B$
\BI$\RESalphaProxy [\alpha] A\,; \RESalphaProxy [\alpha] B \fCenter A\pand B$
\LeftLabel{\scriptsize{$mon$}}
\UI$\RESalphaProxy ([\alpha] A\,; [\alpha] B) \fCenter A \pand B$
%\RightLabel{\tiny{adj}}
\UI$[\alpha] A\,; [\alpha] B \fCenter \{\alpha \} (A \pand B)$
\UI$[\alpha] A\,; [\alpha] B \fCenter \ls\alpha\rs (A \pand B)$
\UI$[\alpha] A \pand [\alpha] B \fCenter \ls\alpha\rs (A\pand B)$
\DisplayProof
 \\
\end{tabular}}}
\end{center}

\newpage

\noindent
$\rule{126.2mm}{0.5pt}$
$\lc\alpha\rc (A \pand B) \dashv\vdash \lc\alpha\rc A \pand \lc\alpha\rc B$

\begin{center}
{\scriptsize{
\begin{tabular}{@{}lr@{}}
\AX$A \fCenter A$
\UI$A\,; B \fCenter A$
\UI$A \pand B \fCenter A$
\UI$\{\alpha\} A \pand B \fCenter \lc\alpha\rc A$
\UI$\lc\alpha\rc (A \pand B) \fCenter \lc\alpha\rc A$
\AX$B \fCenter B$
\UI$A\,; B \fCenter B$
\UI$A \pand B \fCenter B$
\UI$\{\alpha\} A \pand B \fCenter \lc\alpha\rc B$
\UI$\lc\alpha\rc (A \pand B) \fCenter \lc\alpha\rc B$
\BI$\lc\alpha\rc (A \pand B)\,; \lc\alpha\rc (A\wedge B) \fCenter \lc\alpha\rc A \wedge \lc\alpha\rc B$
\UI$\lc\alpha\rc (A \pand B) \fCenter \lc\alpha\rc A \wedge \lc\alpha\rc B$
\DisplayProof

 & 

\AX$A \fCenter A$
\LeftLabel{\scriptsize{\emph{balance}}}
\UI$\{\alpha\} A \fCenter \{\alpha\} A$
\UI$\RESalphaProxy \{\alpha\} A \fCenter A$
\AX$B \fCenter B$
\RightLabel{\scriptsize{\emph{balance}}}
\UI$\{\alpha\} B \fCenter \{\alpha\} B$
\UI$\RESalphaProxy \{\alpha\} B \fCenter B$
\BI$\RESalphaProxy \{\alpha\} A\,; \RESalphaProxy \{\alpha\} B \fCenter A \pand B$
\LeftLabel{\scriptsize{\emph{mon}}}
\UI$\RESalphaProxy (\{\alpha\} A\,; \{\alpha\} B) \fCenter A \pand B$
\UI$\{\alpha\}\RESalphaProxy (\{\alpha\} A\,; \{\alpha\} B) \fCenter \lc \alpha\rc (A \pand B)$
\LeftLabel{\scriptsize{\emph{comp}}}
\UI$\Phi_\alpha\,; (\{\alpha\} A\,; \{\alpha\} B) \fCenter \lc\alpha\rc (A \pand B)$
\UI$(\Phi_\alpha\,; \{\alpha\} A)\,; \{\alpha\} B \fCenter \lc\alpha\rc (A \pand B)$
\UI$\Phi_\alpha\,; \{\alpha\} A \fCenter \lc\alpha\rc (A \pand B) < \{\alpha\} B$
\LeftLabel{\scriptsize{$reduce'$}}
\UI$\{\alpha\} A \fCenter \lc\alpha\rc (A \pand B) < \{\alpha\} B$
\UI$\lc\alpha\rc A \fCenter \lc\alpha\rc (A \pand B) < \{\alpha\} B$
\UI$\lc\alpha\rc A\,; \{\alpha\} B \fCenter \lc\alpha\rc (A \pand B)$
\UI$\{\alpha\} B \fCenter \lc\alpha\rc A > \lc\alpha\rc (A \pand B)$
\UI$\lc\alpha\rc B \fCenter \lc\alpha\rc A > \lc\alpha\rc (A \pand B)$
\UI$\lc\alpha\rc A\,; \lc\alpha\rc B \fCenter \lc\alpha\rc (A \pand B)$
\UI$\lc\alpha\rc A \pand \lc\alpha\rc B \fCenter \lc\alpha\rc (A \pand B)$
\DisplayProof
 \\
\end{tabular}
}}
\end{center}

\noindent
$\rule{126.2mm}{0.5pt}$
$\lc\alpha\rc (A \por B) \dashv\vdash \lc\alpha\rc A \por \lc\alpha\rc B$

\begin{center}
{\scriptsize{
\begin{tabular}{@{}lr@{}}
\AX$A \fCenter A$
\UI$\{\alpha\} A \fCenter \lc\alpha\rc A$
\UI$A \fCenter \RESalphaProxy \lc\alpha\rc A$
\AX$B \fCenter B$
\UI$\{\alpha\} B \fCenter\lc\alpha\rc B$
\UI$B \fCenter \RESalphaProxy \lc\alpha\rc B$
\BI$A \por B \fCenter \RESalphaProxy \lc\alpha\rc A\,; \RESalphaProxy \lc\alpha\rc B$
\UI$A \por B \fCenter \RESalphaProxy (\lc\alpha\rc A\,; \lc\alpha\rc B)$
\UI$\{\alpha\} A \por B \fCenter \lc\alpha\rc A\,; \lc\alpha\rc B$
\UI$\lc\alpha\rc (A \por B) \fCenter \lc\alpha\rc A\,; \lc\alpha\rc B$
\UI$\lc\alpha\rc (A \por B) \fCenter \lc\alpha\rc A \por \lc\alpha\rc B$
\DisplayProof

 & 

\AX$A \fCenter A$
%\UI$A > A \fCenter B$
\UI$A \fCenter A\,; B$
\UI$A \fCenter A \por B$
\UI$\{\alpha\} A \fCenter \lc\alpha\rc (A \por B)$
\UI$\lc\alpha\rc A \fCenter \lc\alpha\rc (A \por B)$
\AX$B \fCenter B$
%\UI$B < B\fCenter A $
\UI$B \fCenter A\,; B$
\UI$B \fCenter A \por B$
\UI$\{\alpha\} B \fCenter \lc\alpha\rc (A \por B)$
\UI$\lc\alpha\rc B \fCenter \lc\alpha\rc (A \por B)$
\BI$\lc\alpha\rc A \por \lc\alpha\rc B \fCenter \lc\alpha\rc (A \por B)\,; \lc\alpha\rc (A \por B)$
\UI$\lc\alpha\rc A \por \lc\alpha\rc B \fCenter \lc\alpha\rc (A \por B)$
\DisplayProof
 \\
\end{tabular}
}}
\end{center}

\noindent
$\rule{126.2mm}{0.5pt}$
$\ls\alpha\rs (A \por B) \dashv\vdash 1_\alpha \pra (\lc\alpha\rc A \por \lc\alpha\rc B)$

\begin{center}
{\scriptsize{
\begin{tabular}{@{}lr@{}}
\AX$A \fCenter A$
\UI$\{\alpha\} A \fCenter \lc\alpha\rc A$
\UI$A \fCenter \RESalphaProxy \lc\alpha\rc A$
\AX$B \fCenter B$
\UI$\{\alpha\} B \fCenter \lc\alpha\rc B$
\UI$B \fCenter \RESalphaProxy \lc\alpha\rc B$
\BI$A\por B \fCenter \RESalphaProxy \lc\alpha\rc A\,; \RESalphaProxy \lc\alpha\rc B$
\UI$A\por B \fCenter \RESalphaProxy (\lc\alpha\rc A\,; \lc\alpha\rc B)$
\UI$[\alpha] (A\por B) \fCenter \{\alpha\} \RESalphaProxy (\lc\alpha\rc A\por \lc\alpha\rc B)$
\RightLabel{\scriptsize{$comp$}}
\UI$[\alpha] (A\por B) \fCenter \Phi_\alpha > (\lc\alpha\rc A\por \lc\alpha\rc B)$
\UI$\Phi_\alpha\,; \ls\alpha\rs (A\por B) \fCenter \lc\alpha\rc A\por \lc\alpha\rc B$
\UI$\Phi_\alpha \fCenter \lc\alpha\rc A\por \lc\alpha\rc B < \ls\alpha\rs (A\por B)$
\UI$1_\alpha\, \fCenter  \lc\alpha\rc A\por \lc\alpha\rc B < \ls\alpha\rs (A\por B)$
\UI$1_\alpha\,; \ls\alpha\rs (A\por B) \fCenter \lc\alpha\rc A\por \lc\alpha\rc B$
\UI$\ls\alpha\rs (A\por B) \fCenter 1_\alpha > (\lc\alpha\rc A\por \lc\alpha\rc B)$
\UI$\ls\alpha\rs (A\por B) \fCenter 1_\alpha \rightarrow (\lc\alpha\rc A\por \lc\alpha\rc B)$
\DisplayProof

 &
\!\!\!\!
\AX$\Phi_\alpha \fCenter 1_\alpha$
\AX$A \fCenter A$
\UI$\{\alpha\} A \fCenter \{\alpha\} A$
\UI$\lc\alpha\rc A \fCenter \{\alpha\} A$
\AX$B \fCenter B$
\UI$\{\alpha\} B \fCenter \{\alpha\} B$
\UI$\lc\alpha\rc B \fCenter \{\alpha\} B$
\BI$\lc\alpha\rc A \por \lc\alpha\rc B \fCenter \{\alpha\} A\,; \{\alpha\} B$
\UI$\lc\alpha\rc A \por \lc\alpha\rc B \fCenter \{\alpha\} (A\,; B)$
\UI$\RESalphaProxy (\lc\alpha\rc A \por \lc\alpha\rc B) \fCenter A\,; B$
\UI$\RESalphaProxy (\lc\alpha\rc A \por \lc\alpha\rc B) \fCenter A \por B$
\UI$\lc\alpha\rc A \por \lc\alpha\rc B \fCenter \{\alpha\} (A \por B)$
\BI$1_\alpha \pra (\lc\alpha\rc A \por \lc\alpha\rc B) \fCenter \Phi_\alpha > \{\alpha\} (A \por B)$
\RightLabel{\scriptsize{$reduce'$}}
\UI$1_\alpha \pra (\lc\alpha\rc A \por \lc\alpha\rc B) \fCenter \{\alpha\} (A \por B)$
\UI$1_\alpha \pra (\lc\alpha\rc A \por \lc\alpha\rc B) \fCenter \ls\alpha\rs (A \por B)$
\DisplayProof
 \\
\end{tabular}
}}
\end{center}

\noindent
$\rule{126.2mm}{0.5pt}$
$\lc\alpha\rc (A \pra B) \dashv\vdash 1_\alpha \pand (\lc\alpha\rc A \pra \lc\alpha\rc B)$

\begin{center}
{\scriptsize{
\begin{tabular}{@{}lr@{}}
\!\!\!\!\!\!\!\!\!\!\!\!\!\!\!
\AX$\Phi_\alpha \fCenter 1_\alpha$
\AX$A \fCenter A$
\UI$\{\alpha\} A \fCenter \{\alpha\} A$
\UI$\lc\alpha\rc A \fCenter \{\alpha\} A$
\UI$\RESalphaProxy \lc\alpha\rc A \fCenter A$
\AX$B \fCenter B$
\UI$\{\alpha\} B \fCenter \lc\alpha\rc B$
\UI$B \fCenter \RESalphaProxy \lc\alpha\rc B$
\BI$A \pra B \fCenter \RESalphaProxy \lc\alpha\rc A > \RESalphaProxy \lc\alpha\rc B$
\UI$A \pra B \fCenter \RESalphaProxy (\lc\alpha\rc A > \lc\alpha\rc B)$
\UI$\{\alpha\} (A \pra B) \fCenter \lc\alpha\rc A > \lc\alpha\rc B$
\UI$\{\alpha\} (A \pra B) \fCenter \lc\alpha\rc A \pra \lc\alpha\rc B$
\BI$\Phi_\alpha\,; \{\alpha\} (A \pra B) \fCenter 1_\alpha \pand (\lc\alpha\rc A \pra \lc\alpha\rc B)$
\LeftLabel{\scriptsize{$reduce'$}}
\UI$\{\alpha\} (A \pra B )\fCenter 1_\alpha \pand (\lc\alpha\rc A \pra \lc\alpha\rc B)$
\UI$\lc\alpha\rc (A \pra B) \fCenter 1_\alpha \pand (\lc\alpha\rc A \pra \lc\alpha\rc B)$
\DisplayProof

 &

\!\!\!\!\!\!\!\!\!\!\!\!\!\!\!\!\!\!\!\!\!\!\!
\AX$A \fCenter A$
\UI$\{\alpha\} A \fCenter \lc\alpha\rc A$
\AX$B \fCenter B$
%\RightLabel{\scriptsize{balance}}
\UI$\{\alpha\} B \fCenter \{\alpha\} B$
\UI$\lc\alpha\rc B \fCenter \{\alpha\} B$
\BI$\lc\alpha\rc A \pra \lc\alpha\rc B \fCenter \{\alpha\} A > \{\alpha\} B$
\UI$\lc\alpha\rc A \pra \lc\alpha\rc B \fCenter \{\alpha\} (A > B)$
\UI$\RESalphaProxy ( \lc\alpha\rc A \pra \lc\alpha\rc B )\fCenter A > B$
\UI$\RESalphaProxy (\lc\alpha\rc A \pra \lc\alpha\rc B )\fCenter A \pra B$
\UI$\{\alpha\} \RESalphaProxy (\lc\alpha\rc A \pra \lc\alpha\rc B) \fCenter \lc \alpha \rc (A \pra B)$
\LeftLabel{\scriptsize{$comp_L^\alpha$}}
\UI$\Phi_\alpha\,;\ ( \lc\alpha\rc A \pra \lc\alpha\rc B ) \fCenter \lc \alpha \rc(A \pra B)$
\UI$ (\lc\alpha\rc A \pra \lc\alpha\rc B )\ ;\ \Phi_\alpha \fCenter \lc \alpha \rc(A \pra B)$
\UIC{$ \Phi_\alpha \fCenter ( \lc\alpha\rc A \pra \lc\alpha\rc B ) > \lc \alpha \rc(A \pra B)$}
\UIC{$ 1_\alpha \fCenter ( \lc\alpha\rc A \pra \lc\alpha\rc B)  > \lc \alpha \rc (A \pra B)$}
\UI$( \lc\alpha\rc A \pra \lc\alpha\rc B )  \ ;\  1_\alpha \fCenter \lc \alpha \rc (A \pra B)$
\UI$1_\alpha\,; (\lc\alpha\rc A\rightarrow \lc\alpha\rc B )\fCenter \lc\alpha\rc (A \pra B)$
\UI$1_\alpha \pand (\lc\alpha\rc A \pra \lc\alpha\rc B) \fCenter \lc\alpha\rc (A \pra B)$
\DisplayProof
 \\
\end{tabular}
}}
\end{center}

\noindent
$\rule{126.2mm}{0.5pt}$
$\ls\alpha\rs (A \pra B) \dashv\vdash \lc\alpha\rc A \pra \lc\alpha\rc B$

\begin{center}
{\scriptsize{
\begin{tabular}{@{}lr@{}}
\AX$A \fCenter A$
\UI$\{\alpha\} A \fCenter \{\alpha\} A$
\UI$\RESalphaProxy \{\alpha\} A \fCenter A$
\AX$B \fCenter B$
\UI$\{\alpha\} B \fCenter \lc\alpha\rc B$
\UI$B \fCenter \RESalphaProxy \lc\alpha\rc B$
\BI$A \pra B \fCenter \RESalphaProxy \{\alpha\} A > \RESalphaProxy \lc\alpha\rc B$
\UI$A \pra B \fCenter \RESalphaProxy (\{\alpha\} A > \lc\alpha\rc B)$
%\UI$\{\alpha\} A \pra B \fCenter \{\alpha\} A > \lc\alpha\rc B$
%\UI$\{\alpha\} (A \pra B) \fCenter \{\alpha\}A > \lc\alpha\rc B$
%\UI$A \pra B \fCenter \RESalphaProxy (\{\alpha\} A > \lc\alpha\rc B)$
\UI$[\alpha] (A \pra B) \fCenter \{\alpha\} \RESalphaProxy (\{\alpha\} A > \lc\alpha\rc B)$
\RightLabel{$comp$}
\UI$[\alpha] (A \pra B) \fCenter \Phi_{\alpha} >( \{\alpha\} A > \lc\alpha\rc B)$
\UI$\Phi_{\alpha}\,; [\alpha] (A \pra B) \fCenter \{\alpha\} A > \lc\alpha\rc B$
\UI$\{\alpha\} A\,; (\Phi_{\alpha}\,; [\alpha] (A \pra B) ) \fCenter \lc\alpha\rc B$
\UI$(\{\alpha\} A\,; \Phi_{\alpha})\,; [\alpha] (A \pra B) \fCenter \lc\alpha\rc B$
\UI$[\alpha] (A \pra B)\,; (\{\alpha\} A\,; \Phi_{\alpha}) \fCenter \lc\alpha\rc B$
\UI$\{\alpha\} A\,; \Phi_\alpha\, \fCenter \ls\alpha\rs (A \pra B) > \lc\alpha\rc B$
\UI$\Phi_\alpha\,; \{\alpha\} A \fCenter \ls\alpha\rs (A \pra B) > \lc\alpha\rc B$
\RightLabel{\scriptsize{\emph{reduce$'$}}}
\UI$\{\alpha\} A \fCenter \ls\alpha\rs (A \pra B) > \lc\alpha\rc B$
\UI$\lc\alpha\rc A \fCenter \ls\alpha\rs (A \pra B) > \lc\alpha\rc B$
\UI$\ls\alpha\rs (A \pra B)\,; \lc\alpha\rc A \fCenter \lc\alpha\rc B$
\UI$\lc\alpha\rc A\,; \ls\alpha\rs (A \pra B) \fCenter \lc\alpha\rc B$
\UI$\ls\alpha\rs (A \pra B) \fCenter \lc\alpha\rc A > \lc\alpha\rc B$
\UI$\ls\alpha\rs (A \pra B) \fCenter \lc\alpha\rc A \pra \lc\alpha\rc B$
\DisplayProof

 &
\!\!\!\!\!\!\!\!\!\!
\AX$A \fCenter A$
\UI$\{\alpha\} A \fCenter \lc\alpha\rc A$
\AX$B \fCenter B$
\UI$\{\alpha\} B \fCenter \{\alpha\} B$
\UI$\lc\alpha\rc B \fCenter \{\alpha\} B$
\BI$\lc\alpha\rc A\rightarrow \lc\alpha\rc B \fCenter \{\alpha\} A > \{\alpha\} B$
\UI$\lc\alpha\rc A\rightarrow \lc\alpha\rc B \fCenter \{\alpha\} (A > B)$
\UI$\RESalphaProxy (\lc\alpha\rc A\rightarrow \lc\alpha\rc B) \fCenter A > B$
\UI$\RESalphaProxy (\lc\alpha\rc A\rightarrow \lc\alpha\rc B) \fCenter A \rightarrow B$
\UI$\lc\alpha\rc A\rightarrow \lc\alpha\rc B \fCenter \{\alpha\} (A \rightarrow B)$
\UI$\lc\alpha\rc A\rightarrow \lc\alpha\rc B \fCenter \ls\alpha\rs (A \rightarrow B)$
\DisplayProof
 \\
\end{tabular}
}}
\end{center}

\newpage
For ease of notation, in the following derivations we assume the actions $\beta$, such that $\alpha\aga\beta$ form the set $\{\beta_i\,|\,1\leq i \leq n\}$.

\noindent
$\rule{126.2mm}{0.5pt}$
$\lc\alpha\rc \lc \aga \rc A \vdash 1_\alpha \pand \bigvee \{\lc\aga\rc \lc\beta\rc A\,|\,\alpha\aga\beta\}$

\begin{center}
{\scriptsize{
\AX$\Phi_\alpha \fCenter 1_\alpha$

\AX$A \fCenter A$
\UI$\{\beta_1\} A \fCenter \lc\beta_1\rc A$

\UI$ \{\aga\} \{\beta_1\} A \fCenter \lc \aga\rc \lc\beta_1\rc A$

%%%%%%%%%%%%%%%%%%%%%%%%%%%%%%%%%%%%%%%%%%%%%%%%%%%%%%%%%%%%%%%%%%%%%%%%%%%%%%%%%

\def\fCenter{\mbox{$\cdots$}}
\AX$\phantom{\{_k} \fCenter \phantom{\{_k}$
\noLine
\UI$\phantom{\{_k} \fCenter \phantom{\{_k}$
\noLine
\UI$\phantom{\{_k} \fCenter \phantom{\{_k}$

%%%%%%%%%%%%%%%%%%%%%%%%%%%%%%%%%%%%%%%%%%%%%%%%%%%%%%%%%%%%%%%%%%%%%%%%%%%%%%%%%%

\def\fCenter{\mbox{$\ \vdash\ $}}
\AX$A \fCenter A$
\UI$\{\beta_n\} A \fCenter \lc\beta_n\rc A$

\UI$ \{\aga\} \{\beta_n\} A \fCenter \lc \aga\rc \lc\beta_n\rc A$

%%%%%%%%%%%%%%%%%%%%%%%%%%%%%%%%%%%%%%%%%%%%%%%%%%%%%%%%%%%%%%%%%%%%%%%%%%%%%%%%%

\LeftLabel{\scriptsize{\emph{swap-out}$'$}}
\TI$\, {\{\alpha\}\{\aga\} A} \fCenter \Bigsemic\Big(\lc\aga\rc \lc\beta_i\rc A\Big)$
\dashedLine
\UI$ {\{\alpha\}\{\aga\} A} \fCenter \bigvee\Big(\lc\aga\rc \lc\beta_i\rc A\Big)$
\BI$\Phi_\alpha\,; {\{\alpha\}\{\aga\} A} \fCenter 1_\alpha \pand \bigvee\Big(\lc\aga\rc \lc\beta_i\rc A\Big)$
\LeftLabel{\scriptsize{\emph{reduce$'$}}}
\UI${\{\alpha\} \{\aga\} A} \fCenter 1_\alpha\pand \bigvee\Big(\lc\aga\rc \lc\beta_i\rc A\Big)$
\UI$\{ \aga \} A \fCenter \RESalphaProxy 1_\alpha\pand \bigvee\Big(\lc\aga\rc \lc\beta_i\rc A\Big)$
\UI$\lc \aga \rc A \fCenter \RESalphaProxy 1_\alpha\pand \bigvee\Big(\lc\aga\rc \lc\beta_i\rc A\Big)$
\UI$\{\alpha\} \lc \aga \rc A \fCenter 1_\alpha\pand \bigvee\Big(\lc \aga \rc \lc\beta_i\rc A\Big)$
\UI$\lc\alpha\rc \lc \aga \rc A \fCenter 1_\alpha\pand \bigvee\Big(\lc \aga \rc \lc\beta_i\rc A\Big)$
\DisplayProof
}}
\end{center}

\noindent
$\rule{126.2mm}{0.5pt}$
$1_\alpha \pand \bigvee \{\lc\aga\rc \lc\beta\rc A\,|\,\alpha\aga\beta\} \vdash \lc\alpha\rc \lc \aga \rc A$

\begin{center}
{\scriptsize
\AXC{$A \fCenter A$}
\UIC{$\{ \aga \} A \fCenter \lc\aga\rc A$}
\UIC{${\{\alpha\}\{ \aga \} A} \fCenter \lc\alpha\rc \lc \aga\rc A$}
\LeftLabel{\scriptsize{\emph{swap-in}$'$}}
\UIC{$\Phi_\alpha\,; \{\aga\}\{\beta_1\} A \fCenter \lc\alpha\rc \lc \aga\rc A$}
\UIC{$\{\aga\}\{\beta_1\} A\,; \Phi_\alpha\, \fCenter \lc\alpha\rc \lc \aga\rc A$}
\UIC{$\Phi_\alpha\,\fCenter \{\aga\}\{\beta_1\} A > \lc\alpha\rc \lc \aga\rc A$}
\UIC{$1_\alpha\,\fCenter \{\aga\}\{\beta_1\} A > \lc\alpha\rc \lc \aga\rc A$}
\UIC{$\{ \aga \}\{\beta_1\} A\,;1_\alpha \fCenter \lc\alpha\rc \lc\aga\rc A$}
\UIC{$ 1_\alpha\,;\{\aga\}\{\beta_1\} A \fCenter \lc\alpha\rc \lc\aga\rc A$}
\UIC{$\{\aga\}\{\beta_1\} A \fCenter 1_\alpha > \lc\alpha\rc \lc \aga\rc A$}
\UIC{$\{\beta_1\} A \fCenter \RESagaProxy (1_\alpha > \lc\alpha\rc \lc \aga \rc A)$}
\UIC{$\lc\beta_1\rc A \fCenter \RESagaProxy (1_\alpha > \lc\alpha\rc \lc \aga \rc A)$}
\UIC{$\{\aga\}\lc\beta_1\rc A \fCenter 1_\alpha > \lc\alpha\rc \lc \aga \rc A$}
\UIC{$\lc\aga\rc\lc\beta_1\rc A \fCenter 1_\alpha > \lc\alpha\rc \lc \aga \rc A$}

%%%%%%%%%%%%%%%%%%%%%%%%%%%%%%%%%%%%%%%%%%%%%%%%%%%%%%%%%%%%%%%%%

\def\fCenter{\mbox{$\cdots$}}
\AX$\phantom{\{_k} \fCenter \phantom{\{_k}$
\noLine
\UI$\phantom{\{_k} \fCenter \phantom{\{_k}$
\noLine
\UI$\phantom{\{_k} \fCenter \phantom{\{_k}$
\noLine
%\LeftLabel{\scriptsize{\emph{swap-in}}}
\UI$\phantom{\{_k} \fCenter \phantom{\{_k}$
\noLine
\UI$\phantom{\{_k} \fCenter \phantom{\{_k}$
\noLine
\UI$\phantom{\{_k} \fCenter \phantom{\{_k}$
\noLine
\UI$\phantom{\{_k} \fCenter \phantom{\{_k}$
\noLine
\UI$\phantom{\{_k} \fCenter \phantom{\{_k}$
\noLine
\UI$\phantom{\{_k} \fCenter \phantom{\{_k}$
\noLine
\UI$\phantom{\{_k} \fCenter \phantom{\{_k}$
\noLine
\UI$\phantom{\{_k} \fCenter \phantom{\{_k}$
\noLine
\UI$\phantom{\{_k} \fCenter \phantom{\{_k}$
\noLine
\UI$\phantom{\{_k} \fCenter \phantom{\{_k}$
\noLine
\UI$\phantom{\{_k} \fCenter \phantom{\{_k}$
%%%%%%%%%%%%%%%%%%%%%%%%%%%%%%%%%%%%%%%%%%%%%%%%%%%%%%%%%%%%%%%%%

\def\fCenter{\mbox{$\ \vdash\ $}}
\AXC{$A \fCenter A$}
\UIC{$\{ \aga \} A \fCenter \lc \aga \rc A$}
\UIC{${\{\alpha\}\{\aga\} A} \fCenter \lc\alpha\rc \lc\aga\rc A$}
\LeftLabel{\scriptsize{\emph{swap-in}$'$}}
\UIC{$\Phi_\alpha\,; \{\aga\}\{\beta_n\} A \fCenter \lc\alpha\rc \lc\aga\rc A$}
\UIC{$\{\aga\}\{\beta_n\} A\,; \Phi_\alpha \fCenter \lc\alpha\rc \lc\aga\rc A$}
\UIC{$\Phi_\alpha \fCenter \{\aga\}\{\beta_n\} A > \lc\alpha\rc \lc\aga\rc A$}
\UIC{$1_\alpha \fCenter \{\aga\}\{\beta_n\} A > \lc\alpha\rc \lc\aga\rc A$}
\UIC{$\{ \aga \}\{\beta_n\} A\,; 1_\alpha \fCenter \lc\alpha\rc \lc\aga\rc A$}
\UIC{$1_\alpha\,; \{\aga\}\{\beta_n\} A \fCenter \lc\alpha\rc \lc \aga \rc A$}
\UIC{$\{\aga\}\{\beta_n\} A \fCenter 1_\alpha > \lc\alpha\rc \lc \aga \rc A$}
\UIC{$\{\beta_n\} A \fCenter \RESagaProxy (1_\alpha > \lc\alpha\rc \lc\aga\rc A)$}
\UIC{$\lc\beta_n\rc A \fCenter \RESagaProxy (1_\alpha > \lc\alpha\rc \lc \aga \rc A)$}
\UIC{$\{\aga\} \lc\beta_n\rc A \fCenter 1_\alpha > \lc\alpha\rc \lc \aga \rc A$}
\UIC{$\lc\aga\rc \lc\beta_n\rc A \fCenter 1_\alpha > \lc\alpha\rc \lc \aga \rc A$}

%%%%%%%%%%%%%%%%%%%%%%%%%%%%%%%%%%%%%%%%%%%%%%%%%%%%%%%%%%%%%%%%%%%%%%%

\dashedLine
\TI$\bigvee\Big(\lc\aga\rc \lc\beta_i\rc A\Big) \fCenter \Bigsemic\Big(1_\alpha > \lc\alpha\rc \lc \aga \rc A\Big)$
\dashedLine
\UI$\bigvee\Big(\lc\aga\rc \lc\beta_i\rc A\Big) \fCenter 1_\alpha > \lc\alpha\rc \lc \aga \rc A)$
\UI$1_\alpha\,; \bigvee\Big(\lc\aga\rc \lc\beta_i\rc A\Big) \fCenter \lc\alpha\rc \lc \aga \rc A$
\UI$1_\alpha\pand \bigvee\Big(\lc\aga\rc \lc\beta_i\rc A\Big) \fCenter \lc\alpha\rc \lc \aga \rc A$
\DisplayProof
}
\end{center}

%%%%%%%%%%%%%%%%%%%%%%%%%%%%%%%%%%%%%%%%%%%%%%%%%%%%%%%%%%%%%%%%%%%%%%%%%%%%%%%%%%%%%%%%%

\newpage

\noindent
$\rule{126.2mm}{0.5pt}$
$\ls\alpha\rs\lc \aga \rc A \vdash Pre(\alpha)\rightarrow \bigvee \{\lc\aga\rc \lc\beta\rc A \,|\, \alpha\aga\beta\}$ \\

\begin{center}
{\scriptsize{
\begin{tabular}{c}

\AX$A \fCenter A$
\UI$\{\beta_1\} A \fCenter \lc\beta_1\rc A$

\UI$ \{\aga\} \{\beta_1\} A \fCenter \lc\aga\rc \lc\beta_1\rc A$

%%%%%%%%%%%%%%%%%%%%%%%%%%%%%%%%%%%%%%%%%%%%%%%%%%%%%%%%%%%%%%%%%

\def\fCenter{\mbox{$\cdots$}}
\AX$\phantom{\{_k} \fCenter \phantom{\{_k}$

\noLine
\UI$\phantom{\{_k} \fCenter \phantom{\{_k}$
\noLine
\UI$\phantom{\{_k} \fCenter \phantom{\{_k}$

%%%%%%%%%%%%%%%%%%%%%%%%%%%%%%%%%%%%%%%%%%%%%%%%%%%%%%%%%%%%%%%%%

\def\fCenter{\mbox{$\ \vdash\ $}}
\AX$A \fCenter A$
\UI$\{\beta_n\} A \fCenter \lc\beta_n\rc A$

\UI$\{\aga\} \{\beta_n\} A \fCenter \lc\aga\rc \lc\beta_n\rc A$

%%%%%%%%%%%%%%%%%%%%%%%%%%%%%%%%%%%%%%%%%%%%%%%%%%%%%%%%%%%%%%%%%%%%%%%%

\LeftLabel{\scriptsize{\emph{swap-out}$'$}}
\TI$ {\{\alpha\}\{ \aga \} A} \fCenter \Bigsemic \Big(\lc\aga\rc \lc\beta_i\rc A\Big)$
\dashedLine
\UI$ {\{\alpha\}\{ \aga \} A} \fCenter \bigvee \Big(\lc\aga\rc \lc\beta_i\rc A\Big)$
\UI$\{\aga\} A \fCenter \RESalphaProxy \bigvee \Big(\lc\aga\rc \lc\beta_i\rc A\Big)$
\UI$\lc\aga\rc A \fCenter \RESalphaProxy \bigvee \Big(\lc \aga \rc \lc\beta_i\rc A\Big)$
\UI$[\alpha] \lc\aga\rc A \fCenter \{\alpha\} \RESalphaProxy \bigvee \Big(\lc\aga\rc \lc\beta_i\rc A\Big)$
\RightLabel{\scriptsize{\emph{comp}}}
\UI$[\alpha] \lc\aga\rc A \fCenter \Phi_\alpha> \bigvee \Big (\lc\aga\rc \lc\beta_i\rc A\Big)$
\UI$\Phi_\alpha\,; [\alpha] \lc \aga \rc A \fCenter \bigvee \Big (\lc\aga\rc \lc\beta_i\rc A\Big)$
\UI$\Phi_\alpha \fCenter \bigvee \Big (\lc\aga\rc \lc\beta_i\rc A\Big) < [\alpha] \lc \aga \rc A$
\UI$1_\alpha \fCenter \bigvee \Big (\lc\aga\rc \lc\beta_i\rc A\Big) < [\alpha] \lc \aga \rc A$
\UI$1_\alpha\,; [\alpha] \lc \aga \rc A\fCenter \bigvee \Big (\lc\aga\rc \lc\beta_i\rc A\Big)$
\UI$\ls\alpha\rs \lc\aga\rc A \fCenter 1_\alpha > \bigvee \Big (\lc\aga\rc \lc\beta_i\rc A\Big)$
\UI$\ls\alpha\rs \lc\aga\rc A \fCenter 1_\alpha \rightarrow \bigvee \Big (\lc\aga\rc \lc\beta_i\rc A\Big)$
\DisplayProof
\end{tabular}
}}
\end{center}

\noindent
$\rule{126.2mm}{0.5pt}$
$Pre(\alpha)\rightarrow \bigvee \{\lc\aga\rc \lc\beta_i\rc A\,|\,\alpha\aga\beta\} \vdash \ls\alpha\rs \lc\aga\rc A$

\begin{center}
{\scriptsize{
\begin{tabular}{c}
\AX$\Phi_\alpha \fCenter 1_\alpha$

\AXC{$A \fCenter A$}
\UIC{$\{ \aga \} A \fCenter \lc\aga\rc A$}
\UIC{${\{\alpha\}\{ \aga \} A} \fCenter \{\alpha\} \lc\aga\rc A$}
\LeftLabel{\scriptsize{\emph{swap-in$'$}}}
\UIC{$\Phi_\alpha\,; \{\aga\} \{\beta_1\} A \fCenter \{\alpha\} \lc\aga\rc A$}
\UIC{$ \{ \aga \} \{\beta_1\} A \fCenter \Phi_\alpha\, > \{\alpha\} \lc\aga\rc A$}
\LeftLabel{\scriptsize{\emph{reduce$'$}}}
\UIC{$\{\aga\} \{\beta_1\} A \fCenter \{\alpha\} \lc\aga\rc A$}
\UIC{$\{\beta_1\} A \fCenter \RESagaProxy \{\alpha\} \lc\aga\rc A$}
\UIC{$\lc\beta_1\rc A \fCenter \RESagaProxy \{\alpha\} \lc\aga\rc A$}
\UIC{$\{\aga\} \lc\beta_1\rc A \fCenter \{\alpha\} \lc\aga\rc A$}
\UIC{$\lc\aga\rc \lc\beta_1\rc A \fCenter \{\alpha\} \lc\aga\rc A$}

%%%%%%%%%%%%%%%%%%%%%%%%%%%%%%%%%%%%%%%%%%%%%%%%%%%%%%%%%%%%%%%%%%%%%

\def\fCenter{\mbox{$\cdots$}}
\AXC{$\phantom{\{_k} \fCenter \phantom{\{_k}$}
\noLine
\UIC{$\phantom{\{_k} \fCenter \phantom{\{_k}$}
\noLine
\UIC{$\phantom{\{_k} \fCenter \phantom{\{_k}$}
\noLine
\UIC{$\phantom{\{_k} \fCenter \phantom{\{_k}$}
\noLine
\UIC{$\phantom{\{_k} \fCenter \phantom{\{_k}$}
\noLine
\UIC{$\phantom{\{_k} \fCenter \phantom{\{_k}$}
\noLine
\UIC{$\phantom{\{_k} \fCenter \phantom{\{_k}$}
\noLine
\UIC{$\phantom{\{_k} \fCenter \phantom{\{_k}$}
\noLine
\UIC{$\phantom{\{_k} \fCenter \phantom{\{_k}$}
\noLine
\UIC{$\phantom{\{_k} \fCenter \phantom{\{_k}$}

%%%%%%%%%%%%%%%%%%%%%%%%%%%%%%%%%%%%%%%%%%%%%%%%%%%%%%%%%%%%%%%%%%%%

\def\fCenter{\mbox{$\ \vdash\ $}}
\AXC{$A \fCenter A$}
\UIC{$\{ \aga \} A \fCenter \lc \aga \rc A$}

\UIC{${\{\alpha\}\{\aga\} A} \fCenter \{\alpha\} \lc\aga\rc A$}
\LeftLabel{\scriptsize{\emph{swap-in$'$}}}
\UIC{$\Phi_\alpha \,; \{\aga\} \{\beta_n\} A \fCenter \{\alpha\} \lc\aga\rc A$}
\UIC{$\{\aga\} \{\beta_n\} A \fCenter \Phi_\alpha \, > \{\alpha\} \lc\aga\rc A$}
\LeftLabel{\scriptsize{\emph{reduce$'$}}}
\UIC{$\{\aga\} \{\beta_n\} A \fCenter \{\alpha\} \lc\aga\rc A$}
\UIC{$\{\beta_n\} A \fCenter \RESagaProxy \{\alpha\} \lc\aga\rc A$}
\UIC{$\lc\beta_n\rc A \fCenter \RESagaProxy \{\alpha\} \lc\aga\rc A$}
\UIC{$\{\aga\} \lc\beta_n\rc A \fCenter \{\alpha\} \lc\aga\rc A$}
\UIC{$\lc\aga\rc \lc\beta_n\rc A \fCenter \{\alpha\} \lc\aga\rc A$}

%%%%%%%%%%%%%%%%%%%%%%%%%%%%%%%%%%%%%%%%%%%%%%%%%%%%%%%%%%%%%%%%%%%%%%%

\dashedLine
\TI$\bigvee\Big(\lc\aga\rc \lc\beta_i\rc A\Big) \fCenter \Bigsemic\Big(\{\alpha\} \lc \aga \rc A\Big)$
\dashedLine
\UI$\bigvee\Big(\lc \aga \rc \lc\beta_i\rc A\Big) \fCenter  \{\alpha\} \lc \aga \rc A$

\BI$1_\alpha \rightarrow \bigvee \Big(\lc\aga\rc \lc\beta_i\rc A\Big) \fCenter \Phi_\alpha > \{\alpha\} \lc \aga \rc A$

\RightLabel{\scriptsize{\emph{reduce$'$}}}
\UI$1_\alpha \rightarrow \bigvee \Big (\lc\aga\rc \lc\beta_i\rc A\Big) \fCenter \{\alpha\} \lc \aga \rc A$

\UI$1_\alpha \rightarrow \bigvee \Big(\lc\aga\rc \lc\beta_i\rc A\Big) \fCenter \ls\alpha\rs \lc \aga \rc A$
\DisplayProof
 \\

\end{tabular}
}}
\end{center}

%%%%%%%%%%%%%%%%%%%%%%%%%%%%%%%%%%%%%%%%%%%%%%%%%%%%%%%%%%%%%%%%%%%%%%%%%%%%%%%%%%%%%%%%%

\newpage

\noindent
$\rule{126.2mm}{0.5pt}$
$[\alpha][\aga] A \vdash Pre(\alpha)\rightarrow \bigwedge \{[\aga][\beta] A\,|\,\alpha\aga\beta\}$

\begin{center}
{\scriptsize{
\AXC{$A \fCenter A$}
\UIC{$[\aga] A \fCenter \{\aga\} A$}
\UIC{$[\alpha][\aga] A \fCenter {\{\alpha\} \{\aga\} A}$}
\RightLabel{\scriptsize{$swap\textrm{-}in'$}}
\UIC{$[\alpha][\aga] A \fCenter \Phi_\alpha > \{ \aga \}\{\beta_1\} A$}
\UIC{$\Phi_\alpha\,; [\alpha][\aga] A \fCenter \{ \aga \}\{\beta_1\} A$}
\UIC{$\RESagaProxy \Phi_\alpha\,; [\alpha][\aga] A) \fCenter \{\beta_1\} A$}
\UIC{$\RESagaProxy \Phi_\alpha\,; [\alpha][\aga] A) \fCenter [\beta_1] A$}
\UIC{$\Phi_\alpha\,; [\alpha][\aga] A \fCenter \{\aga\}[\beta_1] A$}
\UIC{$\Phi_\alpha\,; [\alpha][\aga] A \fCenter [\aga][\beta_1] A$}

%%%%%%%%%%%%%%%%%%%%%%%%%%%%%%%%%%%%%%%%%%%%%%%%%%%%%%%%%%%%%%%%%

\def\fCenter{\mbox{$\cdots$}}
\AX$\phantom{\{_k} \fCenter \phantom{\{_k}$
\noLine
\UI$\phantom{\{_k} \fCenter \phantom{\{_k}$
\noLine
\UI$\phantom{\{_k} \fCenter \phantom{\{_k}$
\noLine
\UI$\phantom{\{_k} \fCenter \phantom{\{_k}$
\noLine
%\RightLabel{\fns{\emph{swap-in}}}
\UI$\phantom{\{_k} \fCenter \phantom{\{_k}$
\noLine
\UI$\phantom{\{_k} \fCenter \phantom{\{_k}$
\noLine
\UI$\phantom{\{_k} \fCenter \phantom{\{_k}$
\noLine
\UI$\phantom{\{_k} \fCenter \phantom{\{_k}$
\noLine
\UI$\phantom{\{_k} \fCenter \phantom{\{_k}$

%%%%%%%%%%%%%%%%%%%%%%%%%%%%%%%%%%%%%%%%%%%%%%%%%%%%%%%

\def\fCenter{\mbox{$\ \vdash\ $}}
\AX$A \fCenter A$
\UI$[\aga] A \fCenter \{ \aga \} A$
\UI$[\alpha][\aga] A \fCenter {\{\alpha\}\{ \aga \} A}$

\RightLabel{\scriptsize{$swap\textrm{-}in'$}}
\UI$[\alpha][\aga] A \fCenter \Phi_\alpha > \{\aga\}\{\beta_n\} A$
\UI$\Phi_\alpha\,; [\alpha][\aga] A \fCenter \{\aga\}\{\beta_n\} A$
\UI$\RESagaProxy (\Phi_\alpha\,; [\alpha][\aga] A) \fCenter \{\beta_n\} A$
\UI$\RESagaProxy (\Phi_\alpha\,; [\alpha][\aga] A) \fCenter [\beta_n] A$
\UI$\Phi_\alpha\,; [\alpha][\aga] A \fCenter \{\aga\}[\beta_n] A$
\UI$\Phi_\alpha\,; [\alpha][\aga] A \fCenter [\aga][\beta_n] A$

%%%%%%%%%%%%%%%%%%%%%%%%%%%%%%%%%%%%%%%%%%%%%%%%%%%%%%%%%%%%%%%

\dashedLine
\TI$\Bigsemic\Big(\Phi_\alpha\,; [\alpha][\aga] A\Big) \fCenter \bigwedge\Big([\aga][\beta_i] A\Big)$
\dashedLine
%\LeftLabel{\scriptsize{$C$}}
\UI$\Phi_\alpha\,; [\alpha][\aga] A \fCenter \bigwedge\Big([\aga][\beta_i] A\Big)$
\UI$ [\alpha][\aga] A\,; \Phi_\alpha \fCenter \bigwedge\Big([\aga][\beta_i] A\Big)$
\UI$\Phi_\alpha \fCenter [\alpha][\aga] A>\bigwedge\Big([\aga][\beta_i] A\Big)$
\UI$1_\alpha \fCenter [\alpha][\aga] A>\bigwedge\Big([\aga][\beta_i] A\Big)$
\UI$ [\alpha][\aga] A\,; 1_\alpha \fCenter \bigwedge\Big([\aga][\beta_i] A\Big)$
\UI$ 1_\alpha\,; [\alpha][\aga] A \fCenter \bigwedge\Big([\aga][\beta_i] A\Big)$
\UI$[\alpha][\aga] A \fCenter 1_\alpha\ > \bigwedge\Big([\aga][\beta_i] A\Big)$
\UI$[\alpha][\aga] A \fCenter 1_\alpha \rightarrow \bigwedge\Big([\aga][\beta_i] A\Big)$
\DisplayProof
}}
\end{center}

\noindent
$\rule{126.2mm}{0.5pt}$
$Pre(\alpha)\rightarrow \bigwedge \{[\aga][\beta] A\,|\, \alpha\aga\beta\} \vdash [\alpha][\aga] A$

\begin{center}
{\scriptsize{
\begin{tabular}{c}
\AX$\Phi_\alpha \fCenter 1_\alpha$

\AX$A \fCenter A$
\UI$[\beta_1] A \fCenter \{\beta_1\} A$
\UI$[\aga][\beta_1] A \fCenter \{\aga\}\{\beta_1\} A$

%%%%%%%%%%%%%%%%%%%%%%%%%%%%%%%%%%%%%%%%%%%%%%%%%%%%%%%%%%%%%%%%%%%%

\def\fCenter{\mbox{$\cdots$}}
\AXC{$\phantom{\{_k} \fCenter \phantom{\{_k}$}
\noLine
\UIC{$\phantom{\{_k} \fCenter \phantom{\{_k}$}
\noLine
\UIC{$\phantom{\{_k} \fCenter \phantom{\{_k}$}

%%%%%%%%%%%%%%%%%%%%%%%%%%%%%%%%%%%%%%%%%%%%%%%%%%%%%%%%%%%%%%%%%%%

\def\fCenter{\mbox{$\ \vdash\ $}}
\AXC{$A \fCenter A$}
\UI$[\beta_n] A \fCenter \{\beta_n\} A$
\UI$[\aga][\beta_n] A \fCenter \{\aga\} \{\beta_n\} A$

%%%%%%%%%%%%%%%%%%%%%%%%%%%%%%%%%%%%%%%%%%%%%%%%%%%%%%%%%%%%%%%%%%%%%

\RightLabel{\scriptsize{$swap\textrm{-}out'$}}
\TI$\Bigsemic\Big([\aga][\beta_i] A\Big) \fCenter \{\alpha\} \{\aga\} A$
\dashedLine
\UI$\bigwedge\Big([\aga][\beta_i] A\Big) \fCenter \{\alpha\} \{\aga\} A$

\BI$1_\alpha\rightarrow \bigwedge \Big([\aga][\beta_i] A\Big) \fCenter \Phi_\alpha > {\{\alpha\} \{\aga\} A}$
\RightLabel{\scriptsize{$reduce'$}}
\UI$ 1_\alpha\rightarrow \bigwedge\Big([\aga][\beta_i] A\Big) \fCenter {\{\alpha\} \{\aga\} A}$
\UI$\RESalphaProxy ( 1_\alpha \rightarrow \bigwedge\Big([\aga][\beta_i] A\Big)) \fCenter \{\aga\} A$
\UI$\RESalphaProxy (1_\alpha \rightarrow \bigwedge\Big([\aga][\beta_i] A\Big)) \fCenter [\aga] A$
\UI$ 1_\alpha \rightarrow \bigwedge\Big([\aga][\beta_i] A\Big) \fCenter \{\alpha\}[\aga] A$
\UI$1_\alpha\rightarrow \bigwedge\Big([\aga][\beta_i] A\Big) \fCenter [\alpha][\aga] A$
\DisplayProof
 \\
\end{tabular}
}}
\end{center}

%%%%%%%%%%%%%%%%%%%%%%%%%%%%%%%%%%%%%%%%%%%%%%%%%%%%%%%%%%%%%%%%%%%%%%%%%%%%%%%%%%%%%%%%%

\newpage

\noindent
$\rule{126.2mm}{0.5pt}$
$\lc\alpha\rc [\aga] A \vdash Pre(\alpha) \pand \bigwedge \{[\aga] \ls\beta\rs A\,|\, \alpha\aga\beta\}$

\begin{center}
{\scriptsize{
\begin{tabular}{@{}c@{}}
\!\!\!\!\!\!\!
\AX$\Phi_\alpha \fCenter 1_\alpha$

%%%%%%%%%%%%%%%%%%%%%%%%%%%%%%%%%%%%%%%%%%%%%%%%%%%%%%%%%%%

\AX$A \fCenter A$
\UIC{$[\aga] A \fCenter \{ \aga \} A$}
\LeftLabel{\scriptsize{$balance$}}
\UIC{$\{\alpha\} [\aga] A \fCenter {\{\alpha\}\{\aga\} A}$}
\LeftLabel{\scriptsize{$swap\textrm{-}in'$}}
\UIC{$\{\alpha\} [\aga] A \fCenter \Phi_\alpha> \{\aga\} \{\beta_1\} A$}
\UIC{$\Phi_\alpha\,; \{\alpha\} [\aga] A \fCenter \{\aga\} \{\beta_1\} A$}
\LeftLabel{\scriptsize{$reduce'$}}
\UIC{$\{\alpha\} [\aga] A \fCenter \{\aga\} \{\beta_1\} A$}
\UIC{$\{\alpha\} [\aga] A \fCenter \{\aga\} \{\beta_1\} A$}
\UIC{$\RESagaProxy \{\alpha\} [\aga] A \fCenter \{\beta_1\} A$}
\UIC{$\RESagaProxy \{\alpha\} [\aga] A \fCenter \ls\beta_1\rs A$}
\UIC{$\{\alpha\} [\aga] A \fCenter \{\aga\} \ls\beta_1\rs A$}
\UIC{$\{\alpha\} [\aga] A \fCenter [\aga] \ls\beta_1\rs A$}

%%%%%%%%%%%%%%%%%%%%%%%%%%%%%%%%%

\def\fCenter{\mbox{$\cdots$}}
\AX$\phantom{\{_k} \fCenter \phantom{\{_k}$
\noLine
\UI$\phantom{\{_k} \fCenter \phantom{\{_k}$
\noLine
%\RightLabel{\fns\emph{balance}}
\UI$\phantom{\{_k} \fCenter \phantom{\{_k}$
\noLine
\UI$\phantom{\{_k} \fCenter \phantom{\{_k}$
\noLine
\UI$\phantom{\{_k} \fCenter \phantom{\{_k}$
\noLine
\UI$\phantom{\{_k} \fCenter \phantom{\{_k}$
\noLine
\UI$\phantom{\{_k} \fCenter \phantom{\{_k}$
\noLine
\UI$\phantom{\{_k} \fCenter \phantom{\{_k}$
\noLine
%\RightLabel{\fns\emph{swap-in}}
\UI$\phantom{\{_k} \fCenter \phantom{\{_k}$
\noLine
\UI$\phantom{\{_k} \fCenter \phantom{\{_k}$
\noLine
\UI$\phantom{\{_k} \fCenter \phantom{\{_k}$
\noLine

%%%%%%%%%%%%%%%%%%%%%%%%%%%%%%%%

\def\fCenter{\mbox{$\ \vdash\ $}}
\AX$A \fCenter A$
\UIC{$[\aga] A \fCenter \{ \aga \} A$}
\RightLabel{\scriptsize{$balance$}}
\UIC{$\{\alpha\} [\aga] A \fCenter {\{\alpha\}\{\aga\} A}$}
\RightLabel{\scriptsize{$swap\textrm{-}in'$}}
\UIC{$\{\alpha\} [\aga] A \fCenter \Phi_\alpha> \{\aga\} \{\beta_{n}\} A$}
\UIC{$\Phi_\alpha\,; \{\alpha\} [\aga] A \fCenter \{\aga\} \{\beta_{n}\} A$}
\RightLabel{\scriptsize{$reduce'$}}
\UIC{$\{\alpha\} [\aga] A \fCenter \{\aga\} \{\beta_{n}\} A$}
\UIC{$\{\alpha\} [\aga] A \fCenter \{\aga\} \{\beta_{n}\} A$}
\UIC{$\RESagaProxy \{\alpha\} [\aga] A \fCenter \{\beta_n\} A$}
\UIC{$\RESagaProxy \{\alpha\} [\aga] A \fCenter \ls\beta_n\rs A$}
\UIC{$\{\alpha\} [\aga] A \fCenter \{\aga\} \ls\beta_n\rs A$}
\UIC{$\{\alpha\} [\aga] A \fCenter [\aga] \ls\beta_n\rs A$}

%%%%%%%%%%%%%%%%%%%%%%%%%%%%%%%%%%

\dashedLine
\TI$\Bigsemic\Big(\{\alpha\} [\aga] A\Big) \fCenter \bigwedge\Big([\aga] \ls\beta_i\rs A\Big)$
\dashedLine
%\LeftLabel{\scriptsize{$C$}}
\UI$\{\alpha\} [\aga] A \fCenter \bigwedge\Big([\aga] \ls\beta_i\rs A\Big)$
\BI$\Phi_\alpha\,; \{\alpha\} [\aga] A \fCenter 1_\alpha \pand \bigwedge\Big([\aga] \ls\beta_i\rs A\Big)$
\LeftLabel{\scriptsize{$reduce'$}}
\UI$\{\alpha\} [\aga] A \fCenter 1_\alpha \pand \bigwedge\Big([\aga] \ls\beta_i\rs A\Big)$
\UI$\lc\alpha\rc [\aga] A \fCenter 1_\alpha \pand \bigwedge\Big([\aga] \ls\beta_i\rs A\Big)$
\DisplayProof
\end{tabular}
}}
\end{center}

\noindent
$\rule{126.2mm}{0.5pt}$
$Pre(\alpha) \pand \bigwedge \{[\aga] \ls\beta\rs A\,|\, \alpha\aga\beta\} \vdash \lc\alpha\rc [\aga] A$

%%%%%%%%%%%%%%%%%%%%%%%%%%%%%%%%%%%%%%
\begin{center}
{\scriptsize{
\begin{tabular}{@{}c@{}}
\AX$A \fCenter A$
\UI$\ls\beta_1\rs A \fCenter \{\beta_1\} A$
\UI$[\aga] \ls\beta_1\rs A \fCenter \{\aga\} \{\beta_1\} A$

\def\fCenter{{\mbox{$\cdots$}}}
\AX${\phantom{A}}\fCenter{\phantom{A}}$
\noLine
\UI${\phantom{\{_{1}}}\fCenter{\phantom{\{_{1}}}$
\noLine
\UI${\phantom{\{_{1}}}\fCenter{\phantom{\{_{1}}}$
\noLine

\def\fCenter{{\mbox{$\ \vdash\ $}}}
\AX$A \fCenter A$
\UI$\ls\beta_{n}\rs A \fCenter \{\beta_{n}\} A$
\UI$[\aga] \ls\beta_{n}\rs A \fCenter \{\aga\} \{\beta_{n}\} A$

\RightLabel{\scriptsize{$swap\textrm{-}out'$}}
\TI$\Bigsemic\Big([\aga] \ls\beta_i\rs A\Big) \fCenter {\{\alpha\}\{ \aga \} A}$
\dashedLine
\UI$\bigwedge\Big([\aga] \ls\beta_i\rs A\Big) \fCenter {\{\alpha\}\{\aga\} A}$
\UI$\RESalphaProxy \bigwedge \Big([\aga] \ls\beta_i\rs A\Big) \fCenter \{\aga\} A$
\UI$\RESalphaProxy \bigwedge \Big([\aga] \ls\beta_i\rs A\Big) \fCenter [\aga] A$
\UI$\{\alpha \}\RESalphaProxy \bigwedge \Big([\aga] \ls\beta_i\rs A \Big) \fCenter \lc \alpha \rc[\aga] A$
\UI$\Phi_\alpha\,;\bigwedge \Big([\aga] \ls\beta_i\rs A\Big) \fCenter \lc\alpha\rc [\aga] A$
\UI$\bigwedge \Big([\aga] \ls\beta_i\rs A\Big)\,; \Phi_\alpha\,\fCenter \lc\alpha\rc [\aga] A$
\UI$\Phi_\alpha\,\fCenter \bigwedge \Big([\aga] \ls\beta_i\rs A\Big) > \lc\alpha\rc [\aga] A$
\UI$1_\alpha\,\fCenter \bigwedge \Big([\aga] \ls\beta_i\rs A\Big) > \lc\alpha\rc [\aga] A$
\UI$\bigwedge \Big([\aga] \ls\beta_i\rs A\Big)\,; 1_\alpha\,\fCenter \lc\alpha\rc [\aga] A$
\UI$1_\alpha\,; \bigwedge \Big([\aga] \ls\beta_i\rs A\Big) \fCenter \lc\alpha\rc [\aga] A$
\UI$1_\alpha\pand \bigwedge \Big([\aga] \ls\beta_i\rs A\Big) \fCenter \lc\alpha\rc [\aga] A$
\DisplayProof
\end{tabular}
}}
\end{center}

%\bibliography{BIB}
\bibliography{BIB,final-coalgebra}
\bibliographystyle{plain}

\end{document}